\documentclass[11pt]{amsart}
\usepackage{amsmath,amsthm, amscd, amssymb, amsfonts, pifont, mathrsfs,mathtools,comment,tikz}
\usepackage[all]{xy}
\usepackage{color}
\usepackage[inline]{enumitem}
\usepackage{bm}
\usepackage{graphicx}
\usepackage{hyperref}

\usepackage{mathrsfs}
\usepackage{multicol}
\usepackage{hhline} 
\usepackage[active]{srcltx}
 \hypersetup{colorlinks=true,citecolor=cyan,linkcolor=blue,linktocpage=true}

\newcommand{\aut}{\mathfrak{t}}
\newcommand{\tipo}{\varOmega}

\newcommand{\gp}{\mathsf G}
\newcommand{\subgp}{\mathsf M}
\newcommand{\normsubgp}{\mathsf T}
\newcommand{\subgpH}{\mathsf H}
\newcommand{\subgpK}{\mathsf K}

\newcommand{\PSL}{\mathbf{PSL}}
\newcommand{\Gb}{\boldsymbol{G}}
\newcommand{\PSp}{\mathbf{PSp}}
\newcommand{\bZ}{\mathbf{Z}}
\newcommand{\Pom}{\mathbf{P\Omega}}
\newcommand{\PSU}{\mathbf{PSU}}

\newcommand{\SU}{\mathbf{SU}}

\newcommand{\nsbgp}{\lhd}

\newcommand{\letra}{\mathsf r}
\newcommand{\letrados}{\mathsf s}
\newcommand{\letratres}{\mathsf y}
\newcommand{\letracuatro}{\mathsf x}

\newcommand{\bq}{\mathfrak{q}}
\newcommand{\bp}{\mathfrak{p}}

\newcommand\NN{\mathbb N}

\newcommand\WO{W[\mathfrak O]}
\newcommand\lb{\lambda}

\newcommand{\Aff}{\operatorname{Aff}}
\newcommand{\diag}{\operatorname{diag}}
\newcommand{\Cent}{\operatorname{Cent}}

\newcommand{\Imm}{\operatorname{Im}}

\newcommand{\supp}{\operatorname{supp}}

\newcommand{\Inn}{\operatorname{Inn}}

\newcommand\toba{\mathscr B}

\newcommand{\trid}{\triangleright}
\newcommand{\tridd}{\triangleright}

\newcommand{\Ha}{{\mathbb H}}
\newcommand{\Me}{{\mathbb M}}
\newcommand{\Fc}{{\mathcal F}}

\newcommand{\kk}{\mathbb F}

\newcommand{\ku}{\Bbbk}

\newcommand{\N}{{\mathbb N}}
\newcommand{\I}{{\mathbb I}}

\newcommand{\G}{{\mathbb G}}
\newcommand{\B}{{\mathbb{B}}}
\newcommand{\T}{{\mathbb{T}}}
\newcommand{\U}{\mathbb{U}}

\newcommand{\Pa}{\mathbb{P}}

\newcommand{\M}{{\mathbb M}}

\newcommand{\F}{{\mathbb F}}
\newcommand{\C}{{\mathcal C}}

\newcommand{\GL}{\mathrm{GL}}

\newcommand{\PGL}{\mathbf{PGL}}
\newcommand{\SL}{\mathbf{SL}}
\newcommand{\Sp}{\mathbf{Sp}}

\newcommand{\Fr}{\operatorname{Fr}}

\newcommand{\Le}{{\mathbb L}}
\newcommand{\Vu}{{\mathbb V}}
\newcommand{\Oc}{{\mathcal O}}

\newcommand{\yd}[1]{{}^{#1}_{#1}\mathcal{YD}}

\newcommand{\End}{\operatorname{End}}
\newcommand{\Aut}{\operatorname{Aut}}

\numberwithin{equation}{section}

\theoremstyle{plain}

\newtheorem{lema}{Lemma}[section]
\newtheorem{theorem}[lema]{Theorem}

\newtheorem{cor}[lema]{Corollary}
\newtheorem{conjecture}[lema]{Conjecture}
\newtheorem{prop}[lema]{Proposition}

\newtheorem{question}[lema]{Question}

\newtheorem{question-app}{Question}

\theoremstyle{definition}
\newtheorem{definition}[lema]{Definition}
\newtheorem{exa}[lema]{Example}

\theoremstyle{remark}
\newtheorem{obs}[lema]{Remark}
\newtheorem{remark}[lema]{Remark}
\newtheorem{step-intro}{Step}

\newcommand\id{\operatorname{id}}

\newcommand\an{\mathbb A_n}

\newcommand\aco{\mathbb A_5}
\newcommand\as{\mathbb A_6}

\newcommand\s{\mathbb S}

\renewcommand{\subjclassname}

\def\pf{\begin{proof}}
\def\epf{\end{proof}}

\newcounter{tabla}\stepcounter{tabla}

\begin{document}

\renewcommand{\baselinestretch}{1.2}

\thispagestyle{empty}

\title[Hopf algebras over Chevalley groups]
{Hopf algebras over Chevalley groups}

\author[Andruskiewitsch]{Nicol\'as Andruskiewitsch}
\address[Andruskiewitsch]{Profesor Em\'erito de la Facultad de Matem\'atica, F\'isica,
Astronom\'\i a y Computaci\'on, Universidad Nacional de C\'ordoba. 
\newline
CIEM-CONICET. 
Medina Allende s/n (5000) Ciudad Universitaria, C\'ordoba, Argentina.
\newline
Department of Mathematics and Data
Science, Vrije Universiteit Brussel, Pleinlaan 2, 1050 Brussels, Belgium
\newline 
Shenzhen International Center for Mathematics, Southern University of Science and Technology, Shenzhen 518055, China}
\email{nicolas.andruskiewitsch@unc.edu.ar}

\author[Carnovale]{Giovanna Carnovale}
\address[Carnovale]{Dipartimento di Matematica Tullio Levi-Civita, 
Universit\`a degli Studi di Padova, 
via Trieste 63, 3512,1 Padova, Italia.} 
\email{carnoval@math.unipd.it, +39-049-8271354}


\begin{abstract}
We show that every finite-dimensional pointed Hopf algebra over a finite simple Chevalley group, different from $\PSL_2(q)$ with $q\equiv 3 \mod 4$ (and from $\PSL_3(2)\simeq\PSL_2(7)$), is isomorphic to the corresponding group algebra.
To do this, we complete the analysis of the Nichols algebras of 
 Yetter-Drinfeld modules over such groups whose support is a semisimple orbit, begun
in \cite{ACG-III,ACG-VII}. In addition to the techniques used in loc. cit., we introduce a general procedure to determine when a semisimple conjugacy class in a Chevalley or Steinberg group is of type C and a new criterion based on the results of \cite{AHV} that applies to arbitrary racks. Throughout the process, we obtain results on Nichols algebras over racks beyond the framework of Chevalley groups.
\end{abstract}

\maketitle

\newpage
\setcounter{tocdepth}{2}
\tableofcontents

\section{Introduction}
\subsection{Pointed Hopf algebras}

\

The method proposed in \cite{A-Schneider} for the classification 
of the finite-dimensional pointed Hopf algebras over an algebraically closed field
$\ku$ of characteristic $0$ contains as a fundamental step the following question:

\begin{question}\label{question:groups}
Determine the Yetter-Drinfeld modules $V$ over a finite group $G$ such that
the Nichols algebra $\toba(V)$ is finite-dimensional.
\end{question}
The question must be addressed by considering different types of groups.

\medbreak
\begin{itemize}[leftmargin=*]\renewcommand{\labelitemi}{$\circ$}
\item 
For $G$ abelian, a complete answer is given in \cite{Heckenberger-advances}. 
The list is rich: it contains the positive parts of small quantum groups and supergroups at roots of $1$
plus quantum analogues of 
restricted enveloping algebras
of some Lie algebras and superalgebras in positive characteristic.

 \medbreak
\item For $G$ solvable, the question was answered in \cite{AHV} 
 when $V$ is simple (there are essentially 10 possible Nichols algebras)
 and in \cite{HV2,HV1} when $V$ is decomposable (the Nichols algebras arise from the abelian case
 up to a small number of exceptions).
\end{itemize}
For simple non-abelian groups the situation is quite different and the extreme case is encoded in the following notion. 
We say that a finite group $G$ \emph{collapses} if any of the following 
three equivalent statements holds,
see \cite[Lemma 1.4]{AFGV-ampa}:

\medbreak
\begin{itemize}[leftmargin=*]
\item If $H$ is a finite-dimensional
pointed Hopf algebra  with group of grouplikes $G(H) \simeq G$, then $H$ is isomorphic to $\ku G$.

 \medbreak
\item If $0 \neq V \in \yd{\ku G}$, the category of Yetter-Drinfeld modules over $G$, then $\dim \toba(V) = \infty$.

\medbreak
\item If $V \in \yd{\ku G}$ is irreducible, then $\dim \toba(V) = \infty$.
\end{itemize}

For finite simple groups, the evidence gathered over the past 17 years 
supports the following conjecture, a guiding principle on this question: 

\begin{conjecture}\label{conj:groups}
If $G$ is a non-abelian finite simple group, then $G$ collapses.   
\end{conjecture}

In fact, the Conjecture, implicitly suggested in talks and articles, first appeared in print in \cite{ACG-VII}, modestly attributed to ``folklore''.  Currently, we know that the following groups collapse:

\begin{itemize}[leftmargin=*]\renewcommand{\labelitemi}{$\circ$}
\item \cite{AFGV-ampa} The alternating groups $\an$, $n \geq 5$.

\medbreak
\item \cite{AFGV-II} The sporadic groups, including the Tits group, other than the Fischer group $Fi_{22}$, the Baby Monster $B$ and the Monster $M$. 

\medbreak
\item \cite{CC-VI} The Suzuki--Ree groups.

\medbreak
\item \cite{ACG-VII}
The groups $\PSL_{n+1}(q)$, ($n \geq 3$), $\PSL_3(q)$ ($q > 2$), and $\PSp_{2n}(q)$ ($n \geq 3$).

\medbreak
\item
\cite{ACG-III}, \cite{fgvI}
The groups  $\PSL_2(q)$, where either $q\equiv 1\mod 4$ or $q >2$ is even.
\end{itemize}

\medbreak
Other finite groups of Lie type were shown to collapse in \cite{ACG-IV}.
In the present paper we prove:

\begin{theorem}\label{thm:intro-groups}
The finite simple Chevalley groups, other than $\PSL_2(q)$
with $q \equiv 3 \mod 4$ (and than $\PSL_3(2)\simeq\PSL_2(7)$), collapse.
\end{theorem}

\medbreak
In other words, if $G$ is a finite simple Chevalley group as in the theorem
and $H$ is a finite-dimensional pointed Hopf algebra  with group of grouplikes $G(H) \simeq G$, then $H$ is isomorphic to $\ku G$.

\medbreak
Because of the Classification Theorem of Finite Simple Groups, 
we see that Conjecture \ref{conj:groups} remains open only
for $\PSL_2(q)$ with $q \equiv 3 \mod 4$, 
the Steinberg groups, 
the Fischer group $Fi_{22}$, the Baby Monster $B$ and the Monster $M$.
The techniques introduced in this work will contribute to the treatment of these groups, but it seems that they will not be sufficient and new ones will be needed.

\medbreak
The proof of Theorem \ref{thm:intro-groups} relies on results
presented in \cite{ACG-I,ACG-II,ACG-III,ACG-IV,ACG-V,ACG-VII, fgvI,fgvII}. 
Actually, we have to show that $\dim \toba(V) = \infty$ whenever 
$V \in \yd{\ku \Gb}$ is irreducible, for any finite Chevalley group $\Gb$.
Recall that the support of an irreducible Yetter-Drinfeld module $V$ over $\Gb$ 
is a conjugacy class of $\Gb$. 
It was shown in the mentioned series that $\dim \toba(V) = \infty$ 
when $\supp V$ is a non-semisimple conjugacy class, up to a very short list of cases. 

\medbreak
For semisimple conjugacy classes the analysis is much more complicated because of the arithmetic
nature of such classes. In this paper we take care of the Nichols algebras associated to
semisimple classes using two new major techniques:  

\begin{itemize}[leftmargin=*]\renewcommand{\labelitemi}{$\circ$}
\item  The new criterion of type $\Omega$, valid for general groups, that allows to decide that
a Yetter-Drinfeld module has an infinite-dimensional Nichols algebra. 
This criterion is based on the main result of \cite{AHV} and depends on the support of the module only. See \S\ref{subsec:new-criterion}.

\medbreak
\item  A general criterion, based on three Lie theoretic conditions, ensuring that a semisimple conjugacy class in a Chevalley or Steinberg group satisfies the powerful criterion of type C
introduced in \cite{ACG-III}.  See Theorem \ref{thm:potente-revised}.
\end{itemize}

Note that the Theorem \ref{thm:potente-revised} follows from 
Theorem \ref{thm:potente-abstracta}, a tool for detecting conjugacy classes of type C 
in a finite group through quasi-simple subgroups. Theorem \ref{thm:potente-abstracta} may likely have applications to other families of finite groups.

\medbreak
Let us also mention that most of the methods used here 
do not allow the calculation of Nichols algebras, 
but rather provide sufficient conditions for their dimension to be infinite.

\begin{table}[t]
\caption{Kthulhu semisimple classes in $\Gb \neq \Pom^+_{N}(q)$.}\label{tab:ss-psl2}
\begin{center}
\begin{tabular}{|c|c|p{4.7cm}|c|}
\hline $\Gb$ & $q$  & class &  Remark  \\
\hline  \hline
$\PSL_{n+1}(q)$,& all & irreducible  & kthulhu   \\
$n+1$ odd prime &&& none of type $\tipo$ \\
\hline
$\PSp_{4}(q)$ & $3$, $5$, $7$  & split involution &
$q=3,5$ kthulhu\\
&&&all of type $\tipo$\\
\hline
$E_8(q)$ & all & regular, contained in a torus indexed by $w$ in class $E_8$ or $E_8(a_5)$.&
\\
\hline
\end{tabular}
\end{center}
\end{table}

\subsection{Nichols algebras over racks}

\

As explained in \cite{AG}, it saves efforts to deal with Nichols algebras of the type
$\toba(X, \bq)$, corresponding to a braided vector space $(\ku^n \otimes \ku X, c^{\bq})$. 
Here 

\medbreak
\begin{itemize} [leftmargin=*]\renewcommand{\labelitemi}{$\circ$} 
\item $X$ is a rack, i.e., $X = (X, \tridd)$ where $X \neq \emptyset$ and 
$\tridd: X \times X \to X$ is a
self-distributive operation such that $y \mapsto x \tridd y$ is bijective for any $x\in X$; 

\medbreak
\item $\bq: X \times X \to \GL(n,\ku)$ is a non-abelian 2-cocycle, i.e., 
it satisfies 
\begin{align*}
\bq_{x,y \tridd z} \bq_{y,z} &= \bq_{x \tridd y, x \tridd z} \bq_{x,z}, &
\text{for all } x,y,z &\in X; 
\end{align*}

\medbreak
\item the braiding $c^{\bq} \in  \GL((\ku^n \otimes \ku X)^{\otimes 2})$
is given with respect to the basis $(e_x)_{x\in X}$ of $\ku X$ by
\begin{align*}
c^{\bq} (v e_x \otimes w e_y) &= \bq_{x,y}(w) e_{x \tridd y} \otimes v e_x,
& x,y &\in X, \quad v,w \in \ku^n.
\end{align*}
\end{itemize}

For example,  a non-empty union of conjugacy classes in a finite group $G$ 
is a rack--these are the only racks we consider. Here the rack operation is conjugation
and is denoted by $\trid$, i.e.,
$x \trid y = xyx^{-1}$.

\medbreak
If $V\in \yd{\ku G}$ is irreducible
with support $X$, then there exists a cocycle $\bq$ such that 
$\toba(V) \simeq \toba(X, \bq)$; but the same pair $(X, \bq)$ may arise
from different Yetter-Drinfeld modules, from different groups. 
Therefore, it is convenient to deal with Question \ref{question:groups} addressing 
the following Question:

\begin{question}\label{question:racks}
\cite[Question 2]{AFGV-ampa}
Determine all pairs $(X, \bq)$, where $X$ is a finite rack
and $\bq$ is a 2-cocycle, such that $\dim \toba(X, \bq) < \infty$.
\end{question}

Parallel to the notion of  the collapse of a finite group, 
 it is convenient to introduce the following terminology.

\begin{itemize}[leftmargin=*]\renewcommand{\labelitemi}{$\circ$}

\item A conjugacy class $X$ of a finite group $G$ \emph{falls down} if $\dim \toba(V) = \infty$
for any $V \in \yd{\ku G}$ whose support is $X$.

 \medbreak
\item \cite[2.2]{AFGV-ampa}
A finite rack $X$ \emph{collapses} if for any 2-cocycle
$\bq: X\times X \to \GL(n, \ku)$  that admits a finite group $\Upsilon$ and 
$W \in \yd{\ku \Upsilon}$ with
$W \simeq (\ku X, c^{\bq})$ as braided vector spaces, 
necessarily $\dim \toba(X, \bq) = \infty$.
\end{itemize}

\medbreak
If a conjugacy class collapses, then it falls down but not conversely.

\medbreak
Now any finite rack projects onto a simple rack, i.e., a rack with at least two elements
that does not admit a projection to a smaller rack with at least two elements. 
It is natural to ask:

\begin{question}\label{question:simple-racks}
Which finite  simple racks collapse? 
\end{question}

As we will see later, we have the criteria of types C, D or F, 
ensuring that a rack collapses.
If a rack meets any of them, 
so does any rack that contains it (we say an over-rack) or projects onto it (a cover-rack
or simply a cover).
Thus, the characterization of racks that are of type C, D or F acquires greater significance. We say that a rack is \emph{kthulhu}  when it is
neither of these types. See \S \ref{subsec:Kthulhu}.

\medbreak
The classification of all simple racks is known \cite{AG}; non-trivial
conjugacy classes of finite simple groups are simple racks.
Contributing to Question \ref{question:simple-racks}, we prove
in this paper:

\begin{theorem} \label{thm:intro-racks}
A  semisimple non-trivial
conjugacy class $X$ of a Chevalley group, other than $\PSL_2(q)$, 
is 
\begin{itemize} [leftmargin=*]\renewcommand{\labelitemi}{$\circ$}
\item of type C, D or F, hence it collapses; or

\medbreak
\item 
it appears in Tables \ref{tab:ss-psl2}, \ref{tab:ss-open-POM-odd}, 
or \ref{tab:ss-open-POM-even}. If $X$ is of type $\tipo$, or not,
then it is indicated in the Table. In the first case, it collapses too. 
\end{itemize}
\end{theorem}

In the Tables, the letter $w$ refers to an element of the Weyl group
that is the parameter of the torus intersecting the conjugacy class,
see \S \ref{subsec:ss-cc}.

\pf For  $\PSL_{n+1}(q)$,   see \cite[Theorem 5.4]{ACG-VII} and Remark \ref{obs:SLp-irreducible}.
For 
$\PSp_{n}(q)$,  see \cite[Theorem 7.1]{ACG-VII}.
For $\Pom_{2n+1}(q)$ or $\Pom_{2n}(q)$,
see \S \ref{subsec:Pom-impar}, \S \ref{subsec:Pom-par} and \S \ref{sec:notC-Omega}.
For $E_6(q)$, $E_7(q)$, $E_8(q)$, $F_4(q)$, $G_2(q)$, see \S \ref{subsec:E6}, 
\S \ref{subsec:E7}, \S \ref{subsec:E8},
\S \ref{subsec:F4} and \S \ref{subsec:G2}.
\epf

\begin{remark}\label{rem:overlap in tables}
The first line in Table \ref{tab:ss-open-POM-odd}, when $d=1$, appears also
in the second line of Table \ref{tab:ss-psl2}; for recursive arguments it is
convenient to keep this overlap.
\end{remark}
\subsection{The strategy}\label{subsec:strategy}

\

Let $X$ be a rack.
As said above, we have the criteria of types C, D or F to ensure
that $X$ collapses, cf. \S \ref{subsec:Kthulhu}.
We provide here the new criterion of type $\tipo$, see \S \ref{subsec:new-criterion},
which complements but does not replace the criteria of types C,  D or F: 
on one hand,
there are racks that are kthulhu, but  of type $\tipo$, see Section \ref{sec:notC-Omega}. On the other hand
type $\tipo$ propagates to over-racks, but it is unclear,
in the current state of the art,  whether it propagates to covers.  

\begin{table}[t]
\caption{  Kthulhu semisimple classes in $\Pom^+_{2n+1}(q)$. 
\newline {\small Here }  $\varTheta \coloneqq\prod_{i=1}^{d}\alpha^{\vee}_{2i-1} (-1)$ and $d\geq1$.}\label{tab:ss-open-POM-odd}
\begin{center}
\begin{tabular}{|c|c|c|c|c|c|}
\hline 
$n$ & $w$  &  $q$  & $t$&  {\small parameter} & {\small type $\tipo$} \\
\hline  \hline
$\begin{matrix} 2d\end{matrix}$ 
&$e$& $3,5,7$  & 
$\varTheta\alpha^{\vee}_{n} (\zeta)$ 
&$\zeta= \pm1$ & {\small yes}
\\ 
\hline &$e$   & $5$  &$\varTheta \alpha^{\vee}_{n} (\omega)$ &
$\F_5^\times=\langle\omega\rangle$ & {\small yes}
\\ 
\cline{2-6}
$2d + 1$ & $s_n$  & $\equiv 3 \mod4$  &$\varTheta\alpha_n^\vee(\zeta)$ & $\zeta^{2}=-1$ & {\small yes}
\\ 
\cline{2-6}
& $s_ns_{\alpha_{n-1} + \alpha_n}$& $3$  &
$\varTheta\alpha_{n-1}^\vee(-\omega^2)\alpha_n^\vee(\omega)$ & $\F_9^\times=\langle\omega\rangle$ & {\small yes}
\\ 
\hline
\end{tabular}
\end{center}
\end{table}

\begin{table}[t]
\caption{  Kthulhu semisimple classes in $\Pom^+_{2n}(q)$.}\label{tab:ss-open-POM-even}
\begin{center}
\begin{tabular}{|c|c|c|c|c|c|}
\hline 
$n$ & $w$  &  $q$  & $t$&  {\small parameter} & {\small type $\tipo$} \\
\hline  \hline
$4$ & $s_1s_3s_4$ & $\equiv 3 \mod4$ & $\alpha^\vee_1(\omega)\alpha^\vee_3(\omega)\alpha^\vee_4(\omega)$ &$
\omega^2=-1$ & {\small yes}\\
\hline
$4$ & $s_1s_3s_4$ & $2$ & $\alpha^\vee_1(\omega)\alpha^\vee_3(\omega)\alpha^\vee_4(\omega)$ &$
\F_4=\langle\omega\rangle$ & {\small yes}\\
\hline
\end{tabular}
\end{center}
\end{table}

\medbreak
For these reasons, we follow a three-step strategy. Let $X$ be 
a semisimple conjugacy class in a simple Chevalley group $\Gb$.
The first step is based on the presentation in Theorem \ref{thm:potente-revised} of

\begin{itemize} [leftmargin=*] 
\item Sufficient Lie-theoretic conditions for $X$  to be of type C. 
\end{itemize}

\medbreak
These conditions can be rephrased in combinatorial terms and drastically reduce the case-by-case analysis. Indeed, instead of looking at conjugacy classes individually, we consider classes that are represented in a stable torus corresponding to a Weyl group element $w$, and focus on the minimal parabolic subgroup of the Weyl group containing $w$, rather than $w$. This way we move from considerations on semisimple elements in $\Gb$ to considerations at the level of root subsystems and their Weyl groups.

\medbreak
Here is the strategy:

\medbreak
\begin{step-intro} \label{step-one}
Carry out the verification of  the hypotheses in Theorem \ref{thm:potente-revised}, for generic representatives of semisimple classes running in all conjugacy classes of tori. If the hypotheses hold, then
the class is of type C, so by Theorem \ref{thm:CDF} the class, its over-racks and its covers collapse. This is carried out case-by-case in Subsections \ref{subsec:Pom-impar}, \ref{subsec:Pom-par}, \ref{subsec:E6}, \ref{subsec:E7}, \ref{subsec:E8}, \ref{subsec:F4} and \ref{subsec:G2}. This step requires a thorough analysis of conjugacy classes in Weyl groups. In addition, some classes are treated with inductive arguments based on the results of our previous articles. 
\end{step-intro}

\medbreak
\begin{step-intro} 
Detect the classes for which the previous method fails that are of type $\tipo$. In this case the class and its over-racks collapse. This is carried out in Section \ref{sec:notC-Omega}, completing the proof of Theorem \ref{thm:intro-racks}.
\end{step-intro}

\medbreak
\begin{step-intro} 
For those classes for which the above methods fail,
use knowledge on  trivial subracks of the class and show that it does not support Yetter-Drinfeld modules whose Nichols algebra is finite-dimensional. In this case the class might not collapse as a rack, and the result  does not propagate to over-racks nor covers in general.  This is carried out in Section \ref{sec:abelian-techniques}.
\end{step-intro}

\medbreak
In conclusion, Theorem \ref{thm:intro-groups} is proved going on through
the list of finite simple Chevalley groups:

\begin{itemize} [leftmargin=*]
\item The groups $\PSL_{n+1}(q)$, where either $n \geq 3$ and  $q$ is  arbitrary or else $n=2$ and 
$q > 2$, collapse by \cite[Theorem III]{ACG-VII}.

\medbreak
\item The groups $\PSL_{2}(q)$ where either $q\equiv 1 \hspace{-5pt} \mod 4$ or $q >2$ is even,
 collapse by \cite{ACG-III}, \cite{fgvI}; see Theorem \ref{thm:psl2-completo}.
 
 \medbreak
\item The groups $\Pom_{2n+1}^+(q)$, $q$ odd and $n\geq 2$, collapse
 by Theorem \ref{thm:bn-completo}.

\medbreak
 \item The groups $\PSp_{2n}(q)$, $n\geq 2$ and $(n,q)\neq(2,2)$, collapse
 by Theorem \ref{thm:cn-completo}.

\medbreak
\item  The groups $\Pom_{2n}^+(q)$, any $q$ and $n\geq 4$, collapse
 by Theorem \ref{thm:dn-completo}.

\medbreak
\item The groups $E_6(q)$, any $q$, collapse
 by Theorem \ref{thm:e6-collapses}.

\medbreak
\item  The groups $E_7(q)$, any $q$, collapse
 by Theorem \ref{thm:e7-collapses}.

\medbreak
\item  The groups $E_8(q)$, any $q$, collapse
 by Theorem \ref{thm:e8-completo}.

\medbreak
\item The groups $F_4(q)$, any $q$, collapse
 by Theorem \ref{thm:f4-collapses}. 
 
\medbreak
\item The groups $G_2(q)$, any $q > 2$, collapse
 by Theorem \ref{thm:g2-collapses}. 
\end{itemize}

 \medbreak
The theorems in the list above, as well as Theorem \ref{thm:intro-racks},
have been proved using results from 
\cite{ACG-I,ACG-II,ACG-III,ACG-IV,ACG-V,ACG-VII,fgvII}, as explained in each case.

\medbreak
Theorem \ref{thm:potente-revised} and the verification of type $\tipo$ require a good understanding of families of subgroups in groups of Lie type containing semisimple elements. On this issue, our work has been very much inspired by the results in \cite[Section 3]{GLRY} and \cite{seitz}. 

\medbreak
We close the introduction with two remarks.

\begin{remark}
The existence of kthulhu classes in some Chevalley groups, for special values 
of $q$ or special families of conjugacy classes, 
suggests that a case-by-case analysis is somehow unavoidable 
and is related to the arithmetic aspects of semisimple classes.
\end{remark}

\begin{remark}
The irreducible semisimple classes in $\PSL_2(q)$ for $q\equiv 3\mod 4$
and in some unitary groups seem to elude the strategy, 
thus requiring different methods. 
We expect that an enhancement of the deformation technique developed in \cite{HMV,AHV} would lead to advances on these cases.
\end{remark}

\subsection{Notation}\label{subsec:notations}

\

For $a,\,b\in \N$, with $a<b$ the set $\{a,\,\ldots,\,b\}$ is denoted by $\I_{a,b}$. We write $\I_b$ for $\I_{1,b}$. If $X$ and $Y$ are subsets of a set $Z$, then the complement of $Y$
in $X$ is denoted by $X \setminus Y$.

\medbreak
The group of $n$-th roots of unity in $\mathbb C$ is denoted by $\mu_n$ and $\mu_\infty=\bigcup_{n\geq0}\mu_n$. The characteristic polynomial of an operator $A$ is denoted by $c_A(X)$ and the $n$-th cyclotomic polynomial is denoted by $\varphi_n(X)$.

\medbreak
Let $G$ be a  group; its group of automorphisms is denoted by $\Aut G$ and its center
by $Z(G)$.
We write $N \leqslant  G$ to express that $N$ is a subgroup
or $G$; $N\nsbgp$ means that $N$ is normal. 

\medbreak
If $p$ is a prime number, then $G_p$ is the set of $p$-elements of $G$. 

\medbreak
Let $X\subset G$. 
The centralizer, respectively, the normalizer,
of $X$ in $H \leqslant  G$ is denoted by $\Cent_H(X)$, respectively $N_H(X)$;
while the subgroup generated by $X$ is denoted by $\langle X\rangle$.

\medbreak
A rack is \emph{trivial}\footnote{Sometimes called abelian, but trivial is preferable
to avoid confusions with other notions.} 
if  $x \tridd y = y$ for any $x,y \in X$ (but the trivial conjugacy class is the class
of the neutral element).
If $x,y\in G$, then $x \trid y\coloneqq xyx^{-1}$; this makes $G$ a rack; 
any subset of $G$ stable under conjugation is a subrack.
The conjugacy class
of $g \in G$ is denoted by $\Oc_g^G$. 
We extend this notation to
the conjugacy action of $H \leqslant G$ on $G$, 
for which the orbit of $g\in G$ is denoted by $\Oc_g^H$. 
If $g\in N_G(H)$ then  $\Oc_g^H$ is a subrack of $\Oc_x^G$. 
For the basics of racks, see \cite{AG}.

\medbreak
We collect notations and conventions on finite groups of Lie type  in 
\S\,\ref{subsec:notation-lie}. 

For background on Hopf algebras and Nichols algebras, and undefined notions mentioned in
the present paper, see \cite{ACG-I,ACG-II,ACG-III,ACG-IV,ACG-V,ACG-VII,andrus-leyva,A-Schneider}.

\section{Nichols algebras over racks}\label{sec:Nichols-racks}
\subsection{Kthulhu conjugacy classes} \label{subsec:Kthulhu}

\

To start with, we recall that a conjugacy class $\Oc$ of a finite group $G$
is of type C, D, or F when the corresponding property below holds. 

\medbreak
\begin{enumerate}[leftmargin=*,label=\rm{(\Alph*)}]\setcounter{enumi}{2}
\item\label{item:typeC-group} There are $H \leqslant G$ and $r,s\in H\cap\Oc$ such that
$rs\neq sr$, $H=\langle \Oc_r^H,\Oc_s^H\rangle$, $\Oc_r^H\neq \Oc_s^H$ and
$\min \{\vert\Oc_r^H\vert, \,\vert\Oc_s^H\vert\}>2$ or $\max \{\vert\Oc_r^H\vert, \,\vert\Oc_s^H\vert\}>4$. 

\medbreak
\item\label{item:typeD-group} There exist $r$, $s\in \Oc$
such that $\Oc_r^{\langle r, s\rangle}\neq \Oc_s^{\langle r, s\rangle}$ and $(rs)^2\neq(sr)^2$.

\medbreak \setcounter{enumi}{5}
\item\label{item:typeF-group}
There are $r_a\in \Oc$, $a\in \I_4 = \{1,2,3,4\}$,
such that $\Oc_{r_a}^{\langle r_a: a\in \I_4\rangle}\neq \Oc_{r_b}^{\langle r_a: a\in \I_4\rangle}$  and
$r_a r_b \neq r_br_a$ for $a  \neq b\in \I_4$.
\end{enumerate}

\medbreak
Being of type C, D or F are properties of the underlying rack
and not of the particular incarnation as a conjugacy class, see \cite{ACG-I,ACG-III,AFGV-ampa}.

\medbreak
A conjugacy class which is neither of type C, D, nor F, is \emph{kthulhu}.
The next Theorem is a consequence of \cite{HS,HV1}.

\begin{theorem}\label{thm:CDF} \cite[Theorem  3.6]{AFGV-ampa},  
\cite[Theorem 2.8]{ACG-I},  \cite[Theorem 2.9]{ACG-III}.
If a conjugacy class $\Oc$ is type C, D or F, then $\Oc$ \emph{collapses}. 
\qed
\end{theorem}

Let $Y$ be a finite rack of type C. Being of type C spreads easily:

\medbreak
 \begin{enumerate}[leftmargin=*,label=\rm{(\roman*)}]
\item\label{item:prop-inclusion}
If $Y \hookrightarrow Z$ is an injective morphism of racks, then $Z$ is of type C.

\medbreak
\item\label{item:prop-projection}
If $X \twoheadrightarrow Y$ is a surjective morphism of racks, then $X$ is of type C.
\end{enumerate}

\medbreak
The properties \ref{item:prop-inclusion} and \ref{item:prop-projection} are also valid for types D and F, see \cite{AFGV-ampa,ACG-I}.

\medbreak
Here is a way to detect conjugacy classes of type C.

\begin{lema}\label{lema:normal-simple}
 Let $G$ be a group and let $N \lhd G$. Assume $N$ is simple non-abelian and let $r\in N$. 
 If  $\Oc_r^G\cap N\neq \Oc_r^N$, then $\Oc_r^G$ is of type C.
\end{lema}

 \pf Evidently, $r \neq e$.
 Pick $y\in \Oc_r^G\cap N\setminus\Oc_r^N$. Clearly, $\Oc_r^N \cap \Oc_y^N = \emptyset$.
 The groups $\langle \Oc_r^N\rangle$ and $\langle \Oc_y^N\rangle$ are non-trivial normal subgroups of $N$, hence $N= \langle \Oc_r^N\rangle= \langle \Oc_y^N\rangle =\langle \Oc_r^N, \Oc_y^N\rangle$. 
In addition, since $Z(N)=e$, there is $s\in \Oc_y^N$ such that $s\trid r\neq r$. Finally, if $\left| \Oc_r^N\right|\leq 2$, then $N=\langle\Oc_r^N\rangle$ would be abelian, and similarly for $\left| \Oc_r^N\right|\leq 2$. 
\epf

An application of Lemma \ref{lema:normal-simple}
is given in Lemma \ref{lem:normal-simple-psu}, \S\ref{subsec:G2}.

\subsection{Detecting type C via quasi-simple subgroups} 
 
 \ 
 
 Here we provide a technique to show that some conjugacy classes 
 in a finite group with a quasi-simple subgroup are of type C.
 We fix the following data:  
 
 \begin{itemize} [leftmargin=*]
\item  $\gp$ is a finite group, 

\medbreak
\item $\subgp \leqslant\gp$, 

\medbreak
\item $\normsubgp \leqslant N_{\gp} (\subgp)$ is an abelian subgroup and

\medbreak
\item  $\letra \in  \normsubgp$.
 \end{itemize}

\begin{theorem} \label{thm:potente-abstracta}
Assume that 
\begin{align} 
\tag*{\fbox{\textsf q}} \label{item:quasi-simple}
&\subgp \leqslant\gp \text{ is quasi-simple,}
\\ 
\tag*{\fbox{\textsf n}} \label{item:non-central} 
& \letra \notin  \Cent_{\gp} (\subgp),
\\
\tag*{\fbox{\textsf w}} \label{item:dos-orbitas-hyp} 
&(\Oc_{\letra}^{\gp} \cap \normsubgp) 
\,\backslash \left(\left(\Oc_{\letra}^{\subgp}  \cap \normsubgp\right) \cup \Cent_{\gp} (\subgp) \right) \neq \emptyset.
\end{align}
Then $\Oc^{\gp}_{\letra}$ is of type C.
\end{theorem}

 Here $\Oc_{\letra}^{\subgp}$ denotes the orbit of $\letra$ by
 the conjugation of $\subgp$ on $\gp$, see \S\, \ref{subsec:notations}.
The notations \ref{item:quasi-simple}, \ref{item:non-central} and \ref{item:dos-orbitas-hyp}
arise from their counterparts in Theorem \ref{thm:potente-revised}.

 \medbreak
 We need a group $\subgpH$ and $\letrados \in \subgpH$ satisfying the conditions in \ref{item:typeC-group},  see page \pageref{item:typeC-group}. 
 The next Lemma provides candidates for $\subgpH$ and $\letrados$ and 
 the verification of one condition. We retain the previous notation, but 
 we still assume neither \ref{item:quasi-simple}, 
 nor \ref{item:non-central} nor \ref{item:dos-orbitas-hyp}.

 \begin{lema}\label{lem:dos-orbitas} 
Assume that  there exists $\letrados  \in  (\Oc_{\letra}^{\gp} \cap  \normsubgp) \,  \backslash \,  (\Oc_{\letra}^{\subgp}  \cap \normsubgp)$  and
set $\subgpH\coloneqq \langle \letra,\letrados, \subgp\rangle = \langle \letra,\letrados\rangle \subgp
=\subgp \langle \letra,\letrados\rangle$. Then 
\begin{align*}
\Oc_{\letrados}^{\subgpH} &\neq \Oc_{\letra}^{\subgpH}.
\end{align*} 
\end{lema}

\begin{proof}
Suppose that $\Oc_{\letrados}^{\subgpH} = \Oc_{\letra}^{\subgpH}$, i.e.,  
$\letrados \in  \Oc_{\letra}^{\subgpH}$. This says that
\begin{align*}
\letrados \in  \subgpH \trid r=  \subgp \trid (\langle \letra,\letrados\rangle \trid \letra) 
\overset{\star}{=}  \subgp \trid \letra
= \Oc_{\letra}^{\subgp},
\end{align*}
contradicting the assumption. Here $\star$ is valid because $\normsubgp$ is abelian.
\end{proof}
 
The following lemma is useful for verifying the other conditions.
 
\begin{lema}\label{lem:generate-size0} 
Assume that  \ref{item:quasi-simple}  holds. 
If there exists $\letratres \in N_{\gp} (\subgp) \, \backslash \,  \Cent_{\gp} (\subgp)$, then 
\begin{align*}
\langle \Oc_{\letratres}^{\subgpK}\rangle &=\subgpK 
&&\textrm{and}
&\left|\Oc_{\letratres}^{\subgpK}\right| & \geq 3.
\end{align*}
where we set 
\begin{align*}
\subgpK &\coloneqq \langle \letratres,\subgp \rangle = \langle \letratres\rangle \, \subgp.
\end{align*}
\end{lema}

\begin{proof}
Consider the subgroups 
\begin{align*}
N_1 &\coloneqq \langle \Oc_{\letratres}^{\subgpK}\rangle \geq N_2\coloneqq N_1\cap \subgp.
\end{align*}
 Since $\letratres$ normalizes $\subgp$, we see that $\subgp\lhd\subgpK$. Also,
$N_1 \lhd \subgpK$ and contains $\langle \letratres\rangle$. 
Now $\letratres \in N_{\gp} (N_2)$, whence $N_1=\langle \letratres\rangle N_2$ and $N_2\lhd \subgp$. 
Since $\subgp$ is quasi-simple, we see  that either $N_2=\subgp$ or else  $N_2\subseteq Z(\subgp)$. 
If $N_2=\subgp$, then  $N_1  =\langle \letratres\rangle N_2 = \langle \letratres\rangle \subgp =
\subgpK$ hence $\subgpK = \langle \Oc_{\letratres}^{\subgpK}\rangle$ as desired.

We show that the second scenario cannot occur. Assume for a contradiction that $N_2\subseteq Z(\subgp)$.
Observe that $[\mathbf m,\letratres]\in N_2$ for any $\mathbf m\in \subgp$. 
The assignment $\mathbf m\mapsto [\mathbf m,\letratres]$ determines then a group homomorphism  
$\psi: \subgp\to Z(\subgp)$. 
Indeed, given $\mathbf m, \mathbf n \in \subgp$, we have
\begin{align*}
[\mathbf m,\letratres] [\mathbf n, \letratres] = \mathbf m(\letratres \mathbf m^{-1} \letratres^{-1}\!)\!( \mathbf n\letratres \mathbf n^{-1}\letratres ^{-1}\!)
=  \mathbf m (\mathbf n\letratres \mathbf n^{-1}\letratres ^{-1}\!)\!(\letratres \mathbf m^{-1} \letratres^{-1}\!)=
[\mathbf m \mathbf n, \letratres].
\end{align*}

Again, either $\ker \psi \subset Z(\subgp)$ or $\ker \psi = \subgp$. 
In the first case, we have a canonical surjective map $\subgp/\ker \psi  \to \subgp/Z(\subgp)$; but 
$\subgp/\ker \psi$ is abelian while
$\subgp/Z(\subgp)$ is non-abelian simple, a contradiction. Thus
 $\ker \psi = \subgp$, that is, $[\subgp,\letratres]=e$, in other words,
$\letratres \in \Cent_{\gp} (\subgp) $, contradicting the hypothesis. 

\medbreak
Finally, $\left|\Oc_{\letratres}^{\subgpK}\right|\geq 3$:  otherwise $\subgpK$, 
and a fortiori $\subgp$, would be abelian.
\end{proof}

\noindent \emph{Proof of Theorem}  \ref{thm:potente-abstracta}.
The hypothesis \ref{item:dos-orbitas-hyp} states that there exists $\letracuatro \in \Oc_{\letra}^{\gp} \cap \normsubgp$ such that
\begin{align*}
 \letracuatro &\notin \Oc_{\letra}^{\subgp}  \cap \normsubgp, &  \letracuatro &\notin \Cent_{\gp} (\subgp).
\end{align*}
Set $\subgpH\coloneqq \langle \letra,\letracuatro, \subgp\rangle = \langle \letra,\letracuatro\rangle \subgp
=\subgp \langle \letra,\letracuatro\rangle$. By Lemma \ref{lem:dos-orbitas} we have 
\begin{align*}
	\Oc_{\letracuatro}^{\subgpH} &\neq \Oc_{\letra}^{\subgpH}.
\end{align*} 
Clearly, $\letracuatro \in N_{\gp} (\subgp) \, \backslash \,  \Cent_{\gp} (\subgp)$;
indeed, $\letracuatro \in \normsubgp \subset N_{\gp} (\subgp)$ by assumption. 
Then Lemma \ref{lem:generate-size0} applied to $\letratres = \letra$
and $\letratres = \letracuatro$ implies  that
\begin{align*}
	\langle \Oc_{\letra}^{\langle \letra \rangle \, \subgp}\rangle &=\langle \letra \rangle \, \subgp 
&	&\text{and}&
		\langle \Oc_{\letracuatro}^{\langle \letracuatro \rangle \, \subgp}\rangle &=\langle \letracuatro \rangle \, \subgp. 
\end{align*}
Therefore
\begin{align*}
\subgpH =   \langle \letra,\letracuatro\rangle \subgp =  \langle \letra\rangle \, \langle\letracuatro\rangle \subgp
= \langle \letra\rangle  \, \langle \Oc_{\letracuatro}^{\langle \letracuatro \rangle \, \subgp}\rangle
= \langle \letra\rangle  \, \langle \Oc_{\letracuatro}^{\subgpH}\rangle
\end{align*}
because $[\letra,\letracuatro]=e$. Swapping the roles of $\letra$ and $\letracuatro$ gives $\subgpH =\langle \letracuatro\rangle \, \langle \Oc_{\letra}^{\subgpH}\rangle$. Thus
\begin{align*}
\subgpH \geqslant \langle \Oc_{\letra}^{\subgpH},\Oc_{\letracuatro}^{\subgpH}\rangle
\geqslant \langle \Oc_{\letra }^{\subgpH},\letracuatro\rangle=\subgpH.
\end{align*}
Hence, $\subgpH = \langle \Oc_{\letra}^{\subgpH},\Oc_{\letracuatro}^{\subgpH}\rangle$.
Lemma \ref{lem:generate-size0} also implies that $\left|\Oc_{\letra }^{\subgpH}\right|,\,\left|\Oc_{\letracuatro}^{\subgpH}\right|\geq 3$. 

\medbreak
Finally, we claim that there exists $\letrados\in \Oc_{\letracuatro}^H$ satisfying $[\letra,\letrados]\neq e$. Indeed, if 
$[\letra, \letratres]=e $ for all $\letratres\in \Oc_{\letracuatro}^H$, then $[\letra, \letratres]=e $ for all 
$\letratres\in \langle \Oc_{\letracuatro}^H\rangle=\langle \letrados \rangle \subgp$, but this contradicts \ref{item:non-central}. \qed

\subsection{Solvable racks} 

\

In this subsection we introduce the notion of \emph{solvable rack}, which is
needed for the new criterion of type $\tipo$ that is discussed in Subsection \ref{subsec:new-criterion}.

\medbreak We start with some notation; see \cite{AG} for undefined terminology.
Given a set $I$, the group of 
bijections $I \to I$ is denoted by $\mathbb S_I$.
Let $(X, \tridd)$ be a rack and let  $\phi:X \to \mathbb S_X$ be 
given by $\phi_x(y) = x \tridd y$.
The inner group $\Inn X$ of $X$ is the subgroup of $\mathbb S_X$  
generated by the image of $\phi$. 

\begin{lema}\label{lema:innX} \cite[Lemma 1.9]{AG}
Let $X$ be a rack such that there exist a group $G$ 
and an injective morphism of racks $\varphi: X \to G$
such that $\Imm \varphi$ is a union of conjugacy classes of $G$.
If $\varphi(X)$ generates $G$,  then $\Inn X \simeq G/ Z(G)$. 
\end{lema}

\begin{definition} \label{def:solvable rack}
A non-trivial rack $X$ is \emph{solvable} 
if there exists a finite solvable  group $G$ and 
a conjugacy class $\Oc$ of $G$ such that 
\begin{align*}
X &\simeq \Oc & &\text{and} & \Oc &\text{ generates } G.
\end{align*}
\end{definition}

\begin{obs}\label{obs:simplify-definition}
The hypothesis $\langle \Oc\rangle = G$ is quite strong; for instance, it
implies that $\Oc$ is indecomposable, see \cite[Lemma 1.15]{AG} (here indecomposable
means that it is not the disjoint union of two subracks).
Also, since $X$ is required to be non-trivial, the group $G$  is necessarily non-abelian
and a fortiori non-cyclic. 
\end{obs}
 
\begin{obs} \label{rem:solvable-rack}
Let $X$ be a solvable rack and let $H$ 
be a finite group with a conjugacy class $\Oc \simeq X$.
Then the subgroup $L$ generated by $\Oc$ is solvable. 
\end{obs}

\pf Let $G$ be as in Definition \ref{def:solvable rack}.
By Lemma \ref{lema:innX}, $\Inn X$ is solvable.
Since $\Inn X \simeq L / Z(L)$, $L$ is also solvable.
\epf

Let $\varGamma$ be a finite abelian group, 
$\langle t\rangle$ a cyclic group acting on $\varGamma$ by group automorphisms and 
$\aut  \in \Aut \varGamma$, $\aut (a)  = t \cdot a$, $a \in \varGamma$. 
The \emph{affine rack} $\Aff(\varGamma,\aut)$ has underlying set
$\varGamma$ and rack multiplication
\begin{align*}
a \tridd b &= \aut(b) + (\id - \aut)(a), & a,b &\in \varGamma.
\end{align*}
In terms of the conjugation  in the semidirect product 
$\varGamma \rtimes \langle t\rangle$, we have
\begin{align*}
(a \tridd b ,t) &= (a, t)(b,t)(a, t)^{-1}, & a,b &\in \varGamma.
\end{align*}
Thus $\Aff(\varGamma, \aut)$ can be identified with the subrack $\varGamma \times \{t\}$ of $\varGamma \rtimes \langle t\rangle$ (considered as rack with the conjugation), 
and $\Oc_{(0,t)}^{\varGamma \rtimes \langle t\rangle}$ 
can be identified with a subrack of $\Aff(\varGamma, \aut)$.

 \begin{lema} \label{lema:twisted-homogeneous}
The following are equivalent: 
\begin{enumerate}[leftmargin=*,label=\rm{(\alph*)}]
\item\label{item:affine-a} 
$\Aff(\varGamma,\aut)$ is indecomposable,

\item\label{item:affine-b}   
$\Aff(\varGamma,\aut)\simeq \Oc_{(0,t)}^{\varGamma \rtimes \langle t\rangle}$,

\item\label{item:affine-c} 
 $\id-\aut \in \End \varGamma$ is surjective,

\item\label{item:affine-d}  $\id-\aut \in \End \varGamma$ is injective,

\item\label{item:affine-e}  the centralizer $\Cent_\varGamma(t)$  is trivial.
\end{enumerate}
Furthermore, an indecomposable affine rack $\Aff(\varGamma, \aut)$ 
is solvable.
\end{lema}
\pf Left to the reader.
\epf

\begin{prop}\label{prop:characterization-affine}
Let $G$ be a finite group, $t \in G$ and $\varGamma \leqslant G$  abelian 
such that the following  conditions hold:
\begin{enumerate}[leftmargin=*,label=\rm{(\alph*)}]
\item\label{item:characterization-affine1}   
$t$ normalizes $\varGamma$; set $\aut  \in \Aut \varGamma$, $\aut (a)  = t \trid a$, $a \in \varGamma$.

\item\label{item:characterization-affine3}  The centralizer $\Cent_\varGamma(t)$ is trivial. 
\end{enumerate}
Then $\Aff(\varGamma,\aut)$ can be identified with a subrack of $\Oc_{t}^{G}$ of size $|\varGamma|$.
\end{prop}

\pf Left to the reader.
\epf

\subsection{A new criterion on Nichols algebras}\label{subsec:new-criterion}

\

Finite-dimensional Nichols algebras on solvable racks are treated with the following result.

\begin{theorem}\label{thm:racks-solvable-collapse} \cite[Theorem 6.14, Corollary 6.15]{AHV}
Let $G$ be a finite solvable non-cyclic group, and 
$V \in \yd{\ku G}$ such that 
$\supp V$ is a conjugacy class of $G$ generating $G$ and $\dim\toba(V)<\infty$. 
Then the rack $\supp V$ is isomorphic to either of the following ones:
\begin{enumerate}[leftmargin=*,label=\rm{(\roman*)}]
\item\label{item:solvable-affine} 
$\Aff(\varGamma, \aut )$ where $\varGamma = \kk _q$, 
$\aut $ is multiplication by $d \in \kk _q^{\times}$ and
\begin{align}\label{eq:list-affine-simple-racks}
(\kk _q,d)\in\{(\kk _3,2), (\kk _4,\omega), (\kk _5,2),
(\kk _5,3),(\kk _7,3),(\kk _7,5)\}, 
\end{align}
where $\omega \in \kk_4$ is different from 0 and 1.

\medbreak
\item\label{item:solvable-non-affine}   
The conjugacy class $\Oc^4_j$ of $j$-cycles in $\mathbb S_4$, where $j = 2$ or $4$. \qed
\end{enumerate}
\end{theorem}

 \ref{item:solvable-affine} covers the case when $\supp V$ is an affine indecomposable non-trivial rack.
Also neither $\Oc^4_2$ nor $\Oc^4_4$ are affine, 
see the proof of \cite[Corollary 6.15]{AHV}.

\medbreak 
Conversely, any  rack in \ref{item:solvable-affine} or \ref{item:solvable-non-affine} is the support of a finite-dimensional Nichols algebra, or more than one. See
\cite[Examples 2.19, 2.20, 2.21, 2.22]{AHV}.

\medbreak 
Theorem \ref{thm:racks-solvable-collapse} provides a criterion for detecting infinite-dimensional Ni\-chols algebras.

\begin{definition}\label{def:typeZ}
A finite rack $Y$ is of \emph{type $\tipo$} if it has a
solvable subrack $X$ that neither belongs to the list \eqref{eq:list-affine-simple-racks}
nor is it isomorphic to $\Oc^4_2$ or $\Oc^4_4$.
\end{definition}

For example, a finite rack that has a non-trivial indecomposable  affine subrack 
that does not belong to the list \eqref{eq:list-affine-simple-racks}
is of type $\tipo$.

\begin{theorem}\label{thm:typeZ-1}
Let $Y$ be a finite rack of type $\tipo$. Then $Y$ collapses. 
\end{theorem}

 \begin{proof} Assume that there exist a 2-cocycle $\bq: Y\times Y \to \GL(n, \ku)$, 
a finite group $\Upsilon$  and  $W \in \yd{\ku \Upsilon}$ such that
$W \simeq (\ku Y, c^{\bq})$ as braided vector spaces. 
By assumption, $W = \bigoplus_{g \in Y} W_g$.
Let $X$ be a solvable subrack of $Y$ that is neither in \eqref{eq:list-affine-simple-racks} nor isomorphic to $\Oc^4_2$ or $\Oc^4_4$.
 Let $G = \langle X \rangle \leqslant  \Upsilon$; then
$V \coloneqq \bigoplus_{g \in X} W_g  \in \yd{\ku G}$.
By Remarks \ref{obs:simplify-definition} and  \ref{rem:solvable-rack}, $G$ is solvable but not cyclic.
By Theorem \ref{thm:racks-solvable-collapse},  $\dim \toba(V) = \infty$, 
hence $\dim \toba(W) = \infty$.
\end{proof}

\begin{cor}\label{cor:solvable-order} Let $G$ be a finite group, $\mathbf S \leq G$ solvable
and $t\in \mathbf S \backslash Z(\mathbf S)$ such that $\mathbf S=\langle \Oc_t^{\mathbf S}\rangle$ and 
$|\Oc_t^{\mathbf S}|\notin \{3,4,5,6,7\}$. Then $\Oc_t^G$ is of type $\tipo$, 
hence collapses.
\end{cor}
\begin{proof} The rack $\Oc_t^G$ is non-trivial; otherwise $\mathbf S$ is abelian
 contradicting that $t\notin  Z(\mathbf S)$.
Hence, $\Oc_t^{\mathbf S}$ is a solvable subrack of $\Oc_t^G$. By size reasons it is neither in \eqref{eq:list-affine-simple-racks} nor isomorphic to $\Oc^4_2$ or $\Oc^4_4$.
\end{proof}

In a favorable situation one may use the following result, whose conditions are rather restrictive but might be easy to verify.

\begin{cor}\label{cor:normal-abelian-order}
Let  $G$ be a finite group, $t \in G$ and $\varGamma \leqslant  G$ abelian satisfying  conditions 
\ref{item:characterization-affine1} and  \ref{item:characterization-affine3} in Proposition
\emph{\ref{prop:characterization-affine}}.  If $(\varGamma, \aut)$
is not in  \eqref{eq:list-affine-simple-racks}, in particular if
$|\varGamma|\notin \{2,3,4,5,7\}$, then $\Oc_t^G$ is of type $\tipo$, 
hence it collapses. \qed
\end{cor}

\begin{exa}\label{exa:SL2-unipotent} Let $G=\SL_2(2^m)\simeq\PSL_2(2^m)$ with $m>1$.  The (unique) class $\Oc$ of non-trivial involutions (i.e., the class of non-trivial unipotent elements) in $G$ is of type $\tipo$, therefore it collapses. 
\end{exa}

\pf
Indeed, $G$ contains a dihedral subgroup  $D$ of order $2(2^m+1)$ (i.e., the normaliser of the non-split torus), so any reflection in $D$ is a representative of $\Oc$.  The group $\varGamma$ of rotations in  $D$ is normalised by $t$, and, as $t$ acts by inversion on it, it centralises only $e$ because $2^m+1$ is odd. If $m\neq 2$, then $|\varGamma|\notin \{2,3,4,5,7\}$. If $m=2$, then  $|\varGamma|=5$ but $(\varGamma, \aut)$ is not in  \eqref{eq:list-affine-simple-racks}. 
\epf

 \subsection{Propagation}
 
\

We compare the criterion given in Theorem \ref{thm:typeZ-1} with the criteria
of types C,  D or F. 

\begin{itemize} [leftmargin=*]\renewcommand{\labelitemi}{$\circ$} 
\item There are  racks that are kthulhu
but of type $\tipo$, see Example \ref{exa:SL2-unipotent} and Section \ref{sec:notC-Omega}. 
\end{itemize}

\medbreak 
Indeed, racks of type C, D, or F have a decomposable subrack 
by definition, while for type $\tipo$ we need indecomposable (solvable) subracks. 

\medbreak 
\begin{itemize} [leftmargin=*]\renewcommand{\labelitemi}{$\circ$} 
\item Type $\tipo$ propagates less than  types C, D or F.
\end{itemize}

\medbreak 
Let $Y$ be a finite rack of type $\tipo$.
If $Y \hookrightarrow Z$ is an injective morphism of racks, then $Z$ is of type 
$\tipo$ too. But if $X \twoheadrightarrow Y$ is a surjective morphism of racks, then it is unclear whether $X$ will be of type $\tipo$.
We ask:

\begin{question}
If $X \twoheadrightarrow Y$ is a surjective morphism of racks and $Y$ is
of type $\tipo$, is necessarily  $X$ of type $\tipo$?
\end{question}

One possible way to address this issue is to focus on cocycles. Namely, 
let $X$ be a finite rack  and let $\bq: X\times X \to \GL(n, \ku)$ be a 2-cocycle.
 We say that $\bq$ is \emph{liftable from a rack $Y$} if
 there exists a surjective morphism of racks 
  $\varpi: X \to Y$ 
 and a 2-cocycle $\bp: Y\times Y \to \GL(n, \ku)$ such that
 $\bq = \bp (\varpi \times \varpi)$. Then 
 $\varpi$ induces a surjective morphism of braided vector spaces 
 $(\ku X, \bq) \twoheadrightarrow (\ku Y, \bp)$, hence a surjective morphism of 
 Nichols algebras
 $\toba(X, \bq) \twoheadrightarrow \toba(Y, \bp)$.

\begin{question}
Let $X$ be a finite rack. Is any 2-cocycle $\bq: X\times X \to \GL(n, \ku)$ 
liftable from a simple rack $Y$?
\end{question}
 
 Constant cocycles are liftable along any surjective morphism of racks 
  $\varpi: Y \to X$. We conclude from Theorem \ref{thm:typeZ-1}:

\begin{lema}
Let $X$ be a finite  rack and let $\bq: X\times X \to \GL(n, \ku)$ 
be a constant 2-cocycle. If there exists a surjective morphism of racks 
$\varpi: X \to Y$ where $Y$ is of type $\tipo$, then 
$\dim \toba(\ku Y, \bq) = \infty$. \qed
\end{lema}

The following remark highlights the interplay of the different criteria.

\begin{obs}
Let  $G$ be a finite group and $t \in G$. 
Assume that  $t\neq t^j\in\Oc_t^G$, for some $j\in\NN$, and that $\varGamma \leqslant  G$ is abelian and satisfies  conditions 
\ref{item:characterization-affine1} and  \ref{item:characterization-affine3} in Proposition
\ref{prop:characterization-affine}. 

If $|\varGamma|>4$, or $|\varGamma|=3$ and $\Cent_{\varGamma}(t^j)\neq \varGamma$, or $|\varGamma|=4$ and $|\Cent_{\varGamma}(t^j)|\neq 2$, then $\Oc_t^G$ is of type C, with $r\coloneqq t^j$, $s\coloneqq tx$ for some $x\in\varGamma$ and $H=\langle t,\varGamma\rangle$.  

\medbreak
In this situation, due to better propagation, it is more convenient to use the type C criterion rather than Corollary \ref{cor:normal-abelian-order}. However, type $\tipo$ can be more effective when dealing with conjugacy classes of involutions,
see Example \ref{exa:SL2-unipotent}. 
\end{obs}

\section{Preliminaries on finite groups of Lie type}
\subsection{Groups of Lie type}\label{subsec:notation-lie}

\

We fix $q=p^m$, where $p$ is a prime number and  $m\geq 1$. 
The Galois field with $q$ elements is denoted by $\kk_q$.

\medbreak
Let $\G$ be a simple algebraic group over $\ku \coloneqq\overline{\mathbb F_p}$ 
of rank $n$ as in \cite{malle-testerman}.
We shall use the following notation:

\medbreak
\begin{itemize} [leftmargin=*]\renewcommand{\labelitemi}{$\circ$} 
\item $\G_{sc}$ and $\G_{ad}$ are the simply-connected cover and the adjoint quotient  of $\G$, respectively;

\medbreak
\item $\pi\colon \G\to \G_{ad}$ is the natural isogeny; 

\medbreak
\item $\T$ is a fixed maximal torus in $\G$, $W=N_{\G}(\T)/\T$ is the Weyl group of $\G$;

\medbreak\item $\Phi$ is the root system of $\G$ with respect to $\T$ and $\Delta=\{\alpha_1,\,\ldots,\,\alpha_n\}$ is a fixed set of simple roots with corresponding set of positive roots $\Phi^+$, their numbering follows \cite{hu-la};

\medbreak
\item $\alpha_0$ is the highest root in $\Phi$ with respect to $\Delta$;

\medbreak\item For $\gamma\in \Phi$,  $\gamma^\vee\colon \ku\to \T$ is the cocharacter satisfying 
\begin{align*}
\chi(\gamma^\vee(\zeta))=\zeta^{\frac{2(\chi,\gamma)}{(\gamma,\gamma)}}
\end{align*}
for $\zeta\in\ku^\times$ and $\chi$ in the weight lattice of $\G$. For any $t\in \T$ there exist $\zeta_1,\,\ldots,\,\zeta_n\in\ku^\times$ such that $t=\prod_{j=1}^n\alpha_j^\vee(\zeta_j)$, with uniqueness if $\G=\G_{sc}$.
 \end{itemize}

\medbreak
\begin{itemize} [leftmargin=*]\renewcommand{\labelitemi}{$\circ$} 
\item $\U_{\alpha}$ is the root subgroup of $\G$ corresponding to $\alpha$ in $\Phi$ and $x_\alpha\colon \ku\to \U_{\alpha}$ is a fixed isomorphism of algebraic groups;

\medbreak
\item $\B\coloneqq \langle\T,\,\U_{\alpha},\;\;\alpha\in\Phi^+\rangle$
is the Borel subgroup associated with $\Delta$;

\medbreak
\item $s_i\coloneqq s_{\alpha_i}$, $i\in \I_{0,n}$; so $S=\{s_1,\,\ldots,\,s_n\}$  is  a set of Coxeter generators for $W$ and $w_0$ is the corresponding longest element;

\medbreak\item If $J\subseteq S$ then $\Delta_J\subseteq \Delta$ is the corresponding set of simple roots, $\Phi_J=\Phi\cap {\mathbb Z}\Delta_J$, and  $\Le_J$, respectively $\Pa_J$ are the corresponding standard Levi and parabolic subgroups in $\G$.

\medbreak
\item $F$ denotes either a Frobenius or a Steinberg  endomorphism, that is, $F=\vartheta\circ \Fr_q$ where $\vartheta$ is a (possibly trivial) automorphism of $\G$ preserving $\T$ and $\B$, and $\Fr_q$ is a Frobenius endomorphism acting on $\T$ by raising each element to the $q$-th power; 

\medbreak
\item $\G^F$, $\B^F$, etc., are the subgroups of points fixed by $F$;

\medbreak
\item  $\theta$ is the automorphism of ${\mathbb R}\Phi$ 
induced by $\vartheta$, and, by abuse of notation, 
the automorphism of $W\leqslant  {\GL}({\mathbb R}\Phi)$ 
induced by conjugation by $\theta$;

\medbreak
\item   $\mathcal F$ is the following list of finite groups: 
 \begin{equation*}\mathcal F\coloneqq \{\mathrm{SL}_2(2), \mathrm{SL}_2(3), \mathrm{SU}_3(2),\mathrm{Sp}_4(2), G_2(2), \,^2\!B_2(2),\,\,^2\!G_2(3),\,^2\!F_4(2)\}.\end{equation*}
 \end{itemize}

Recall that a non-trivial finite group $H$ is \emph{quasi-simple} if it is perfect and $H/Z(H)$ is simple. Then $H/Z(H)$ is non-abelian, every proper normal subgroup of $H$ is central and every non-trivial homomorphic image of $H$ is quasi-simple, \cite[Chapter 6]{suzuki2}.  

If $\G_{sc}^F\not\in\mathcal F$ then it is quasi-simple and the restriction of the canonical projection $\pi$ to $\G_{sc}^F$ is a group homomorphism
$\pi: \G_{sc}^F \to \G_{ad}^F$ with image $[\G_{ad}^F,\G_{ad}^F]$ and kernel $\bZ:=Z(\G_{sc})^F=Z(\G_{sc}^F)$, so
\begin{align}\label{eq:defn-G-finite-simple}
\Gb:=\G_{sc}^F/Z(\G_{sc})^F=[\G_{sc}^F,\G_{sc}^F]/Z(\G_{sc}^F)\simeq [\G_{ad}^F,\G_{ad}^F]
\end{align} is a  non-abelian finite simple group.

\medbreak
If $\Phi$ is classical, by abuse of notation we will sometimes 
indicate an element in $\G$ or $\G^F$ by a matrix with 
the understanding that it corresponds to the matrix through the usual isogeny.

\medbreak 
We shall say that a root subsystem in a doubly-laced root system 
is of type $\widetilde{A}_n$, respectively $A_n$, if it is of type $A$, 
rank $n$ and consists of short, respectively long roots.

\subsection{Semisimple conjugacy classes}\label{subsec:ss-cc}

\

The main focus of this paper is on semisimple conjugacy classes, i.e., 
conjugacy classes of semisimple elements. If $x \in \G$ is semisimple, 
then $\Oc^{\G}_{x}$
intersects the maximal torus $\T$, since all maximal tori are conjugated in $\G$.
However, for $x \in \G^F$ semisimple,  $\Oc^{\G^F}_{x}$ does not necessarily intersect
$\T^F$; since $x$ belongs to some maximal torus of $\G$, 
we are led to consider the $F$-stable maximal tori of $\G$. 
We summarize the necessary information:

\begin{itemize} [leftmargin=*]\renewcommand{\labelitemi}{$\circ$} 
\item It is known and easy to see that an $F$-stable maximal
torus of $\G$ is  a conjugate $g\T g^{-1}$ 
where $g\in \G$ is such that $g^{-1}F(g)\in N_{\G}(\T)$. 
Hence, we will say that the maximal tori in $\G^F$ are
the subgroups of $F$-fixed elements in $F$-stable maximal tori.

\medbreak
\item Two maximal tori $(g\T g^{-1})^F$ and $(h\T h^{-1})^F$ in $\G^F$ are $\G^F$-conjugate if and only if $g^{-1}F(g)$ and $h^{-1}F(h)$ represent $\theta$-conjugate elements in $W$, see \cite[Proposition 25.1]{malle-testerman}, where $\theta$-conjugation is given by $\sigma\cdot_\theta \tau=\sigma\tau\theta(\sigma)^{-1}$.
The ($\theta$-class of the) element $g^{-1}F(g)\in N_{\G}(\T)/\T=W$ is called the associated element in the Weyl group.

\medbreak
\item For $w\in W$ we denote by $\T_w^F$ a representative of the $\G^F$-conjugacy class of maximal tori in $\G^F$  with associated element $w$. That is, 
\begin{align*}
\T_w&= g\T g^{-1}, & &\text{for some $g$ such that}& \overline{g^{-1}F(g)} &= w  \in N_{\G}(\T)/\T.
\end{align*}

\medbreak
\item Let $n_w\in N_{\G^F}(\T)$, let $w\coloneqq n_w\T\in W$ and let $g\in \G^F$ be such that $g^{-1}F(g)=n_w$. We set 
\begin{align}\label{eq:def-Fw}
F_w\coloneqq {\rm Ad}(n_w)\circ F.
\end{align}
Evidently, $\T$ is $F_w$-stable. Then $\T_w^F=g\T^{F_w}g^{-1}$. We set 
\begin{align}\label{eq:def-Gw}
G_w \coloneqq &\G^{F_w}\simeq \G^F &&\text{so}&\bZ &=Z(\G^F)=Z(\G)^F=Z(\G^{F_w}).
\end{align}
\end{itemize}

\medbreak

Table \ref{tab:center} collects information on the groups $\bZ$ for all groups $\G$ we are going to deal with. 

\begin{table}
\caption{Center of $\G_{\rm sc}^F$}\label{tab:center}
\begin{center}
\begin{tabular}{|c|c|c|c|}
\hline 
$\Phi$ & $n$  &  $q$  & $\bZ$ \\
\hline  \hline
$B_n$ &any &even & $e$\\
\cline{2-4}
& any &odd & $\langle \alpha_n^\vee(-1)\rangle$\\
\hline
$D_n$ &any &even & $e$\\
\cline{2-4}
& even &odd & $\langle  \prod_{i=1}^n\alpha_i^\vee((-1)^i),\alpha_{n-1}^\vee(-1)\alpha_n^\vee(-1)  \rangle$\\
\cline{2-4}
& odd &$\equiv1\mod 4$ & $\langle   \prod_{i=1}^{n-1}\alpha_i^\vee((-1)^i)\alpha_{n-1}^\vee(\omega)\alpha_n^\vee(\omega)  \rangle$\\
&&& $\omega^4=e$, $\omega^2=-1$\\
\cline{3-4}
& odd &$\equiv3\mod 4$ & $\langle  \alpha_{n-1}^\vee(-1)\alpha_n^\vee(-1)  \rangle$\\
\hline
$E_6$ && $\equiv0,2\mod 3$ & $e$\\
\cline{3-4}
& &$\equiv1\mod 3$& $\langle \alpha^\vee_1(\omega)\alpha^\vee_3(\omega^2)\alpha^\vee_5(\omega)\alpha^\vee_6(\omega^2)\rangle$\\
&&&$\omega\neq e$, $\omega^3=e$\\
\hline
$E_7$ && even & $e$\\
\cline{3-4}
& &odd & $\langle \alpha_2^\vee(-1)\alpha_5^\vee(-1)\alpha_7^\vee(-1)\rangle$\\
\hline
$E_8,F_4,G_2$&&any&$e$\\
\hline
\end{tabular}
\end{center}
\end{table}

\medbreak
We next define cuspidal elements:

\medbreak
\begin{itemize}[leftmargin=*]
\item An element  $w\in W$ is \emph{cuspidal} (or elliptic) 
if no conjugate of $w$ lies in $\langle J\rangle$ for any $J\subsetneq  S$ 
such that $\theta(J)=J$. 
\end{itemize}

\medbreak
Equivalently, $w$ is cuspidal if the characteristic polynomial of $w\theta$ in the geometric representation of $W$ is not divisible by $X-1$, \cite[Exercise 3.16]{geck-pfeiffer}, \cite[Lemma 7.2]{He}, \cite[Proposition 2.2.1]{Ciubotaru-He}.

\medbreak
\begin{itemize}[leftmargin=*]
\item 
A semisimple element in $\G^F$ is \emph{cuspidal} if it lies only in 
$F$-stable maximal tori $\T_w$ with $w$ cuspidal. 
A cuspidal  element $s$ in $\G^F$ is regular, i.e., the 
identity connected component $\Cent_{\G}(s)^\circ$ of $\Cent_{\G}(s)$ 
is the unique maximal torus containing $s$, see \cite[Proposition 3.13]{ACG-VII}. 
\end{itemize}

Let  $s\in\G^F$ be a semisimple element. 
Then there exist $w\in W$ and  $g\in\G$ satisfying $g^{-1}F(g)\T=w\in N_{\G}(\T)/\T=W$
such that $s\in \T_w^F=g\T^{F_w}g^{-1}$. That is, there exist a finite torus
$\T^{F_w}$, $t\in \T^{F_w}$
and  $g\in\G$ such that
\begin{align}\label{eq:ss-rep-torus}
s &= gtg^{-1}.
\end{align}
Conjugation by $g$ induces a rack isomorphism, henceforth called  the \emph{standard} one, 
between $\Oc_t^{G_w}$ and $\Oc_s^{\G^F}$.

\medbreak
\begin{definition}
For an $F$-stable semisimple conjugacy class $\Oc_s^{\G^F}$, we set 
\begin{equation}\label{eq:W_O}
W[\Oc_s^{\G^F}] \coloneqq \{w\in W~|~\Oc_s^{\G^F}\cap \T_w^F\neq\emptyset\}=\{w\in W~|~\Oc_t^{\G}\cap \T^{F_w}\neq\emptyset\}.\end{equation} 
In words, the set $W[\Oc_s^{\G^F}]$ keeps the record of the maximal tori that intersect the 
conjugacy class $\Oc_s^{\G^F}$; it 
is a union of $\theta$-twisted conjugacy classes in $W$. 
\end{definition}

\begin{obs}\label{obs:not-independent}Let ${\boldsymbol{\pi}}\colon \G\to \G_1$ be an isogeny of simple algebraic groups, with 
$F$-stable kernel. We denote by $F$ the induced morphism on $\G_1$, so  $W[\Oc_{{\boldsymbol{\pi}}(s)}^{\G_1^F}]$ is also defined. Similar to \eqref{eq:W_O}, 
we define
\begin{equation*}
W[\Oc_{{\boldsymbol{\pi}}(s)}^{{\boldsymbol{\pi}}(\G^F)}] \coloneqq  \{w\in W~|~\Oc_{{{\boldsymbol{\pi}}(s)}}^{{\boldsymbol{\pi}}(\G^F)}\cap ({\boldsymbol{\pi}}(\T_w))^F\neq\emptyset \}. 
\end{equation*}
We have a chain of inclusions  
\begin{align}\label{eq:incl-wo}W[\Oc_s^{\G^F}]\subseteq 
W[\Oc_{{\boldsymbol{\pi}}(s)}^{\boldsymbol{\pi}(\G^F)}]
\subseteq W[\Oc_{{\boldsymbol{\pi}}(s)}^{\G_1^F}].\end{align}
Indeed, we have
\begin{align*}
{\boldsymbol{\pi}}(\Oc_s^{\G^F}\cap \T_w^F)\subseteq  {\boldsymbol{\pi}}(\Oc_s^{\G^F})\cap {{\boldsymbol{\pi}}}( \T_w^F)&= \Oc_{{\boldsymbol{\pi}}(s)}^{{\boldsymbol{\pi}}(\G^F)}\cap {\boldsymbol{\pi}}( \T_w^F)
\\
&\subseteq
\Oc_{{\boldsymbol{\pi}}(s)}^{\G_1^F}\cap ({\boldsymbol{\pi}}(\T_w))^F. 
\end{align*}

\begin{itemize} [leftmargin=*]\renewcommand{\labelitemi}{$\circ$} 
\item 
 The first inclusion in \eqref{eq:incl-wo} is an equality. 
 \end{itemize}

\medbreak
 For, if $\mathbf x\in   {{\boldsymbol{\pi}}}(\Oc_s^{\G^F})\cap {{\boldsymbol{\pi}}}( \T_w^F)$ then $\mathbf x={\boldsymbol{\pi}}(x)$ for some $ x\in \Oc_s^{\G^F}$ and $\mathbf x={\boldsymbol{\pi}}(x')$ for some $x'\in \T_w^F\cap Z(\G)^F x\subset \T_w^F\cap \bZ x$.  But then $x\in \T_w^F$. 

\medbreak
\begin{itemize} [leftmargin=*]\renewcommand{\labelitemi}{$\circ$} 
\item The second inclusion in \eqref{eq:incl-wo} need not be an equality because the restriction of ${\boldsymbol{\pi}}$ to $\G^F$ in general is not surjective on $\G_1^F$.
\end{itemize}
\end{obs}

\medbreak
\begin{exa}
Let  $\G=\SL_2(\ku)$, $\G_1=\PSL_2(\ku)=\PGL_2(\ku)$, $F=\Fr_q$ where $q$ is odd
and let $\T_1$ be the maximal torus of diagonal matrices in $\G$. Then $\G^F=\SL_2(q)$ and ${\boldsymbol{\pi}} (\G^F)=\PSL_2(q)\subsetneq \PGL_2(q)=\G_1^F$.  
If $a^q=-a$ for $a\in \ku^\times$ and $t\coloneqq \diag (a,a^{-1})$, 
then  $W[\Oc_s^{\G^F}]$ contains only the non-trivial element in $W$, but ${\boldsymbol{\pi}}(s)\in {\boldsymbol{\pi}}(\T_1)^F$, so $W[\Oc_{{\boldsymbol{\pi}}(s)}^{\G_1^F}]=W$. 
\end{exa} 

The racks $\Oc_t^{G_w}$ and $\Oc_s^{\G^F}$ are isomorphic (via the standard isomorphism).
For any $w\in W[\Oc_s^{\G^F}]$ we will work with the former. 

\subsection{Cuspidal classes in Weyl groups of type \texorpdfstring{$B_n$}{} and \texorpdfstring{$D_n$}{}}\label{sec:cuspidal-classical}

\

For later purposes, we recall the description of cuspidal conjugacy classes in 
Weyl groups of type $B_n$ and $D_n$ from \cite{geck-pfeiffer}.

\medbreak
\subsubsection{Cuspidal classes, type \texorpdfstring{$B_n$}{}}

Cuspidal conjugacy classes in Weyl groups of type $B_n$ are parametrized by the set of partitions $\lambda=(\lambda_1,\,\ldots,\,\lambda_r)$ of $n$, \cite[Proposition 3.4.6]{geck-pfeiffer}. More precisely, given a partition $\lambda$ we set 
\begin{align*}
\Lambda_1\coloneqq 0, &&\Lambda_{j}\coloneqq \sum_{i=1}^{j-1}\lambda_i, \textrm{ for }j\in\I_{2,r},&&\gamma_j\coloneqq \sum_{i=0}^{\Lambda_j}\alpha_{n-i}, \textrm{ for }j\in\I_{r}.
\end{align*} 
Consider 
\begin{align*}
 w_{\lambda_j}
 &=s_{\gamma_j}s_{n-\Lambda_j-1}s_{n-\Lambda_j-2}\cdots s_{n-\Lambda_{j}-(\lambda_j-1)},&
 j &\in\I_{r}.
\end{align*}
Observe that $w_{\lambda_j}$ is a Coxeter element in the Weyl group $W_{\Phi_{\lambda,j}}$ of the root subsystem $\Phi_{\lambda,j}$ of type $B_{\lambda_j}$ with base 
\begin{align*}\{\alpha_{n-\Lambda_{j}-\lambda_j+1},\,\alpha_{n-\Lambda_{j}-\lambda_j+2},\,\ldots,\, \alpha_{n-\Lambda_j-1},\,\gamma_j\}.\end{align*}
If we embed the Weyl group of type $B_{\lambda_j}$ into $\mathbb S_{2\lambda_j+1}$ as in \cite[Section 3.4]{BL}, Coxeter elements correspond to $2\lambda_j$-cycles. Hence 
\begin{align*}
\Cent_{W_{\Phi_{\lambda,j}}}(w_{\lambda,j})=\Cent_{\mathbb S_{2\lambda_j+1}}(w_{\lambda_j})\cap W_{\Phi_{\lambda,j}}= \langle w_{\lambda,j}\rangle \text{ and }  |w_{\lambda_j}|=2\lambda_j.
\end{align*}
In addition, if $i\neq j$ then $\Phi_{\lambda,j}\perp\Phi_{\lambda,i}$, whence $w_{\lambda_j}w_{\lambda_i}=w_{\lambda_i} w_{\lambda_j}$, so
\begin{align*}
|\Cent_{\langle J_{w_\lambda}\rangle}(w_\lambda)|\geq 2^r\prod_{i=1}^r\lambda_i.
\end{align*}
By \cite[Proposition 3.4.6]{geck-pfeiffer}, a representative of the 
cuspidal class corresponding to $\lambda$ is given by 
\begin{align}\label{eq:cuspidal-bn}
 w_{\lambda}=\prod_{j=1}^r w_{\lambda_j}.
\end{align}

Thus, the cuspidal classes in $\langle s_{n-l+1},\,\ldots,\,s_n \rangle$ 
where $1 \leq l \leq n$
are parametrized by partitions $\lambda$ of $l$ rather than $n$, and represented by the $w_\lambda$ described  above.

\medbreak
\subsubsection{Cuspidal classes, type \texorpdfstring{$D_n$}{}}
The classification of cuspidal conjugacy classes in a Weyl group $W_D$ of type $D_n$ is given in \cite[Proposition 3.4.11]{geck-pfeiffer} in terms of the inclusion of the root system $\Phi_D$ of type $D_n$ into the root system $\Phi_B$ of type $B_n$ as the set of long roots. 

\medbreak
We denote by $\alpha_n'$, for the remainder of this section only, the $n$-th simple root in $\Phi_D$,
to distinguish it from the $n$-th simple root $\alpha_n\in\Phi_B$, so $\alpha_n'=\alpha_{n-1}+2\alpha_n$. 
Through the inclusion of $W_D$ into the Weyl group $W_B$ of $\Phi_B$, the cuspidal elements in $W_D$ are precisely the cuspidal elements in $W_B$ that are contained in $W_D$. Retaining notation from \eqref{eq:cuspidal-bn},
the cuspidal elements in $W_D$ are the elements $w_\lambda$ corresponding to a partition $\lambda$ of $n$ with an even number $r=2c$ of terms. 

\medbreak
Observe that $w_{\lambda_i}\in W_B\setminus W_D$ for any $i\in\I_l$ but since $[W_B:W_D]=2$, there holds $w_{\lambda_i}w_{\lambda_j}\in W_D$ for any $i,j$.  
In detail, the cuspidal element in $W_D$ associated with $\lambda=(\lambda_1,\,\ldots,\,\lambda_{2c})$  is  
\begin{align*}
w_\lambda &= \prod_{i=1}^c(w_{\lambda_{2i-1}}w_{\lambda_{2i}}),
\end{align*}
where $w_{\lambda_{2i-1}}w_{\lambda_{2i}}$ lies in the Weyl group of the root subsystem $\Phi_{\lambda,2i-1,2i}$ of type $D_{\lambda_{2i-1}+\lambda_{2i}}$ consisting of the long roots in the root system 
\[(\mathbb Z\Phi_{\lambda,2i-1}+\mathbb Z\Phi_{\lambda,2i})\cap \Phi_B\] 
of type $B_{\lambda_{2i-1}+\lambda_{2i}}$.  The order of $w_{\lambda_{2i-1}}w_{\lambda_{2i}}$ is
\begin{align*}
\vert w_{\lambda_{2i-1}}w_{\lambda_{2i}} \vert &= 2 \textrm{lcm}(\lambda_{2i},\lambda_{2i-1}),
& i &\in \I_{c}.
\end{align*}

Observe that for any $\lambda_{2i-1}\geq 3$, the element $w_{\lambda_{2i-1}}w_{\lambda_{2i}}$ preserves the root system $\Phi'_{\lambda,{2i-1}}\coloneqq \Phi_{\lambda,2i-1}\cap\Phi_D$, which is of type $D_{\lambda_{2i-1}}$. For $i=1$, this is the root system with base $\{\alpha'_n,\alpha_{n-1},\,\ldots,\,\alpha_{n-\lambda_1+1}\}$. 

\medbreak
The centraliser of $w_\lambda$ in $W_D$ contains $w_{\lambda_i}w_{\lambda_j}$ for any $i,j \in  \I_{2c}$. If $\lambda_i=\lambda_{i+1}$, then it also contains the involution,
which swaps $w_{\lambda_{i}}$ and $w_{\lambda_{i+1}}$, given by
\[
\prod_{d=0}^{\lambda_i-1}s_{\alpha_{n-\Lambda_{i+1}-d}+\cdots+\alpha_{n-\Lambda_i-1-d}}.
\]

\medbreak
The discussion for type $D_n$ is still valid for type $D_3=A_3$. 
In this case, there is only one cuspidal class, 
namely the class containing  Coxeter elements, corresponding to the partition $(2,1)$.

\section{Detecting semisimple classes of type C}\label{sec:notation-M}

\subsection{Overview}\label{subsec-overview}

\ 

In this Section we assume that $\G=\G_{sc}$ is a 
simple, simply connected, algebraic group and that $F$ is a Frobenius or Steinberg endomorphism of $\G$. We assume that $\G_{sc}^F$ is not in the family $\Fc$.
Let $\Gb= \G_{sc}^F/Z(\G^F_{sc})$ as in \eqref{eq:defn-G-finite-simple}. 

\medbreak 
We are looking for a procedure that allows us to decide 
if a semisimple orbit $\Oc^{\Gb}_x$ of $\Gb$ is of type C, so
that any Nichols algebra over it is infinite-dimensional. 
Now $\Oc^{\Gb}_x$ intersects
at least one of the maximal tori $\T^F_w$, $w \in W$. Previously,
we have worked out the conjugacy classes intersecting
some special tori $\T^F_w$, for instance:

\medbreak
\begin{itemize}[leftmargin=*]
\item \cite[Theorem 1.1]{ACG-III}
If $\Gb = \PSL_{n+1}(q)$, $n \geq 2$, and $x$ is not irreducible, then $\Oc^{\Gb}_x$ collapses.
Here `$x$ is irreducible' is equivalent to `$\Oc^{\Gb}_x$ intersects only 
maximal tori $\T^F_w$, where $w \in W$ lies in the conjugacy class of a Coxeter element'.

\medbreak
\item \cite[Theorem 4.1]{ACG-VII}
If $\Gb \neq \PSL_{n+1}(q)$,  plus a minor technical assumption, 
and $\Oc^{\Gb}_x$ intersects the split torus $\T^F$, then it collapses.
\end{itemize}

\noindent
For more information on previous results about collapsing semisimple classes, see page \pageref{sec:typeC}.

\medbreak
We focus on the classes that intersect a given maximal torus. Concretely,

\begin{center}
\emph{we fix until further notice an element $w \in W$}.
\end{center}
 
Our objects of interest are the conjugacy classes $\Oc_{\pi(s)}^{\Gb}$,
where $s = gtg^{-1}$ as in \eqref{eq:ss-rep-torus} for $t \in \T^{F_w}\setminus e$.

\medbreak
In other words, let $\mathfrak O=\Oc_t^{G_w}$, where $G_w =\G^{F_w}$ as in \eqref{eq:def-Gw}
and $F_w$ is as in \eqref{eq:def-Fw}. Then $\Oc:=\pi(\mathfrak O)=\Oc_{\pi(t)}^{\pi(G_w)}$ is an orbit in 
\[G_w/Z(G_w)=G_w/\bZ\simeq \G^F/Z(\G)^F= \Gb.\]
Observe that $\Oc\simeq \Oc_{\pi(s)}^{\Gb}$. 
For simplicity, we will deal with $\mathfrak O$ and $\Oc$ 
rather than $\Oc_s^{\G^F}$ and $\Oc_{\pi(s)}^{\Gb}$.

\medbreak
In this Section we provide sufficient conditions for semisimple classes 
to be of type C by means of the following notion.
Recall that the maximal torus $\T$ of $\G$ is both $F$-stable and $F_w$-stable.

\begin{definition} A connected reductive subgroup of  $\G$ is \emph{$w$-suitable} if it
is  $F_w$-stable  and contains  $\T$.
\end{definition}

Our goal is to produce, from a $w$-suitable subgroup $\Me$, a group $H$ 
as in the definition of type C, cf. page \pageref{item:typeC-group}.
Now, there are many $w$-suitable subgroups  (e.g., $\G$ itself), 
but to make this work, we need some strong conditions. 
Indeed, we prove that if a $w$-suitable subgroup $\Me$ and $t \in \T^{F_w}\setminus e$
satisfy the conditions \ref{S}, \ref{R} and \ref{W} below, then 
$\Oc \simeq \Oc_{\pi(s)}^{\Gb}$ (see above)
is of type C, cf. Theorem \ref{thm:potente-revised} in \S \ref{subsec:key-theorem}. 

We then analyze these conditions in \S\ref{subsec:notlist} and 
describe in \S \ref{subsec:applications} some situations
where Theorem \ref{thm:potente-revised} applies.

\subsection{The key Theorem}\label{subsec:key-theorem}

\

We need the following notation. 
Let $\Me$ be a $w$-suitable subgroup of $\G$.

\medbreak
\begin{itemize} [leftmargin=*]\renewcommand{\labelitemi}{$\circ$} 

\item $\Phi_{\Me}$ is the $w\theta$-stable root subsystem of $\Phi$ consisting of  the roots of the semisimple group $[\Me,\Me]$, and $\Phi_{\Me}^+\coloneqq \Phi^+\cap\Phi_{\Me}$, so 
\begin{align*}\Me = \langle \T,\U_{\pm\gamma}:\,\gamma\in \Phi_{\Me}\rangle,&&[\Me,\Me]=\langle\U_{\pm\gamma}:\,\gamma\in \Phi_{\Me}\rangle;\end{align*}
\end{itemize}

That is, $w$-suitable subgroups of $\G$ are in bijective correspondence 
with $w\theta$-stable root subsystems of $\Phi$.

\medbreak
\begin{itemize} [leftmargin=*]\renewcommand{\labelitemi}{$\circ$} 
\item $\Delta_{\Me}$ is the base of $\Phi_{\Me}$ corresponding to the positive system $\Phi^+_{\Me}$;

\medbreak
\item $W_{\Me} \coloneq N_{\M}(\T)/\T$ is the Weyl group of $\M$, hence also of $[\Me,\Me]$;

\medbreak
\item $M\coloneqq [\Me,\Me]^{F_w}$, a finite group of Lie type;

\end{itemize}

\begin{obs}\label{obs:normalize}
Since $\T$ normalizes $\Me$, it normalizes $[\Me,\Me]$ and so $\T^{F_w}$ normalizes $M$ and $[M,M]$. 
\end{obs}

\begin{theorem}\label{thm:potente-revised}Let $\G=\G_{sc}$. 
	Let $\Me$ be a $w$-suitable subgroup of $\G$
	and $t \in \T^{F_w}\setminus e$. 
	Assume that the following conditions hold:
\begin{align}&\pi([M,M])\textrm{ is quasi-simple  (quasi-simplicity);}
\label{S}\tag*{\fbox{\footnotesize Q}}\\
&[t,[M,M]_p]\neq e \textrm{ (non-centrality); }\label{R}\tag*{\fbox{\footnotesize N}}
\\
&\Cent_W(w\theta)\cdot t\setminus \left(\left(\bZ \Cent_{W_{\Me}}(w\theta)\cdot t\right)\cup \Cent_{\T}([M,M]_p)\right)\neq \emptyset \label{W}\tag*{\fbox{\footnotesize W}}
\\
&\textrm{ (orbit condition).} \notag
\end{align}
	Then, $\Oc$ is of type C. 
 \end{theorem} 
 
 Here \ref{S} stands for quasi-simplicity and \ref{R} for non-central, while
 \ref{W} refers to the fact that we are reducing to a problem on the Weyl group.

\begin{proof}
This is a consequence of Theorem \ref{thm:potente-abstracta} applied to the following data:
	
	\begin{itemize}
		\item  $\gp = \pi(G_w)\simeq G_w/\bZ$, 
		
		\medbreak
		\item $\subgp = \pi([M,M])$, 
		
		\medbreak
		\item $\normsubgp = \pi(\T^{F_w})$ and
		
		\medbreak
		\item  $\letra= \pi(t)$.
	\end{itemize}
	
We already observed that $\pi(\T^{F_w})$ normalizes $\pi([M,M])$.

We need to verify that the hypotheses of  
Theorem \ref{thm:potente-abstracta} are valid.
Hypothesis \ref{item:quasi-simple} is just \ref{S} while hypothesis \ref{item:non-central}
follows from \ref{R}. Indeed, if $[\pi(t),\pi([M,M])]=e$, then 
$t\trid x\in [M,M]_p\cap \bZ x=\{x\}$ for any $x\in [M,M]_p$, contradicting \ref{R}. 

It remains to check that \ref{item:dos-orbitas-hyp} holds, i.e., that
\begin{align*}
\left(\Oc_{\letra}^{\gp} \cap \normsubgp \right)
\,\backslash \left(\left(\Oc_{\letra}^{\subgp}  \cap \normsubgp\right) \cup \Cent_{\gp} (\subgp) \right) \neq \emptyset.
\end{align*}
We show that it follows from \ref{W}.

First we claim that
\begin{align*}
&\Oc_t^{G_w}\cap\T^{F_w} = \Oc^{N_{G_w}(\T)}_t =\Cent_W(w\theta)\cdot t.
\end{align*}
The first equality follows combining  \cite[Exercise 30.13]{malle-testerman} with the Bruhat decomposition in $G_w=\G^{F_w}$ \cite[Theorem 24.1]{malle-testerman} as we now show. Assume that $t_1=x_1\trid t\in \T^{F_w}$ for some $x_1\in G_w$.
Uniqueness of the factor in $N_{G_w}(\T)$ in the Bruhat decomposition applied to $x_1$ and to $x_1t=t_1 x_1\in G_w$
ensures that $x_1$ may be replaced by an element in $N_{G_w}(\T)$.
The second equality follows from the isomorphism $\Cent_W(w\theta)\simeq N_{G_w}(\T)/\T^{F_w}$,  \cite[Proposition 25.3 (a)]{malle-testerman}.

In a similar fashion we prove that 
\begin{align}\label{eq:orbit-W}
\Oc_t^{[M,M]}\cap \T^{F_w} &\subseteq\Oc_t^M\cap\T^{F_w}=\Oc_t^{N_{M}(\T)}=\Cent_{W_{\Me}}(w\theta)\cdot t.
\end{align}
Hence, \ref{W} and the inclusion $\T^{F_w}\leq G_w$ give
\begin{align*}\left(\Oc_t^{G_w}\cap\T^{F_w}\right)\setminus \left(\bZ\left( \Oc_t^{[M,M]}\cap \T^{F_w}\right)\cup \Cent_{G_w}([M,M]_p)\right)
\neq \emptyset.\end{align*} 
Let $y\in \Oc_t^{G_w}\cap\T^{F_w}$ with $y\not\in \bZ\left(\Oc_t^{[M,M]}\cap \T^{F_w}\right)$ and $y\not\in \Cent_{G_w}([M,M]_p)$. 
We claim that
\begin{align*}
\pi(y)\in \left(\Oc_{\letra}^{\gp} \cap \normsubgp \right)
\,\backslash \left(\left(\Oc_{\letra}^{\subgp}  \cap \normsubgp\right) \cup \Cent_{\gp} (\subgp) \right).
\end{align*}
The inclusion $\pi(y)\in \Oc_{\letra}^{\gp}\cap\normsubgp$ is immediate. 
If $\pi(y)\in \Oc_{\letra}^{\subgp}\cap \normsubgp=\pi( \Oc_t^{[M,M]})\cap \pi(\T^{F_w})$, the inclusion $\bZ\subset \T^{F_w}$ gives  $y\in \bZ(\Oc_t^{[M,M]}\cap \T^{F_w})$, a contradiction. Finally, if $[\pi(y),\subgp]=e$, then $[\pi(y),\subgp_p]=e$, contradicting $y\not\in \Cent_{G_w}([M,M]_p)$. 
\end{proof}

\subsection{Implementation of Theorem \ref{thm:potente-revised}} \label{subsec:notlist}

\

To implement the strategy of Step \ref{step-one}, page \pageref{step-one},
we first verify whether, for a given $t\in \T^{F_w}\subset\G_{sc}$ we can exhibit 
a $w$-suitable group $\Me$, corresponding to a $w$-stable root subsystem $\Phi_{\Me}$ of $\Phi$, satisfying conditions \ref{S}, \ref{R} and \ref{W}, 
so that Theorem \ref{thm:potente-revised} applies.

\medbreak
We provide next a few arguments that are useful to produce such subgroups 
and to verify the desired conditions, retaining notation from  
\S\, \ref{subsec:key-theorem}.

\subsubsection{Quasi-simplicity condition}\label{subsubsec:notlist}

\

The simplest situation where to find groups that satisfy \ref{S} is when $\Phi_{\Me}$ is irreducible.

\begin{lema}\label{lem:perfect-S} If $[\Me,\Me]$ is simple and $M$ is perfect, then \ref{S} holds. In particular, \ref{S} holds if $\Me$ is a Levi subgroup with $w\theta$-stable irreducible root system $\Phi_{\Me}$ satisfying
\begin{align}\label{eq:not-list}&(\Phi_{\Me},q)\not\in\{(A_1,2),\, (A_1,3),\,(B_2,2)\}, \mbox{ and }\\
\nonumber&\mbox{if }(\Phi_{\Me},q)=(A_2,2), \mbox{ then }
w\theta \mbox{ acts on }\Phi_{\Me} \mbox{ as an element of $W_{\Me}$.}\end{align}
\end{lema}
\begin{proof} By \cite[Corollary 24.14]{malle-testerman} the group $M$ is  quasi-simple, so its homomorphic image $\pi(M)=\pi([M,M])$ is quasi-simple. 
\end{proof}

Observe that \eqref{eq:not-list} requires $\Phi_{\Me}$  and $q$ not to be too small, but a too large $\Phi_{\Me}$  may cause \ref{W} to fail.

\medbreak

The condition on $M$ to be perfect is not necessary, as we see now.
\begin{lema}\label{lem:S-b2}If $q=2$ and $\Phi_{\Me}$ is of type $B_2$, then \ref{S} holds. 
\end{lema}
\begin{proof}
In this case, $\Phi$ is doubly-laced, $\theta=\id$ and $\G=\G_{sc}=\G_{ad}$ as abstract groups, so $Z(\G)=e$ and $\pi$ is the identity. In addition, $[\Me,\Me]\simeq \Sp_4(\ku)\simeq\PSp_4(\ku)$ as abstract groups, whence $M\simeq \Sp_4(2)\simeq\mathbb S_6$ and $\pi([M,M])=[M,M]\simeq\mathbb A_6$, a non-abelian simple group. 
\end{proof}

\medbreak
The following Lemma addresses the case in which there is no natural candidate for a root system $\Phi_{\Me}$ satisfying the hypothesis of Lemma \ref{lem:perfect-S}.

\begin{lema}\label{lem:cyclically-permuted}Let $d\geq 1$ and let $\Phi_1,\,\ldots,\,\Phi_d$ be the irreducible components of $\Phi_{\Me}$ and assume $w\Phi_i=\Phi_{i+1\mod d}$ for all $i\in\I_{d}$. Let $\widetilde{\Me}$ be the simply connected cover of $[\Me,\Me]$, with decomposition into simple factors $\widetilde{\Me}=\prod_{i=1}^d\widetilde{\Me}_i$ and let $\mathbf F_w$ be the lift of $F_w$ to $\widetilde{\Me}$. If 
$\widetilde{\Me}_1^{\mathbf F^d_w}$ is quasi-simple, then \ref{S} holds.
\end{lema}

\begin{proof}The existence of $\mathbf F_w$ is guaranteed by \cite[Proposition 9.18]{malle-testerman}, and \cite[Exercises 30.2, 30.4]{malle-testerman} give $\widetilde{\Me}\simeq \widetilde{\Me}_1^{\mathbf F_w^d}$. 

Let $Z$ be the kernel of the projection $\widetilde{\Me}\to [\Me,\Me]$. By construction, it is a central $\mathbf F_w$-stable subgroup of $\widetilde{\Me}$ and $[\Me,\Me]\simeq \prod_{i=1}^d\widetilde{\Me}_i/Z$. We claim that $[M,M]\simeq \widetilde{\Me}_1^{\mathbf F_w^d}/Z^{\mathbf F_w}$.

On the one hand, the inclusion $\widetilde{\Me}^{\mathbf F_w}\hookrightarrow\widetilde{\Me}$ gives rise to an inclusion $\widetilde{\Me}^{\mathbf F_w}/Z^{\mathbf F_w}\hookrightarrow \widetilde{\Me}/Z\simeq[\Me,\Me]$, whence an inclusion
$\widetilde{\Me}_1^{\mathbf F^d_w}/Z^{\mathbf F_w}\hookrightarrow M$, hence a canonical inclusion  $\iota\coloneqq\widetilde{\Me}_1^{\mathbf F^d_w}/Z^{\mathbf F_w}\hookrightarrow [M,M]$ by quasi-simplicity of ${\Me}_1^{\mathbf F^d_w}$. 

On the other hand, let $x,\,y\in M$. Then, there are $x_i, y_i\in \widetilde{\Me}_i$ for $i\in\I_{d}$ such that
\begin{align*}
&x=(x_1,\,\ldots,\,x_d)Z,&& y:=(y_1,\,\ldots,\,y_d)Z,&&\\
&\mathbf F_w(x_i)\in  x_{i+1\,{\rm mod}\, d}Z,&&\mathbf F_w(y_i)\in y_{i+1\,{\rm mod}\, d}Z&&\mbox{ for any  }i\in\I_{d}.
\end{align*}
Passing to  their commutator gives
\begin{align*}[x,y]=([x_1,y_1],\,\ldots,\,[x_d,y_d])Z, &&\mathbf F_w([x_i,y_i])=  [x_{i+1\,{\rm mod}\, d},y_{i+1\,{\rm mod}\, d}].\end{align*} So, any generator of $[M,M]$ lies in a coset for $Z$ represented by an element  in 
$[\widetilde \Me,\widetilde\Me]^{\mathbf F_w}=\widetilde \Me^{\mathbf F_w}\simeq \widetilde\Me_1^{\mathbf F^d_w}$, showing that $\iota$ is also onto.

Quasi-simplicity of $\Me_1^{\mathbf F^d_w}$ implies quasi-simplicity of $[M,M]$ and of all its homomorphic images, giving \ref{S}.
\end{proof}

The case $d=1$, i.e., when $\Phi_{\Me}$ is irreducible, can also be deduced from \cite[Proposition 24.21, Theorem 24.17]{malle-testerman}.

\subsubsection{Non-centrality condition}\label{subsec:non-centrality}

\

Let $x\in\T^{F_w}$. To estimate $\Cent_{[M,M]}(x)$ or $\Cent_M(x)$ we use knowledge of $\Cent_{[\Me,\Me]}(x)$.
Even if $x$ might not necessarily lie in $[\Me,\Me]$, conjugation by $x$ is a semisimple endomorphism of $[\Me,\Me]$, hence, if the latter is simply-connected, \cite[Theorem 8.1]{steinberg-endo} guarantees that $\Cent_{[\Me,\Me]}(x)$ is connected. 

\medbreak
The simple connectedness of $[\Me,\Me]$ always holds if $\Me$ is a Levi subgroup of a parabolic subgroup of $\G$, or if $\Phi_{\Me}$ consists of all the long roots in $\Phi$, \cite[Corollary II.5.4]{sp-st}. 

\medbreak
If $\Cent_{[\Me,\Me]}(x)$ is connected, then by \cite[Theorem 2.2]{hu-cc} we have
\begin{equation}\label{eq:centraliser}\Cent_{[\Me,\Me]}(x)=\langle \T\cap[\Me,\Me],\U_{\gamma}~|~\gamma\in \Phi_{\Me},\;\gamma(x)=1\rangle.\end{equation}

\begin{lema}\label{lem:non-centrality}Assume that $[\Me,\Me]$ is simply-connected and that $n_w$ acts on $\Me$ as an inner automorphism. If $\gamma(x)\neq1$ for some $\gamma\in\Phi_{\Me}$, then $[M_p,x]\neq e$.
\end{lema}
\begin{proof}
Under the above assumption $\Cent_{[\Me,\Me]}(x)$ is an $F_w$-stable reductive group with root system $\Phi_x=\{\gamma\in\Phi_{\Me}~|~\gamma(x)=1\}$, and $\Phi_x\subsetneq\Phi_{\Me}$. 
By \cite[Corollary 24.6]{malle-testerman}, any Sylow $p$-subgroup of $M$ has order $q^{|\Phi_{\Me}^+|}$ and any Sylow $p$-subgroup of $\Cent_M(x)=\Cent_{[\Me,\Me]}(x)^{F_w}\simeq \Cent_{[\Me,\Me]}(x)^{F}$ has order $q^{|\Phi_x^+|}$. Hence, there are $p$-elements in $M$ that are not contained in $\Cent_M(x)$.  \end{proof}

\begin{obs}\label{obs:non-commute}
\begin{enumerate}[leftmargin=*,label=\rm{(\roman*)}]
\item If $[M,M]=\langle [M,M]_p\rangle$, then $[t,[M,M]_p]\neq e$ if and only if  $[t,[M,M]]\neq e$. If $[\Me,\Me]$ is simply connected or if $[M,M]$ is simple,
then $[M,M]=\langle [M,M]_p\rangle$ always holds, \cite[Theorem 24.15]{malle-testerman}.

\medbreak
\item If $[\Me,\Me]$ is simply-connected, $[M,M]$ is non-abelian,  and $\gamma(t)\neq1$ for any $\gamma\in \Phi_{\Me}$, then $\Cent_{[\Me,\Me]}(t)=\T\cap[\Me,\Me]$, that is, $t$ is regular in $[\Me,\Me]$, so $[M,M]_p\cap \Cent_{[M,M]}(t)=e$.  
\end{enumerate}
\end{obs}

We end this subsection by describing a construction of an $F_w$-stable standard Levi subgroup  satisfying a stronger condition than \ref{R}.

\medbreak
Any $\sigma\in\WO$ is cuspidal in some $\langle J_\sigma\rangle$ for some $\theta$-stable $J_\sigma\subseteq S$.  For $\sigma \in \WO$ we set $E_{J_\sigma}\coloneqq \mathbb R \Delta_{J_\sigma}$. Then, $\sigma$ acts on $E_{J_\sigma}$ and by cuspidality, ${\rm rk}(\id-\sigma\theta)=\dim E_{J_\sigma}=|J_\sigma|$. 

\medbreak

An element $\sigma $ in $\WO$ is called \emph{minimal} if $|J_\sigma|\leq |J_{\sigma'}|$ for any $\sigma'\in \WO$. In particular, $J_\sigma$ is minimal with respect to inclusion. If $\mathfrak O$ is cuspidal then $J_\sigma=S$, if $\id\in \WO$, then $\sigma=\id$ and $J_\sigma= \emptyset$. Conversely, if  $J_\sigma=S$ for some \emph{minimal} $\sigma\in\WO$, then $\mathfrak O$ is cuspidal.

\begin{lema}\label{lem:Jw}
Assume $\id\not\in  \WO$ and let $w\in \WO$ be minimal. Then, 
$\gamma(t)\neq 1$ for any $\gamma\in \Phi_{J_w}$. Thus, $\Cent_{\Me}(t)=\T$ for any $\Me\subseteq \Le_{J_w}$.
\end{lema}
\noindent \emph{Proof.} 
First of all we borrow an argument from \cite[Lemma 2]{carter-cc} to prove that $s_\alpha w$ is not cuspidal in $\langle J_w\rangle$ for any $\alpha\in\Phi_{J_w}=\Phi\cap E_{J_w}$.

Since $w$ is cuspidal in $\langle J_w\rangle$, the endomorphism $(\id-w\theta)$ has maximal rank and so there is $v\in E_{J_w}$ such that $(\id-w\theta)v=\alpha$, that is,  $w\theta v=v-\alpha$. Then, by orthogonality of $w\in W$ and $\theta$ we have 
\begin{align*}(v,v)=(w\theta v,w\theta v)=(v-\alpha,v-\alpha)=(v,v)-2(v,\alpha)+(\alpha,\alpha).\end{align*}
Therefore, $s_\alpha v=v-2\frac{(v,\alpha)}{(\alpha,\alpha)}\alpha=v-\alpha$. But then, $v\in{\rm Ker}(\id-s_\alpha w\theta)$, giving the claim.

Now, if $\alpha(t)=1$ for some $\alpha\in \Phi_{J_w}$, then, $s_\alpha w\theta\cdot t^q=s_\alpha \cdot t=t$, whence $t\in \T^{F_{s_\alpha w}}$. Since $s_\alpha w$ is not cuspidal in $J_w$, there is a conjugate of $s_\alpha w$ that is contained in $\langle J\rangle\subsetneq \langle J_w\rangle$, contradicting minimality of $w$.

\medbreak

The last statement follows from \eqref{eq:centraliser} because 
\begin{align*}
\T\subseteq \Cent_{\Me}(t)\subseteq \Cent_{\Le_{J_w}}(t)
&=\T\Cent_{[\Le_{J_w},\Le_{J_w}]}(t)=\T\Cent_{[\Le_{J_w},\Le_{J_w}]}(t)^\circ
\\
&=\langle \T, \U_{\pm\gamma},\,\gamma\in\Phi_{J_w},\,\gamma(t)=1\rangle=\T. \qed
\end{align*}

When dealing with Chevalley groups, we will often take $\Me\subseteq\Le_{J_w}$ for a minimal $w$  in $\WO$.

\subsubsection{Orbit condition}\label{subsec:orbit-condition}

\

Condition \ref{W} holds if there exists  $\sigma \in \Cent_W(w\theta)$ satisfying two requirements:
\begin{align}&\sigma\cdot t\not\in \bZ \Cent_{W_{\Me}}(w\theta )\cdot t,&\textrm{ (disjoint orbit condition),}\label{W1}\tag*{\fbox{\footnotesize W1}}\\
&\sigma \cdot t\not\in \Cent_{\T}([M,M]_p)&\textrm{ (non centrality condition).}
\label{W2}\tag*{\fbox{\footnotesize W2}}\end{align} 

\begin{obs}\label{obs:M-in-LJ} If $t$ satisfies \ref{R}, then \ref{W2} holds for any $\sigma$ normalising $\Phi_{\Me}$. In particular, \ref{W2} holds if $\Me\subsetneq\Le_{J_w}$ and $\sigma\in \langle J_w\rangle$.
\end{obs}

\begin{lema}\label{lem:index}If $\Phi_{\Me}\subset \Phi_{J_w}$ and $[\Cent_{\langle J_w\rangle}(w\theta):\Cent_{W_{\Me}}(w\theta)]>|\bZ|$, then \ref{W} holds. 
\end{lema}
\begin{proof}
Lemma \ref{lem:Jw} gives 
\begin{align}\label{eq:reg}&\Cent_{\Me}(\sigma\cdot t)\subseteq \Cent_{\Le_{J_w}}(\sigma\cdot t)\subseteq \sigma \cdot \T=\T, &&\forall\sigma\in \langle J_w\rangle.
\end{align}
Setting $\sigma=\id$ gives $\Cent_{\Me}(t)\cap N_{\G}(\T)\subseteq \Cent_{\Le_{J_w}}(t)\cap N_{\G}(\T)=\T$, whence
\begin{align*}
&|\Cent_{W_{\Me}}(w)\cdot t|=|\Cent_{W_{\Me}}(w)|,&&|\Cent_{W_{\langle J_w\rangle}}(w)\cdot t|=|\Cent_{W_{\langle J_w\rangle}}(w)|.
\end{align*} 
Since $[\Cent_{\langle J_w\rangle}(w\theta):\Cent_{W_{\Me}}(w\theta)]>|\bZ|$ we have
\begin{align*}|C_{W_{\langle J_w\rangle}}(w\theta)\cdot t|>|\bZ||C_{W_{\Me}}(w\theta)\cdot t|\end{align*}
ensuring \ref{W1} for some $\sigma\in\langle J_w\rangle$, so \ref{W2} follows from \eqref{eq:reg}.  \end{proof}

 In general, our strategy to deal with \ref{W1} is to look for one or more candidates for $\sigma$ and look at the conditions under which \ref{W1} fails. This usually leads to strong restrictions on $t$ that either contradict \ref{R}, or force $t$ to lie in a subgroup of Lie type that we have already dealt with.

\medbreak
 
When the longest element $w_0$ in $W$ acts as $-\id$, it is a natural candidate for $\sigma$.

\begin{obs}\label{obs:w0}Assume that  $w_0$ in $W$ acts as $-\id$. Then $w_0\cdot t=t^{-1}$ and $w_0\in \Cent_W(w\theta)$ for any $w\in W$.
\begin{enumerate}[leftmargin=*,label=\rm{(\roman*)}]
\item\label{item:w0-uno} If condition \ref{R} holds, then $[[M,M]_p,t^{-1}]\neq e$ because $[[M,M]_p,t]\neq e$, so \ref{W2} holds for $\sigma=w_0$. Thus, in order to verify \ref{W} it is enough to show that $t^{-1}\not\in \bZ \Cent_{W_{\Me}}(w\theta)\cdot t$. 

\medbreak
\item\label{item:w0-due} Assume that $\theta=\id$, i.e., $\G^F$ is Chevalley, and that $\id\not\in \WO$, and let  $w\in \WO$ be minimal. If $w_0\cdot t=zt$ for some $z\in \bZ$, then $w^2=\id$. 
\end{enumerate}

Indeed, if $w_0\cdot t= t^{-1}=zt$, then $z^2=e$, so $t^4=e$. Since $\id\not\in\WO$, then necessarily $|z|=2$ and $q\equiv 3\mod4$. But then, $F(t)=t^q=t^{-1}$, so $t=F_w(t)=w t^{-1}$ and Lemma \ref{lem:Jw} forces $w^2=\id$. 
\end{obs}

The following observation is useful in verifying \ref{W1}.
\begin{obs}\label{obs:coefficient}
If $\Me\subseteq\Le_J$ for some $J\subsetneq S$, then since $s_\alpha\cdot t\in t\alpha^\vee(\ku^\times)$ for any $\alpha\in\Delta$, it follows that $W_{\Me}(w)\cdot t\subseteq t\prod_{\alpha\in\Delta_J}\alpha^\vee(\ku^\times)$. So for any $\beta\in\Delta\setminus\Delta_J$, the coefficient of $\beta^\vee$ in the expression of  any element of $\Cent_{W_{\Me}}(w)\cdot t$ coincides with the coefficient of  $\beta^\vee$ in the expression of  $t$. In particular, if the coefficient of $\beta^\vee$ in $z\in \bZ$ is non-trivial for some $\beta\in\Delta\setminus \Delta_J$, then $\sigma\cdot t\not\in z \Cent_{\Me}(w)\cdot t$. 
\end{obs}

\subsection{{Some situations where Theorem \ref{thm:potente-revised} applies}}\label{subsec:applications}

\

Here we have compiled several situations where \ref{S}, \ref{R}, and \ref{W} hold, and therefore Theorem \ref{thm:potente-revised}  applies. This facilitates case-by-case analysis in the following section.

\begin{lema}\label{lem:levi}Let $t\in \T^{F_w}\setminus \bZ$. Assume there is $J\subsetneq S$ such that
\begin{enumerate}[leftmargin=*,label=\rm{(\roman*)}]
\item\label{item:sub} $\Delta_J=\Delta_{J_1}\coprod \Delta_{J_2}$ where $\emptyset \neq\Delta_{J_i}\subseteq \Delta$ and $w(\Phi_{J_i})=\Phi_{J_i}$, for $i=1,2$;

\medbreak
\item\label{item:perp} $\Delta_{J_1}$ is irreducible and $\Delta_{J_1}\perp\Delta_{J_2}$; 

\medbreak
\item\label{item:reg} $\gamma(t)\neq 1$ for any $\gamma\in \Phi_{J}$;

\medbreak
\item\label{item:not-list} $(\Phi_{J_1},q)$ satisfies the conditions in \eqref{eq:not-list} for $F_w$;

\medbreak
\item $|\bZ|< |\Cent_{\langle J_2\rangle}(w\theta)|$.
\end{enumerate}
Then $\Oc$ is of type C.
\end{lema}
\begin{proof}
We show that \ref{S}, \ref{R} and \ref{W} hold for the simple group $\Me=\Le_{J_1}$ in $\G=\G_{sc}$. By \cite[Corollary II.5.4]{sp-st}  we have $[\Me,\Me]=[\Me,\Me]_{sc}$, a simple algebraic group so \ref{item:not-list} ensures that $M$ is perfect, giving \ref{S} in virtue of Lemma \ref{lem:perfect-S}.

\medbreak
By \ref{item:reg} and \eqref{eq:centraliser} we obtain
$\Cent_{\Me}(t)=\Cent_{\Le_J}(t)=\T \Cent_{[\Le_J,\Le_J]}(t)=\T$, so \ref{R} follows from $[M,M]_p\cap \Cent_{M}(t)=e$.

\medbreak
We verify  \ref{W}. If $\tau\cdot t=\sigma\cdot t$ for some $\tau\in \Cent_{\langle J_2\rangle}(w\theta)\cdot t$ and some $\sigma\in \Cent_{\langle J_1\rangle}(w\theta)\cdot t$, then for any representative $\dot{\sigma}$ of $\sigma$, respectively  $\dot{\tau}$ of  $\tau$, in $N_{\G}(\T)$, there holds $\dot{\sigma}^{-1}\dot{\tau}\in \Cent_{\Le_J}(t)=\T$, so $\sigma=\tau$, which is possible only if $\sigma=\tau=\id$. Hence, 
\begin{equation*}|\Cent_{\langle J\rangle}(w\theta)\cdot t|=|\Cent_{\langle J_2\rangle}(w\theta)||\Cent_{\langle J_1\rangle}(w\theta)\cdot t|>|\bZ \Cent_{\langle J_1\rangle}(w\theta)\cdot t|.\end{equation*}

In addition, conjugation by $\dot{\tau}$ preserves $\Le_J$ and $\T$, so $\Cent_{\Le_J}(\sigma\cdot t)=\Cent_{\Le_J}(t)=\T$, whilst $M_p\not\subset \T$, confirming  \ref{W}.
\end{proof}

\begin{cor}\label{cor:cuspidal-Ze}
Assume that $\bZ=e$ and that $w$ is cuspidal. If there is an $\Me$ satisfying \ref{S} and 
\begin{align}\label{eq:cuspidal-Ze}
\Cent_{W_{\Me}}(w)\subsetneq \Cent_W(w)
\end{align}
then $\Oc$ is of type C.
\end{cor}
\begin{proof}Since $w$ is cuspidal, $J_w=S$ so $\Cent_{\G}(\sigma\cdot t)=\T$ for any $\sigma\in W$ by Lemma \ref{lem:Jw}, (see also \cite[Proposition 3.13]{ACG-VII}). Therefore \ref{R} and \ref{W2} are immediate and $\Cent_{W_{\Me}}(w)\subsetneq \Cent_W(w)$ implies \ref{W1} by regularity of $t$. 
\end{proof}

\begin{lema}\label{lem:A-odd}Assume that $F$ is a Chevalley endomorphism, and that  $w_0\in W$ acts as $-\id$ on $\Phi$. 
Assume in addition that $\id\not\in W[\mathfrak O]$, and let $w$ be minimal in $\WO$. If $\Delta_{J_w}$ has a component $\Delta_{J}$, possibly equal to $\Delta_{J_w}$, that is of type $A_{2l}$ for some $l\geq 1$, then $\Oc$ is of type C. 
\end{lema}
\begin{proof}We verify  \ref{S}, \ref{R} and \ref{W} for $\Me\coloneqq \Le_{J}=\langle \T, \U_\gamma, \gamma\in\Phi_{J}\rangle$. First of all $[\Me,\Me]\simeq \SL_{2l+1}(\ku)$ because $\G$ is simply-connected and $\Me$ is the Levi subgroup of a parabolic subgroup of $\G$. Since $w\in\langle J\rangle$, the endomorphism $F_w$ acts on $\Le_J$ and $[\Le_J,\Le_J]$ as a Chevalley endomorphism, so Lemma \ref{lem:perfect-S} applies giving \ref{S}. Lemma \ref{lem:Jw} guarantees that $\Cent_{\Me}(t)=\Cent_{\Le_{J_w}}(t)=\T$, so $[M_p,t]\neq\emptyset$, giving \ref{R}.  

\medbreak
To verify \ref{W}, we use Remark \ref{obs:w0} \ref{item:w0-uno} and show that $w_0\cdot t=t^{-1}$ does not lie in  $\bZ \Cent_{W_{\Me}}(w)\cdot t$. Observe that $w$ is cuspidal in $\langle J_w\rangle$ by construction, so the component $w_J$ of $w$ in $\langle J\rangle$ is a Coxeter element in $\mathbb S_{2l+1}$, that is, a $(2l+1)$-cycle. Then, $\Cent_{W_{\Me}}(w)=\langle w_J\rangle$ and $w_0\in \Cent_W(w)\setminus \Cent_{W_{\Me}}(w)$. 

\medbreak
Assume for a contradiction that $t^{-1}=w_0\cdot t=z w_J^a\cdot t$ for some $a\in \I_{0,2l}$ and some $z\in\bZ$. Then,  $t=w^2_0\cdot t=z^2 w_J^{2a}\cdot t$. Now, since $w_0$ acts as $-\id$ on $\Phi$, then the exponent of $\bZ$ is $\leq 2$, forcing $w_J^{2a}t=t$. As $\Cent_{\Me}(t)=\T$, necessarily $|w_J|=2l+1$ divides $2a$, hence $a=0$ and $w_0\cdot t=zt$. Remark \ref{obs:w0} \ref{item:w0-due} gives $w^2=e$, a contradiction.
\end{proof}

\begin{lema}\label{lem:Jw-contains-D}Assume that $F$ is a Chevalley endomorphism and that $w\neq e$ is minimal in $\WO$.
If $J_w$ has a component of type $D_d$ for some $d\geq 4$, then $\Oc$ is of type C. 
\end{lema}
\begin{proof}
In this situation,  $\Phi$ is of type $D_n$, $E_6$, $E_7$ or $E_8$. Thus $|\bZ|\leq 4$. We use the description of cuspidal elements in Weyl groups of type $D_n$ in Subsection  \ref{sec:cuspidal-classical}, 
from which we retain notation.

\medbreak
Let $J$ be the component of $J_w$ of type $D_d$ and let $w_J$ be the factor of $w$ in $\langle J\rangle$. We denote the simple roots of $\Phi_J$ by $\beta_1,\,\ldots,\,\beta_d$, following the enumeration in \cite{hu-la}. Without loss of generality $w_J=w_{\lambda}\in\langle J\rangle$ for some partition $\lambda=(\lambda_1,\,\ldots,\,\lambda_{2c})$ of $d$. We divide the proof according to the shape of $\lambda$.

\medbreak
\noindent{\bf Case $\lambda_1=1$}. Then necessarily $d=2c$ and $c\geq 2$. In this case $w_J$ acts as $-\id$ on $\Phi_J$, so it preserves all its roots subsystems and lies in the center of $\langle J\rangle$. We verify the hypothesis of Theorem \ref{thm:potente-revised} for the Levi subgroup $\Me=\langle \T, \U_{\pm\beta_{d-2}},\U_{\pm\beta_{d-1}},\U_{\pm\beta_{d}}\rangle\leq \Le_{\langle J\rangle}$. Then $[\Me,\Me]$ is simply connected of type $A_3$ and $w_J$ acts as an outer automorphism. Condition \ref{S} follows from Lemma \ref{lem:perfect-S} and  \ref{R} follows from Lemma \ref{lem:Jw}. 

\medbreak
To prove \ref{W} we verify the hypothesis of Lemma \ref{lem:index}. By construction, 
\begin{align*}\Cent_{\langle J\rangle}(w)=\Cent_{\langle J\rangle}(w_J)=\langle J\rangle\end{align*} is a Weyl group of type $D_{2c}$, of order $\geq 8|\mathbb S_4|$ and \begin{align*}\Cent_{W_{\Me}}(w)=\Cent_{W_{\Me}}(w_J)=W_{\Me}\simeq \mathbb S_4,\end{align*} 
whence 
\begin{align*}
[\Cent_{\langle J_w\rangle}(w\theta):\Cent_{W_{\Me}}(w\theta)]\geq [\Cent_{\langle J_w\rangle}(w\theta):\Cent_{W_{\Me}}(w\theta)] \geq 8>|\bZ|. 
\end{align*}

\medbreak
\noindent{\bf Case $\lambda_1=2$.} Then either $\lambda_2=2$ or $c\geq 2$ and $\lambda_2=\lambda_3=\lambda_4=1$. We verify the conditions of Theorem \ref{thm:potente-revised} for the Levi subgroup $\Me=\langle \T,\U_{\beta_d},\,\U_{\beta_{d-1}}\rangle$. Since $[\Me,\Me]$ is simply-connected, it is isomorphic to $\SL_2(\ku)\times \SL_2(\ku)$. 
By a direct calculation, we see that  $w_{\lambda}(\beta_d)=\beta_{d-1}$ and $w_\lambda(\beta_{d-1})=-\beta_d$, so $M\simeq \SL_2(q^2)$, giving \ref{S} by Lemma \ref{lem:perfect-S}. Lemma \ref{lem:Jw} gives \ref{R}. 
To verify condition \ref{W} we show that  
\begin{align*}[\Cent_{\langle J\rangle}(w):\Cent_{W_{\Me}}(w)]=[\Cent_{\langle J\rangle}(w_J):\Cent_{W_{\Me}}(w_J)]\geq 5\end{align*} and invoke Lemma \ref{lem:index}.
 If $\lambda_1=\lambda_{2}=2$, then 
 \begin{align*}
&e,& &w_{\lambda_1}w_{\lambda_2},& & s_{\beta_{n-2}+\beta_{n-1}}s_{\beta_{n-3}+\beta_{n-2}}, && w_{\lambda_2}^2,&&
s_{\beta_{n-2}+\beta_{n-1}}s_{\beta_{n-3}+\beta_{n-2}} w_{\lambda_2}^2
 \end{align*}
 represent different cosets of $\Cent_{W_{\Me}}(w_J)$ in $\Cent_{\langle J\rangle}(w_J)$. 
 
 If, instead, $\lambda_2=\lambda_3=\lambda_4=1$, then we use 
 \begin{align*}
 &e, && w_{\lambda_1}w_{\lambda_2}, && w_{\lambda_1}w_{\lambda_3}, 
 && w_{\lambda_1}w_{\lambda_4}, &&w_{\lambda_2}w_{\lambda_3}.
 \end{align*}

\medbreak
\noindent{\bf Case $\lambda_1\geq3$.} We verify the hypotheses of Theorem \ref{thm:potente-revised} for $\Phi_{\Me}=\Phi'_{\lambda,1}$. By construction $\Me$ is a standard Levi subgroup, so $[\Me,\Me]$ is simple and simply-connected of type $D_{\lambda_1}$, and $w$ acts as an outer automorphism on $\Phi_{\Me}$. Condition \ref{S} follows from Lemma \ref{lem:perfect-S} because $\lambda_1\geq 3$. Lemma \ref{lem:Jw}  ensures that $\Cent_{\Me}(t)\subseteq \T$, giving condition \ref{R}.

We now verify condition \ref{W}. By construction 
\begin{align*}&\Cent_{\langle J\rangle}(w_J)=\Cent_{\langle J\rangle}(w)\leq \Cent_{\langle J_w\rangle}(w),\\
&\Cent_{W_{\Me}}(w_J)=\Cent_{W_{\Me}}(w)=\langle w_{\lambda_1}^2\rangle,\end{align*}
where the latter equality follows from \cite[Proposition 30]{carter-cc} applied to the Weyl group of type $B_{\lambda_1}$ containing $w_{\lambda_1}$. We analyse different situations separately.

\medbreak

\noindent{\bf Subcase $\lambda_1\geq 3$ and $\lambda_2\geq 3$ or $c\geq 2$.} We exhibit the representatives of $\geq 5$ distinct left cosets of $\Cent_{W_{\Me}}(w)$ in $\Cent_{\langle J_w\rangle}(w)$ and invoke Lemma \ref{lem:index}.  If $\lambda_2\geq 3$, then $|w_{\lambda_2}|=2\lambda_2\geq 6$ and we take $w_{\lambda_1}^{2-a}w_{\lambda_2}^a$ for $a\in\I_{0,2\lambda_2-1}$. If  $c\geq2$, then we take $w_{\lambda_i}w_{\lambda_j}$ for $1\leq i<j\leq 2c$.

\medbreak

\noindent{\bf Subcase $\lambda=(\lambda_1,1)$ with $\lambda_1\geq 3$.}
Then $w_J=w_{\lambda_1}w_{\lambda_2}$ with $|w_{\lambda_2}|=2$ and 
\begin{align*}w_J=w_\lambda=s_{\beta_1+\cdots+\beta_{d-1}}s_{\beta_1+\cdots+\beta_{d-2}+\beta_d}s_{\beta_{d-1}}\cdots s_{\beta_2}\end{align*} 
is conjugate to a Coxeter element of $\langle J\rangle$, as 
\begin{align}\label{eq:coxeter}\mathbf c\coloneq s_{\beta_d}s_{\beta_{d-1}}\cdots s_{\beta_1}=s_{\beta_1+\beta_2+\cdots+\beta_{n-1}+\beta_n}w_Js_{\beta_1+\beta_2+\cdots+\beta_{n-1}+\beta_n}.
\end{align} Hence, \cite[Proposition 30]{carter-cc} gives $\Cent_{\langle J\rangle}(w)=\Cent_{\langle J\rangle}(w_J)=\langle w_J\rangle$ and $\Cent_{W_{\Me}}(w_J)=\langle w_J^2\rangle=\langle w_{\lambda_1}^2\rangle$. 

Assume for a contradiction that \ref{W} fails. In particular 
\begin{align}\label{eq:Dd-noW}w_{\lambda_1}w_{\lambda_2}\cdot t=z w_{\lambda_1}^{2a}\cdot t&&\textrm{ for some }z\in \bZ\textrm{ and some }a\in\I_{0,\lambda_1-1}.\end{align} Regularity of $t$ in $\Le_{J}\leq \Le_{J_w}$ ensures that $z\neq e$, so such an equality never occurs for $E_8$, and necessarily $|z|\in \{2,3,4\}$.  The case $|z|=3$ may occur only if $\Phi$ is of type $E_6$, and is excluded because we would have  
\begin{align*}(w^3_{\lambda_1}w_{\lambda_2})\cdot t=(w_{\lambda_1}w_{\lambda_2})^{3}\cdot t=w_{\lambda_1}^{6a}\cdot t,\end{align*} contradicting regularity of $t$ in $\Le_{J}$. Therefore if \eqref{eq:Dd-noW} holds, $\Phi$ is either of type $D_n$ or $E_7$, $q$ is odd, and $|z|\in\{2,4\}$. 

\medbreak
Let us first focus on $E_7$, so $z=\alpha_2^\vee(-1)\alpha_5^\vee(-1)\alpha_7^\vee(-1)$ and $4\leq d\leq  6$. If $d<6$, then \ref{W} follows from Remark \ref{obs:coefficient} applied to $\beta=\alpha_7$ and regularity of $t$ in $\Me\leq \Le_{J_w}$. Let $d=6$, so $J=J_w$ and $|w_{\lambda_1}|=2\lambda_1=10$. Since we are assuming that \ref{W} fails, we also have $w_0\cdot t\in \bZ \Cent_{W_{\Me}}(w)\cdot t$ by Remark \ref{obs:w0} \ref{item:w0-uno}. Then $t^{-1}=w_0\cdot t=z' w_{\lambda_1}^{2b}\cdot t$ for some $b\in\I_{0,\lambda_1-1}=\I_{0,4}$ and some $z'\in \bZ$. Then $t=w_{\lambda_1}^{4b}\cdot t$, so regularity of $t$ in $\Me$ forces $10|4b$. Hence, $b=0$ and Remark \ref{obs:w0} \ref{item:w0-due} gives $|w|=|w_{\lambda_1}w_{\lambda_2}|\leq 2$, a contradiction. This concludes the verification of \ref{W}  when $\Phi$ is of type $E_7$.

\medbreak
Let us now focus on $D_n$. Here $\beta_i=\alpha_{n-d+i}$ for $i=1,\,\ldots,\,d$. It is more convenient to replace $w_J$ by the $\langle J\rangle$-conjugate element $\mathbf c=s_{n-d+1}\cdots s_n$ as in \eqref{eq:coxeter}, so $\mathbf c_J^{\lambda_1}=w_{0J}$ is the longest element in $\langle J\rangle$, see \cite[Exercise 2, \S 3.19]{hu-rgcg}. Accordingly, we replace $t$ by a $G_w$-conjugate $t'$ in the torus associated with $w'=s_{\alpha_{n-d+1}+\cdots +\alpha_n}ws_{\alpha_{n-d+1}+\cdots +\alpha_n}$. By construction, $w'$ is again minimal and $J_{w'}=J_w$. Also, $\mathbf c=\mathbf c_1\mathbf c_2$ where 
$\mathbf c_i\coloneq s_{\alpha_{n-d+1}+\cdots +\alpha_n}w_{\lambda_i}s_{\alpha_{n-d+1}+\cdots +\alpha_n}$ for $i=1,2$. We replace $\Phi_{\Me}$ by 
$s_{\alpha_{n-d+1}+\cdots +\alpha_n}(\Phi_{\Me})$ so $\Cent_{\langle J\rangle}(w)=\langle \mathbf c\rangle$ and $\Cent_{W_{\Me}}(w)=\langle \mathbf c_1^2\rangle$. Then, if \eqref{eq:Dd-noW} holds, 
\begin{align}\label{eq:w0J}
w_{0J}\cdot t'=\mathbf c^{\lambda_1}\cdot t'=z^{\lambda_1}\mathbf c_1^{2a\lambda_1}\cdot t'=z^{\lambda_1} t', &&z^{2\lambda_1}=e,
\end{align}
 so $z^{\lambda_1}=\alpha^\vee_{n-1}(\eta)\alpha_n^\vee(\eta)$ for some $\eta\in\{\pm1\}$. We now show that \eqref{eq:w0J} cannot occur. Let $t'=\prod_{i=1}^n\alpha_i^\vee(\xi_i)$, with $\xi_i\in\overline{\F}^\times_q$. 

\medbreak
If $d$ is odd, then $n\geq d\geq 5$ and \eqref{eq:w0J} gives 
\begin{align*}
\xi_i^2=1,\;\mbox{ for }n-4\leq i\leq n-2,&& \xi_n\xi_{n-1}=\eta.
\end{align*}
Therefore 
\begin{align*}
&\alpha_{n-3}(t')=\xi_{n-2}\xi_{n-4},&&\alpha_{n-2}(t')=\xi_{n-3}\eta,\\
&\alpha_{n-1}(t')=\xi_{n-1}^2\xi_{n-2}, &&\alpha_{n}(t')=\xi^{-2}_{n-1}\xi_{n-2}. 
\end{align*}

\medbreak

If $d$ is even, then $n\geq d\geq 4$ and \eqref{eq:w0J} gives 
\begin{align*}
\xi_i^2=1,\;\mbox{ for }i\in\{n-2,n-3\},&& \xi_i^2=\eta,\;\mbox{ for }i\in\{n-1,n\}.
\end{align*}
Therefore 
\begin{align*}
&\alpha_{n-3}(t')=\xi_{n-2}\xi_{n-4},&&\alpha_{n-2}(t')=\xi_{n-3}\xi_{n-1}\xi_n,\\
&\alpha_{n-1}(t')=\eta\xi_{n-2}, &&\alpha_{n}(t')=\eta\xi_{n-2}. 
\end{align*}
For both possibilities for the parity of $d$, at least one among $\alpha_{n-3}(t')$, $\alpha_{n-2}(t')$, and $(\sum_{j=n-3}^n\alpha_j)(t')$ equals $1$, contradicting Lemma \ref{lem:Jw}.

\medbreak

\noindent{\bf Subcase $\lambda=(\lambda_1,2)$ with $\lambda_1\geq 3$.} We replace $\Me$ with  the non-standard Levi subgroup whose root system is generated by the orthogonal roots
$\beta_{d-\lambda_1}$ and $w_J\beta_{d-\lambda_1}=w_{\lambda_2}\beta_{d-\lambda_1}$. Condition \ref{S} follows from Lemma \ref{lem:cyclically-permuted}, whilst \ref{R} follows because 
$\Me\subset \Le_J$. For \ref{W} we verify the hypothesis of Lemma \ref{lem:index}. A direct calculation shows that 
$\Cent_{W_{\Me}}(w)=\Cent_{W_{\Me}}(w_J)=\langle w_{\lambda_2}^2\rangle$ and that $\Cent_{\langle J\rangle}(w)=\Cent_{\langle J\rangle}(w_J)=\langle w_J\rangle$. Then, $w_J^a=(w_{\lambda_1}w_{\lambda_2})^a$ for $a\in\I_{0,2\lambda_1-1}$ represent distinct cosets of  $\Cent_{W_{\Me}}(w)$ in $\Cent_{\langle J\rangle}(w)$.
\end{proof}

\section{Semisimple classes of type C in Chevalley groups}\label{sec:typeC}
Let us first summarize what is already known about semisimple classes in Chevalley groups.

\medbreak
\begin{itemize}[leftmargin=*]
\item Semisimple classes in $\PSL_{n+1}(q)$, for $n>1$ have been dealt with 
in \cite{ACG-III,ACG-VII}. All classes collapse  except from the kthulhu classes 
listed in Table \ref{tab:ss-psl2}, which fall down, \cite[Theorems II,~III]{ACG-VII}. 
For $\PSL_2(q)$, 
see Section \ref{sec:abelian-techniques}.

\medbreak
\item Semisimple classes in $\PSp_{2n}(q)$ have been handled in \cite[Theorem 6.1]{ACG-VII}. All classes collapse except from the kthulhu classes listed in Table \ref{tab:ss-psl2}  when $n=2$ and $q\in\{3,5,7\}$.

\medbreak
\item
Semisimple classes in $\Pom^+_{2n}(q)$ and $\Pom^+_{2n+1}(q)$ that are represented in $\T^{F_w}$ for $w$ lying in the maximal parabolic subgroup $W_{A_{n-1}}$ of $W$ of type $A_{n-1}$ are dealt with in \cite[Theorem 6.2]{ACG-VII}, with a few remaining open cases occurring for $w=e$  and $q\in\{3,5,7\}$. The latter will be dealt with in Section \ref{sec:notC-Omega} together with the  open cases in $\PSp_{4}(3)$, $\PSp_{4}(5)$ and $\PSp_{4}(7)$.

\medbreak
\item Apart from the above mentioned open cases for $\Gb=\Pom^+_{2n+1}(q)$ with $q\in\{3,5,7\}$, if $\Gb \neq \PSL_{n+1}(q)$ any non-trivial class intersecting the split torus $\T^F$  collapses, see \cite[Theorem 4.1]{ACG-VII}.
\end{itemize}
\medbreak

In this Section we consider the non-split semisimple classes in simple Chevalley groups other than $\PSL_{n+1}(q)$ and $\PSp_{2n}(q)$. We verify whether the hypotheses of Theorem \ref{thm:potente-revised} hold and  single out the possible exceptions. 

\medbreak

In the whole section $\G=\G_{\rm sc}$ and $\mathfrak O\subseteq \G^F$ is a semisimple class. We fix $t\in \T^{F_w}\subseteq \G$ such that $\Oc_t^{G_w}\simeq\mathfrak O$ is a standard isomorphism and set $\Oc:=\Oc_{\pi(t)}^{\pi(G_w)}\simeq\pi(\mathfrak O)$. We also fix the following notation: \begin{align*}t=\prod_{j=1}^n\alpha_j^\vee(\zeta_j), &&\zeta_j\in\overline{\F_q},\end{align*} with numbering of the simple roots as in \cite{hu-la}. It is convenient to adopt the convention that $\zeta_{j}=1$ for $j\leq0$. 
\medbreak

 We assume that $w$ is a minimal element in $\WO$ and, unless otherwise stated, that $e\not\in\WO$. The verification of the hypotheses of Theorem \ref{thm:potente-revised} will mainly rely on the shape of $J_w$ rather than on the specific $w$ in $\langle J_w\rangle$.

\medbreak

The following observation will be used for inductive arguments.

 \begin{obs}\label{obs:zeta=1}
Assume $\zeta_j=1$ for some $j\in\I_{n}$, and that the root system generated by $\Delta\setminus\{\alpha_j\}$ is $w$-stable. Then, $\Ha\coloneqq \langle \U_{\pm\alpha_i}, \alpha_i\in\Delta\setminus\{\alpha_j\}\rangle$  is semisimple and simply-connected and $\Ha^{F_w}$ is a product of finite groups of Lie type $\mathbf H_1\times\cdots \times\mathbf H_d$. Hence $\Oc$ contains the subrack $\Oc_{\pi(t)}^{\pi(\mathbf H_1\times\cdots\times\mathbf H_d)}$, which is a product of classes in $\pi(\mathbf H_i)$ for $i\in\I_{d}$. If one of the factors projects to a rack of type C, respectively D, respectively F, then $\Oc$ is of type C, respectively D, respectively F. 
\end{obs}

 \subsection{Semisimple classes in \texorpdfstring{$\Pom^+_{2n+1}(q)$}{}}
 \label{subsec:Pom-impar}

\

In this Subsection $\G$ is of type $B_n$ where $n\geq 3$. When $q$ is even, $\Gb\simeq\PSp_{2n}(q)$; thus,  we restrict to the case in which $q$ is odd. The center of $\G^F$ is described in Table \ref{tab:center}. We fix in this subsection only $z\coloneqq \alpha_n^\vee(-1)$, the non-trivial central element in $[\G^F,\G^F]$. We recall that $w_0$ acts as $-\id$ on $\Phi$. By \cite[Theorem 6.2]{ACG-VII}, we restrict to the case in which $\WO\cap W_{A_{n-1}}=\emptyset$. 

\medbreak

A direct calculation using \cite[Lemma 8.19]{malle-testerman} gives
\begin{align*}
&s_1\cdot t=\alpha_1^\vee(\zeta_1^{-2}\zeta_2)t,&&s_j\cdot t=t\alpha_j^\vee(\zeta_{j-1}\zeta_j^{-2}\zeta_{j+1}),\ j\neq1,n-1,n, \\
&s_{n-1}\cdot t=t\alpha_{n-1}^\vee(\zeta_{n-2}\zeta_{n-1}^{-2}\zeta_n^2), &&s_n\cdot t=t\alpha_n^\vee(\zeta_{n-1}\zeta_n^{-2}).
\end{align*}

\medbreak
\begin{lema}\label{lem:Bn-not-irreducible}
Assume that $\Phi_{J_w}$ is not irreducible and that one of the following conditions hold:

\medbreak
\begin{enumerate}[leftmargin=*,label=\rm{(\roman*)}]
\item\label{item:Bn-rank2-2} The root system $\Phi_{J_w}$ has a component $\Phi_{J}$ of  rank $\geq 2$ and the rank of $\Phi_{J_w}\setminus \Phi_{J}$, a possibly reducible root system, is $\geq 2$. 
 
\medbreak
\item\label{item:Bn-two-componentsq>3} $q>3$, the root system $\Phi_{J_w}$ has a component $\Phi_{J}$ of rank $1$, and the rank of $\Phi_{J_w}\setminus \Phi_{J}$, a possibly reducible root system, is $\geq 2$. 
 
\medbreak
\item\label{item:Bn-Al-even}  $\Phi_{J_w}$ is of type $A_{2l}+\widetilde{A}_1$, with $l\geq1$.

\medbreak
\item\label{item:Bn-2componentsA1-q>3} $\Phi_{J_w}$ is of type $A_1+\widetilde{A}_1$ and $q>3$. 
 
\medbreak
\item\label{item:Bn-2componentsA1Bj}$\Phi_{J_w}$ is of type $A_1+{B}_j$ for some $j\in\I_{2,n-2}$. 
\end{enumerate}
\end{lema}
Then $\Oc$ is of type C.
\begin{proof}If either \ref{item:Bn-rank2-2} or \ref{item:Bn-two-componentsq>3} holds, then  the hypotheses of Lemma \ref{lem:levi} are verified, in virtue of Lemma \ref{lem:Jw}. If \ref{item:Bn-Al-even} holds, then we invoke Lemma \ref{lem:A-odd}.
Assume \ref{item:Bn-2componentsA1-q>3}. Without loss of generality  $\Delta_{J_w}=\{\alpha_1,\alpha_n\}$ and $t\in \T^{F_{s_1s_n}}$, i.e., 
\begin{align}\label{eq:Bn-s1sn}
&s_1s_n\cdot t^q=\alpha_1^\vee(\zeta_1^{-q}\zeta_2^q)\left(\prod_{j=2}^{n-1}\alpha_j(\zeta_j^q)\right) \alpha_n^\vee(\zeta_{n-1}^{q}\alpha_n^{-q})=t,\mbox{ i.e.,}\\
\nonumber&\zeta_1^{q+1}=\zeta_2,\quad \zeta_j\in\F_q^\times \mbox{ for }j=2,\,\ldots,\,n-1,\quad \zeta_n^{q+1}=\zeta_{n-1}^2.
\end{align} 
We consider the standard Levi subgroup $\Me=\langle \T, \U_{\pm\alpha_n}\rangle$. Then \ref{S} holds by Lemma \ref{lem:perfect-S} and Lemma \ref{lem:Jw} guarantees \ref{R}. By Remark \ref{obs:w0} \ref{item:w0-uno}, condition \ref{W}  follows if we show that $w_0\cdot t\not=
z^as_n^b\cdot t$ for any $a,\,b\in\{0,1\}$. Assume that we have an equality. Remark \ref{obs:w0} \ref{item:w0-due} applied to $\alpha_1^\vee$ implies  that $\zeta_1^2=1$, so $\zeta_1\in\F_q$ and $\zeta_2=1$. But then, $\alpha_1(t)=1$, contradicting Lemma \ref{lem:Jw}.

Assume \ref{item:Bn-2componentsA1Bj} holds. Without loss of generality  ${J_w}=\{s_1\}\cup J$ with $J=\{s_{n-j+1},\ldots,\,s_n\}$ so $t\in\T^{F_w}$ gives \begin{align*}t=w\cdot t^q\in\alpha_1^{\vee}(\zeta_1^{-q}\zeta^q_2)\alpha_2^\vee(\zeta_2^q)\prod_{i=3}^n\alpha_i(\ku)\end{align*} with $\zeta_2\in\F_q$ and
$\zeta_1^{q+1}=\zeta_2$.  We take $\Me=\Le_J$, so \ref{S}  holds by Lemma \ref{lem:perfect-S} and \ref{R} follows from Lemma \ref{lem:Jw}. We use Remark \ref{obs:w0} to verify \ref{W}. If $t^{-1}=z \sigma\cdot t$ for some $\sigma\in W_{\Me}$ then $\zeta_1^2=\zeta_2^2=1$. Hence $\zeta_2=\zeta_1^{q+1}=1$ and $\alpha_1(t)=1$, contradicting Lemma \ref{lem:Jw}.
\end{proof}

\begin{lema}\label{lem:Bn-Ad+A1}Assume that $\Phi_{J_w}$ is of type $A_d+ \widetilde{A}_1$ with $d\geq 2$. Then $\Oc$ is of type C. 
\end{lema}
\begin{proof}
If $d$ is even, we invoke Lemma \ref{lem:A-odd}, so we assume for the rest of the proof
that $d=2l-1$ for some $l\geq 2$.
Up to $W$-conjugacy we may assume that $J_w=J_1\cup\{s_n\}$ with $J_1=\{ s_1,\,\ldots,\,s_d\}$ and that $w=w_1s_n=s_nw_1$ where  $w_1=s_1\cdots s_d\in\langle J_1\rangle\simeq \mathbb S_{2l}$ is a $2l$-cycle. Whence $|w_1|=|w|=2l$.

\medbreak
We verify the hypotheses of Theorem \ref{thm:potente-revised} for 
the group 
\[\Me=\langle \T, \U_\delta,\,\delta\in\Phi_{\Me}\rangle\]
where $\Phi_M$  is the root system of type $lA_1$ with basis $\beta_j\coloneqq \alpha_j+\cdots+\alpha_{j+l-1}$ for $j=1,\,\ldots,\,l$, that is orthogonal to $\alpha_n$.
Since $\beta_j=w_1\beta_{j-1}=w\beta_j$ for $j=2,\,\ldots,\,l$ and $w_1\beta_l=w\beta_l=-\beta_1$, the system $\Phi_{\Me}$ is $w$-stable. As $l\geq 2$, condition \ref{S} holds by Lemma \ref{lem:cyclically-permuted}. Condition \ref{R} follows from Lemma \ref{lem:Jw} and the inclusion  $\Phi_{\Me}\subset \Phi_{J_w}$. 

\medbreak 

A direct calculation shows that $\Cent_{W_{\Me}}(w)=\langle w_1^l\rangle=\langle \prod_{j=1}^ls_{\beta_j}\rangle$. By construction,  $w_1\in \Cent_W(w)\setminus \Cent_{W_{\Me}}(w)$, and 
since $w_1\in \langle J_w\rangle$ and $\Me\subset \Le_{J_w}$, there holds $\Cent_{\Me}(w_1\cdot t)=\T$. We claim that $w_1\cdot t\not\in \bZ \Cent_{W_{\Me}}(w)\cdot t$.
Indeed, if $w_1\cdot t=z^aw_1^{lb}\cdot t$ for some $a,\,b\in\{0,1\}$, then by looking at the coefficient of $\alpha_n^\vee$ in the equality, we deduce that necessarily $a=0$, which in turn implies that $w_1^{l-1}\cdot t=t$. However, $w_1^{l-1}\in W_{\Me}$, with $\Cent_{\Me}(t)=\T$, a contradiction. Theorem \ref{thm:potente-revised} applies.\end{proof}

\begin{lema}\label{lem:Bn-q=3}Assume that $q=3$ and that $\Phi_{J_w}$ is of type $A_1+\widetilde{A}_1$. Then $\Oc$ is of type C.
\end{lema}
\begin{proof}Without loss of generality we may assume that $\Delta_{J_w}=\{\alpha_1,\alpha_n\}$ so $\alpha_2,\,\alpha_{n-1}\not\in\Delta_{J_w}$. The relations \eqref{eq:Bn-s1sn} are still in force.  We verify the hypotheses of Theorem \ref{thm:potente-revised}, considering the subgroup $\Me=\langle \T, \U_\delta,\,\delta\in \Phi_{\Me}\rangle$, where $\Phi_{\Me}$ is generated by $\alpha=\alpha_1+\cdots+\alpha_{n-1}$ and $\beta=w\alpha=\alpha+2\alpha_n-\alpha_1$.  First of all, $\Phi_{\Me}$ is of type $A_2$ and $w$ acts on $w$ interchanging the simple roots $\alpha$ and $\beta$, so $M\coloneqq [\Me,\Me]^F\simeq \SU_3(3)=\PSU_3(3)$ is a non-abelian finite simple group, giving \ref{S}.

\medbreak
We verify  \ref{W} using Remark \ref{obs:w0}. By construction, we have 
\[\Cent_{W_{\Me}}(w)=\langle s_{\alpha+\beta}\rangle=\langle s_0\rangle.\] 
Assume that $t^{-1}=w_0\cdot t= z^a s_0^b\cdot t$ for some $a,b\in\{0,1\}$. If $b=0$, the equality implies $\zeta_1^2=1$, so $\zeta_2=\zeta_1^{4}=1$ gives $\alpha_1(t)=1$, contradicting Lemma \ref{lem:Jw}. Hence $b=1$.

\medbreak
A direct calculation shows that
\begin{align*}s_{0}\cdot t=t\, (\alpha_0)^\vee(\zeta_2^{-1})=t\,\alpha_1^\vee(\zeta_2^{-1})\left(\prod_{i=2}^{n-1}\alpha_i^\vee(\zeta_2^{-2})\right)\alpha_n^\vee(\zeta_2^{-1}).\end{align*}
Comparing the coefficient of $\alpha_1^\vee$ in the equality $t^{-1}= z^a s_0\cdot t$
gives $\zeta_1^2=\zeta_2$. But then $\alpha_1(t)=1$, impossible by Lemma \ref{lem:Jw}. Hence, $w_0\cdot t\not\in \bZ \Cent_{W_{\Me}}(w)$.

\medbreak
Finally, we check \ref{R}. A direct calculation shows that, for any $n\geq 3$, we have
\begin{align*}
\alpha(t)=\zeta_1\zeta_n^{-2}\zeta_{n-1},&& \beta(t)=\zeta_1^{-1}\zeta_2\zeta_n^2\zeta_{n-1}^{-1},&& (\alpha+\beta)(t)=\zeta_2. \end{align*}

\medbreak
Assume first that $(\alpha(t),\beta(t),(\alpha+\beta)(t))\neq(1,1,1)$. Then, the positive  root subgroup $U\coloneqq \langle \U_{\alpha},\U_{\beta}\rangle^{F_w}$ of $M$ is a $3$-Sylow subgroup  of order $27$ that is not contained in $\Cent_{M}(t)$. 

\medbreak
Assume now that $(\alpha(t),\beta(t),(\alpha+\beta)(t))=(1,1,1)$. In this case, we  replace $\Phi_{\Me}$ by $\Phi'_{\Me}=s_1\Phi_{\Me}$, the root system with base $\alpha'=s_1\alpha$ and $\beta'=s_1\beta$. Note that $\alpha+\beta=s_1(\alpha+\beta)$, so $\Cent_{s_1 W_{\Me} s_1}(w)=\Cent_{W_{\Me}}(w)$. In addition, 
\[
(s_1\alpha)(t)=\alpha(t)\alpha_1(t)^{-1}=\alpha_1(t)^{-1}\neq 1
\]
by  Lemma \ref{lem:Jw}. Hence, we verify \ref{S} and \ref{W} as we did for $\Phi_{\Me}$. Also,  \ref{R} holds because $(\alpha'(t),\beta'(t),(\alpha'+\beta')(t))\neq(1,1,1)$. \end{proof}

\bigskip

We are left with the most laborious case:  $q$ is odd  and $w$ is cuspidal in $\langle J_w\rangle$ where $\Delta_{J_w}=\{\alpha_{n-l+1},\,\ldots,\,\alpha_n\}$ is of type $B_l$ for some $l\in\I_{n}$. We use the description of cuspidal elements recalled in Subsection \ref{sec:cuspidal-classical}, from which we retain notation. Until the end of the subsection we assume that $w={w_\lambda}=\prod_{i=1}^rw_{\lambda_i}$ for some partition $\lambda=(\lambda_1,\,\ldots,\,\lambda_r)$ of $l\leq n$.

\begin{lema}\label{lem:Bn-lambda1>2}Assume $\lambda_1\geq3$ and odd. Then $\Oc$ is of type $C$.
\end{lema}
\begin{proof}We take $\Phi_{\Me}$ to be the $w$-stable root subsystem of type $D_{\lambda_1}$, consisting of all long roots in $\Phi_{\lambda,1}$. 
It has base $\alpha_{n-\lb_1+1},\,\ldots,\,\alpha_{n-1},\,\alpha_{n-1}+2\alpha_n$. 
Then, \ref{S} holds by Lemma \ref{lem:cyclically-permuted}. Since $\Phi_{\lambda,1}\subseteq \Phi_{J_w}$, Lemma \ref{lem:Jw} guarantees \ref{R}.  

\medbreak 

We now show \ref{W} making use of Remark \ref{obs:w0}.  
We observe that 
\[\Cent_{W_{\Me}}(w)=\Cent_{W_{\Me}}(w_{\lb_1}).\]
In addition, the Coxeter element $w_{\lambda_1}\in W_{\Phi_{\lambda,1}}$ does not lie in $W_{\Me}$ because all simple roots of $\Phi_{\lambda,1}$ but one lie in $\Phi_{\Me}$. On the other hand,  $w_{\lb_1}^2\in W_{\Me}$ because $[W_{\Phi_{\lambda,1}}:W_{\Me}]=2$.  
Hence, 
\begin{align*}\Cent_{W_{\Me}}(w)=\Cent_{W_{\Phi{\lb,1}}}(w_{\lb_1})\cap W_{\Me}=\langle w_{\lb_1}^2\rangle.\end{align*}

Assume for a contradiction that $t^{-1}=z^aw_{\lb_1}^{2b}\cdot t$ for some $a\in\{0,1\}$ and $b\in\I_{0,\lb_1-1
}$. Then, $t=z^{a}w_{\lambda_1}^{2b}\cdot t^{-1}=w_{\lambda_1}^{4b}\cdot t=w_{\lambda_1}^{4b}\cdot t$. Therefore $2\lambda_1|4b$ because $\Cent_{\Me}(t)=\T$ in virtue of Lemma \ref{lem:Jw}. Since $\lambda_1$ is odd, $\lambda_1|b$, so $t^{-1}=z^at$. We invoke Remark \ref{obs:w0} \ref{item:w0-due} reaching a contradiction, as $|w|$ is a multiple of $2\lambda_1\geq6$. 
\end{proof}

\begin{lema}\label{lem:Bn-lambda1=2d}If $\lambda_1=2d$ for some $d\geq 1$, then $\Oc$ is of type $C$.
\end{lema}
\begin{proof}We take $\Phi_{\Me}$ to be the root subsystem of $\Phi_{\lambda,1}$, consisting of the $2\lambda_1$ long roots in the $\langle w_{\lambda_1}\rangle$-orbit of 
\[
\alpha\coloneqq \alpha_{n-\lambda_1+1}+\cdots+\alpha_{n-d}.
\]
Now, $w_{\lb_1}^a\cdot \alpha\perp w_{\lb_1}^b\cdot\alpha$ whenever $a-b\not\equiv 0\mod\lb_1$ and $w_{\lb_1}^{\lb_1}\cdot\alpha=-\alpha$, so $\Phi_{\Me}$ is of type ${\lb_1}A_1$ and $w$ cyclically permutes the components. Now, $F_{w^{\lb_1}}$ acts as an endomorphism of $\langle \U_{\pm\alpha}\rangle$ that differs from $\Fr_{q^{\lb_1}}$ by an inner automorphism. Lemma \ref{lem:cyclically-permuted} gives \ref{S}. Since $\Me\leq \Le_{J_w}$, by Lemma \ref{lem:Jw} we get $\Cent_{\Me}(t)=\T$, whence \ref{R} holds. 

\medbreak

Now, $\langle w_{\lambda_1}\rangle\leq \Cent_W(w_{\lambda})$ and the inclusion $\Phi_{\Me}\subset \Phi_{\lb,1}=\Phi_{J_w}$ gives 
\begin{align*}\Cent_{\Me}(w_{\lb_1}\cdot t)\subseteq \Cent_{\Le_{J_w}}(w_{\lb_1}\cdot t)=\T.\end{align*}
Hence, \ref{W} follows if we show that $w_{\lb_1}\cdot t\not\in \bZ \Cent_{W_{\Me}}(w)\cdot t$. First of all, $W_{\Me}$ consists only of involutions and 
\begin{equation*}\Cent_{W_{\Me}}(w_\lambda)=\Cent_{W_{\Me}}(w_{\lambda_1})=\langle \prod_{i=1}^{\lambda_1}s_{w_{\lambda_1}^i\alpha}\rangle=\langle w_{\lambda_1}^{\lambda_1}\rangle\end{equation*} has order $2$. Also, $w_{\lambda_1}\cdot t\not\in \{t,w_{\lambda_1}^{\lambda_1}\cdot t\}$ because $\Cent_{\Le_{J_w}}(t)=\T$. Assume for a contradiction that $w_{\lambda_1}\cdot t=zw_{\lambda_1}^{a\lambda_1}\cdot t$ for some $a\in\{0,1\}$. Then,  either 
$w_{\lambda_1}\cdot t=z t$ or $w_{\lambda_1}\cdot t=z^{\lambda_1+1} t$. In both cases,  $w_{\lambda_1}^2t=t$ with $|w_{\lambda_1}|=2\lambda_1$, contradicting once more that $\Cent_{\Le_{J_w}}(t)=\T$.\end{proof}

\medbreak

It remains to consider the case in which $r=l\leq n$ 
and $\lb_j=1$ for $j\in\I_l$, that is, when  $\Lambda_j=j-1$ for all $j$. 
In this case $w_{\lb_j}=s_{\gamma_j}$ and $\gamma_j=\sum_{i=n-j+1}^n\alpha_i$ for all $j$. Hence, $\gamma_j\perp\gamma_i$ for $i\neq j$, 
so $w=w_{\lb}=s_{\gamma_1}\cdots s_{\gamma_l}$ is an involution. In detail the action on $t$ reads

\begin{align*}
&s_{\gamma_j}\cdot t=\begin{cases}
t\gamma_n^\vee(\zeta_{1}^{-1}), & \text{if $j=n$,}\\
t\gamma_j^\vee(\zeta_{n-j}\zeta_{n-j+1}^{-1}), & \text{if $1<j<n$,}\\
t\alpha_n^\vee(\zeta_{n-1}\zeta_{n}^{-2}), & \text{if $j=1$,}
\end{cases}\\
&\gamma_j^\vee(\eta)=\left(\prod_{i=n-j+1}^{n-1}\alpha_{i}^\vee(\eta^2)\right)\alpha_n^\vee(\eta), &\eta\in \ku^\times,\, \textrm{ if }j>1, \\
&\gamma_1^\vee(\eta)=\alpha_n^\vee(\eta), &\eta\in \ku^\times,\\
&w \cdot t=t\left(\prod_{i=n-l+1}^{n-1}\alpha_i^\vee(\zeta_i^{-2}\zeta_{n-l}^2)\right)\alpha_n^\vee(\zeta_n^{-2}\zeta_{n-l}),& \textrm{if } 1<l<n,\\
&w \cdot t=s_n\cdot t=t\alpha_n^\vee(\zeta_n^{-2}\zeta_{n-1}),& \textrm{if }l=1,\\
&w \cdot t=w_0\cdot t=t^{-1},& \textrm{if } l=n.
\end{align*}
Depending on $l$, the condition $t\in\T^{F_w}$ reads:
\begin{align*}
&\zeta_i^{q-1}=1,\, \textrm{ for } 1\leq i\leq n-1,& \zeta_n^{q+1}=\zeta_{n-1},&&\textrm{ when }l=1;\\
&\zeta_i^{q+1}=1,\, \textrm{ for } 1\leq i\leq n,&&& \textrm{when }l=n;
\end{align*}
and 
\begin{align*}
&\zeta_i^{q-1}=1,\,\textrm{ for }1\leq i\leq n-l,&\zeta_n^{q+1}=\zeta_{n-l},\\
&\zeta_i^{q+1}=\zeta^2_{n-l},\,\textrm{ for }n-l+1\leq i\leq n-1,&\textrm{ when }1<l<n.\end{align*}

\begin{lema}\label{lem:bn-lambda1=1-l=3}
Assume that $\lambda_1=1$ and that neither of the following happens:
\begin{align}\label{eq:l1}
l&=1,& &\begin{aligned}
t&=\bigl(\prod_{i=1}^{\left[\frac{n}{2}\right]}\alpha^\vee_{2i-1}(-1)\bigr)\alpha_n^\vee(\zeta_n),\\ \zeta_n^{q+1} &=(-1)^{n-1};
\end{aligned}
\\
\label{eq:l2}
l&=2, & &\begin{aligned}
t& =\bigl(\prod_{i=1}^{\left[\frac{n-1}{2}\right]}\alpha^\vee_{2i-1}(-1)\bigr)\alpha_{n-1}^\vee(\zeta_{n-1})\alpha_n^\vee(\zeta_n),
\\
\zeta_n^{q+1} & =(-1)^{n},\;\;\zeta_{n-1}^{q+1}=1.
\end{aligned}
\end{align}
Then, $\Oc$ is of type C, D, or F.
\end{lema}
\begin{proof}We consider the root subsystem $\Phi_{\Me}$ with base $\{\alpha_{n-1},\alpha_n\}$. Then, $[\Me,\Me]$ is simple simply-connected of type $B_2$ so $[\Me,\Me]\simeq\Sp_4(\ku)$, whence $M=\Sp_4(q)=[M,M]$, giving \ref{S} by Lemma \ref{lem:perfect-S}. Since $\alpha_n(t)\neq1$, condition \ref{R} follows from Lemma \ref{lem:non-centrality}.

\medbreak

We turn to the condition \ref{W}.  In our situation, $w$ acts as $-\id$ on $\Phi_{\Me}$   so $\Cent_{W_{\Me}}(w)=W_{\Me}$. If $\zeta_i^2\neq 1$ for some $i\leq n-2$, then we invoke Remarks \ref{obs:w0} and  \ref{obs:coefficient}. 

\medbreak

We assume for the rest of the proof that $\zeta_i^2=1$ for all $i\leq n-2$. 

\medbreak

We observe that if $j\neq n-l$ and $j\leq n-2$, then $s_j\in \Cent_W(w)\setminus W_{\Me}$ and $\alpha_n(s_j\cdot t)=\alpha_n(t)\neq 1$. Lemma \ref{lem:non-centrality} applied  to $s_j\cdot t$ gives $s_j\cdot t\not\in \Cent_{\T}(M)$. 
Hence, if  $s_j\cdot t\not\in \bZ W_{\Me}\cdot t$ for some $j\in\I_{n-2}\setminus\{n-l\}$, then \ref{W} holds.

\medbreak

Assume on the contrary that  $s_j\cdot t\in \bZ W_{\Me}\cdot t$ for all $j\in\I_{n-2}\setminus\{n-l\}$. 
The coefficient of $\alpha_j^\vee$ in the expression of $s_j\cdot t$ is 
$\zeta_1^{-1}\zeta_2$ for $j=1$ and $\zeta_{j-1}\zeta_j^{-1}\zeta_{j+1}$ for $j\in\I_{2,n-2}$.
On the other hand, the coefficient of $\alpha_j^\vee$ in the expression of any element in $\bZ W_{\Me}\cdot t$ is $\zeta_j$ for all $j\leq n-2$.
Hence, we necessarily have 
\begin{align}\label{eq:bn-zeta}
&\zeta_2=1 & \textrm{ if }l\neq n-1,\\
\nonumber&\zeta_{j+1}=\zeta_{j-1} & \textrm{ for all }j\neq 1,\,n-l,\  j\leq n-2. 
\end{align}
Let $l\geq 3$. Then, \eqref{eq:bn-zeta} for $j=n-l+1$ gives $\zeta_{n-l}=\zeta_{n-l+2}$ if $n\neq l$ and $\zeta_2=1$ if $n=l$. In both cases,  $t\in \T^{F_{s_{\gamma_1}\cdots s_{\gamma_{l-2}}}}$, contradicting the minimality of $l$. Hence, condition \ref{W} holds for $l\geq 3$. 

\medbreak

Let $l\leq 2$. Then, \ref{W} has been verified unless either
$l=2$ and 
\begin{align*}t=\Big(\prod_{i=1}^{\left[\frac{n-1}{2}\right]}\alpha^\vee_{2i-1}(\epsilon)\Big)\alpha_{n-1}^\vee(\zeta_{n-1})\alpha_n^\vee(\zeta_n),&&\epsilon^2=1,&&\zeta_n^{q+1}=\epsilon^{n},&&\zeta_{n-1}^{q+1}=1,\end{align*} 
or $l=1$ and 
\begin{align*}t=\Big(\prod_{i=1}^{\left[\frac{n}{2}\right]}\alpha^\vee_{2i-1}(\epsilon)\Big)\alpha_n^\vee(\zeta_n),&&\epsilon^2=1,&&\zeta_n^{q+1}=\epsilon^{n-1}.\end{align*}
If $\epsilon=1$, then $t\in \langle \U_{\pm\alpha_n},\U_{\pm\alpha_{n-1}}\rangle^{F_{s_n}}\simeq \Sp_4(q)$ and
it does not lie in the split torus therein. In addition, $\bZ=Z(\Sp_4(q))$, hence, $\Oc$ is of type C, D, or F by \cite[Theorem I]{ACG-VII}. The case $\epsilon=-1$ is discarded by hypothesis. 
\end{proof}

\begin{obs}\label{obs:bn-q3-l2}If $q=3$ and $n$ is even, then \eqref{eq:l2} cannot occur.
Indeed, applying Lemma \ref{lem:Jw} to the roots: $\alpha_{n-1}+\alpha_n,\,\alpha_{n-1}+2\alpha_n,\, \alpha_{n-1}$, gives
\begin{align*}
&\zeta_{n-1}\neq 1,&\zeta_n^2=-1,&&\zeta_{n-1}^2\neq -1
\end{align*}
so $\zeta_{n-1}=-1$ whence $\alpha_n(t)=1$, a contradiction.  
\end{obs}

\begin{lema}\label{lem:Bn-l=1,2} Let
us assume that we are not in any of the following situations:
\begin{align}
\label{eq:bn-discarded-case1}
&\begin{aligned}
&n=2d+1,\, d\geq 1,&&q\equiv3\mod 4,\\
&t=\bigl(\prod_{i=1}^{d}\alpha^\vee_{2i-1}(-1)\bigr)\alpha_n^\vee(\zeta_n),&& \zeta_n^{2}=-1;
\end{aligned}
\\
\label{eq:bn-discarded-case2}
&\begin{aligned}
&n=2d+1,\, d\geq 1,&&q=3, \\
&t=\bigl(\prod_{i=1}^{d}\alpha^\vee_{2i-1}(-1)\bigr)\alpha_{n-1}^\vee(-\omega^2)\alpha_n^\vee(\omega), &&\F_9^\times=\langle\omega\rangle.
\end{aligned}
\end{align}
Then $\Oc$ is of type C, D or F.
\end{lema}
\begin{proof}
It is enough to consider $t$ as in \eqref{eq:l1} or \eqref{eq:l2}, so either $w=s_n$ or $w=s_ns_{\alpha_{n-1}+\alpha_n}$. 

\medbreak
Assume first that $q>3$.  Let $\gamma=\alpha_1+\cdots+\alpha_n$. We consider the $F_w$-stable non-standard Levi subgroup $\Me\coloneqq \langle \T, \U_{\pm\gamma}\rangle$. Then  $[\Me,\Me]\simeq\SL_2(\ku)$ and
\begin{align*}M\coloneqq [\Me,\Me]^{F_{s_n}}\simeq[\Me,\Me]^{F_{s_ns_{\gamma_2}}}\simeq[\Me,\Me]^{Fr_q}\simeq \SL_2(q).\end{align*} Then, $M\simeq [M,M]$ and Lemma \ref{lem:perfect-S} gives \ref{S}. In addition,  $\gamma_n(t)=-1$, whence $\Cent_{\Me}(t)=\T$, so \ref{R} holds. 

\medbreak
Now, $W_{\Me}=\langle s_{\gamma}\rangle=\Cent_{W_{\Me}}(w)$ and $s_{\gamma}\cdot t=\gamma^\vee(-1)t=t\alpha_n^\vee(-1)$, so $\bZ \Cent_{W_{\Me}}(w)\cdot t=\bZ t$. Remark \ref{obs:w0} applies unless $t^2=1$ or $t^2=z$. Hence 
\ref{W} holds unless $(\zeta^2_{n-1},\zeta_n^2)=(1,\pm1)$. 
We analyze when this can happen.

\medbreak
The condition $e\not\in\WO$ forces $(\zeta^2_{n-1},\zeta_n^2)=(1,-1)$ and $q\equiv3\mod4$, so also $\zeta_n^{q+1}=1$. Thus  this situation can only occur if we are in \eqref{eq:l1} and $n$ is odd or in \eqref{eq:l2} and $n$ is even. 
The first scenario is the discarded case \eqref{eq:bn-discarded-case1}. 

\medbreak
In the  second scenario the equality $(\zeta^2_{n-1},\zeta_n^2)=(1,-1)$ cannot hold  because Lemma \ref{lem:Jw} gives  
\begin{align*}
1 &\neq \alpha_n(t)=\zeta_n^{2}\zeta_{n-1}^{-1}=-\zeta_{n-1}& &\text{and}&
1 &\neq(\alpha_{n-1}+\alpha_n)(t)=\zeta_{n-2}^{-1}\zeta_{n-1}=\zeta_{n-1},
\end{align*}
which implies that $\zeta_{n-1}^2\neq1$.

\medbreak

In the rest of the proof we assume that $q=3$.  
By Remark \ref{obs:bn-q3-l2} the case \eqref{eq:l2} occurs only if $n$ is odd. 
Then $\zeta_n^2\neq\zeta_{n-1}$ by Lemma \ref{lem:Jw} 
so this is the discarded case \eqref{eq:bn-discarded-case2}.

The case \eqref{eq:l1} with $n=2d+1$ turns out to be the discarded case \eqref{eq:bn-discarded-case1}, since $\zeta_n^{q+1}=\zeta_n^4=1$ and $\zeta_n^2\neq 1$ because $e\not\in\WO$.

\medbreak

Assume we are in \eqref{eq:l1} with $n=2d$. Then $t$ lies in the subgroup $\langle \prod_{i=1}^{d-1}\alpha^\vee_{2i-1}(-1)\rangle \times\langle\U_{\pm\alpha_{n-1}},\U_{\pm\alpha_n}\rangle\simeq \mathbb Z/2\mathbb Z\times \Sp_4(\ku)$, a product of $s_n$-stable subgroups.  Then the class of $\pi(t)$ contains a subrack isomorphic to the class of $t'\coloneqq \pi(\alpha^\vee_{n-1}(-1)\alpha^\vee_n(\zeta_n))$ in $\PSp_4(q)$. Since $\zeta_n^4=-1$, the element $\pi(t')\in\PSp_4(q)$ is not an involution, hence its class is of type C, D or F by \cite[Theorem I]{ACG-VII}. We conclude by Remark \ref{obs:zeta=1}.
\end{proof}

\medbreak

Combining the results in this subsection with those in \cite{ACG-VII} we obtain the following result for semisimple conjugacy classes with arbitrary $\WO$.
\begin{theorem}\label{thm:bn-typeC}
Let  $\Oc=\pi(\mathfrak O)$ be a non-trivial semisimple conjugacy class in $\Pom_{2n+1}(q)$. Assume $t\in\T^{F_w}$ with $w$ minimal and that we are not in one of these situations: \eqref{eq:bn-discarded-case1} with $w=s_n$, or \eqref{eq:bn-discarded-case2} with $w=s_ns_{\alpha_{n-1}+\alpha_n}$, nor
\begin{align}
\label{eq:bn-w=1-odd}w&=e,\, n=2d+1,\, d\geq 1, &q=5,\\
\nonumber t&= \bigl(\prod_{i=1}^{d} \alpha^{\vee}_{2i-1} (-1)\bigl) \alpha^{\vee}_{n} (\omega), &\F_5^\times=\langle\omega\rangle,\\
\label{eq:bn-w=1-even}w&=e,\, n=2d,\, d\geq 1, &q=3,5,7,\\
\nonumber t&= \bigl(\prod_{i=1}^{d} \alpha^{\vee}_{2i-1} (-1)\bigl) \alpha^{\vee}_{n} (\zeta), &\zeta= \pm1.
\end{align}
Then, $\mathfrak O$ and $\Oc$ are of type C, D or F. 
\end{theorem}
\begin{proof}If $n=2$ then the statement is \cite[Theorem I]{ACG-VII}. Let $n\geq 3$. If $e\in\WO$, we invoke \cite[Theorem 4.1]{ACG-VII}. Note that if $n$ is odd, the element in \cite[Equation (4.2)]{ACG-VII} occurs only if $q\equiv1\mod 4$, so only for $q=5$. If $e\not\in\WO$ the statement follows from Lemmata \ref{lem:Bn-not-irreducible},  \ref{lem:Bn-Ad+A1}, \ref{lem:Bn-q=3}, \ref{lem:Bn-lambda1>2}, \ref{lem:Bn-lambda1=2d}, \ref{lem:bn-lambda1=1-l=3} and \ref{lem:Bn-l=1,2}.
\end{proof}

\begin{obs}
If $t$ is as in \eqref{eq:bn-discarded-case1}, \eqref{eq:bn-w=1-odd} and \eqref{eq:bn-w=1-even} then $\pi(t)$ is the diagonal matrix ${\diag}(-\id_n,1,-\id_n)\in \pi(\T)^F$. If $n$ is even or if $n$ is odd and $q\equiv1\mod 4$, then $\pi(t)\in\pi(\T^F)$. If $t$ is as in  \eqref{eq:bn-discarded-case2}, then 
$\pi(t)={\diag}(-\id_{n-2},\omega^2,\id_3,\omega^{-2},-\id_{n-2})$. 
\end{obs}

\subsection{Semisimple classes in \texorpdfstring{$\Pom^+_{2n}(q)$}{}} \label{subsec:Pom-par}

\

In this Subsection, we assume that  $\G$ is of type $D_n$ for $n\geq 4$.
For recursive arguments, we will use the usual convention that $D_2=2A_1$ and  $D_3=A_3$. The center of $\G^F$ is described in Table \ref{tab:center}. We will need the following equalities
\begin{align}\label{eq:dn-formulas}
\begin{aligned}
 w_0\cdot t &= \begin{cases} t^{-1}& \textrm{ if $n$ is even,}\\
 \left(\prod_{j=1}^{n-2}\alpha_j^\vee(\zeta^{-1}_j)\right)\alpha_{n-1}^\vee(\zeta^{-1}_n)\alpha_n^\vee(\zeta^{-1}_{n-1})& \textrm{ if $n$ is odd.}\end{cases}\\
 \alpha_0 &= \alpha_1+2\alpha_2+\cdots+2\alpha_{n-2}+\alpha_{n-1}+\alpha_n,\\
 s_0\cdot 
t&=\alpha_1^\vee(\zeta_1\zeta_2^{-1})\left(\prod_{i=2}^{n-2}\alpha_i^\vee(\zeta_i\zeta_2^{-2})\right)\alpha_{n-1}^\vee(\zeta_{n-1}\zeta_2^{-1})\alpha_n^\vee(\zeta_n\zeta_{2}^{-1}). 
\end{aligned}
\end{align}

The cases in which $\alpha_n\not\in J_w$ are covered by \cite[Theorem 6.2]{ACG-VII}. If $\alpha_{n-1}\not\in J_w$, applying  the outer automorphism of $\G$ interchanging $\alpha_n$ and $\alpha_{n-1}$, we can reduce to the case $\alpha_n\not\in J_w$. Hence, we may assume that $\alpha_{n-1},\alpha_{n}\in J_w$. Then, $J_w$ is always $w_0$-stable and Lemma \ref{lem:Jw} gives 
\begin{align*}
\alpha_j(w_0\cdot t)&=(w_0(\alpha_j))(t)\neq 1 &&\text{if} & s_j&\in J_w.
\end{align*}

\medbreak

We denote by $J_n$ the component of $J_w$ containing $s_n$. 
We verify the hypotheses of Theorem \ref{thm:potente-revised}, 
analyzing separate cases according to the size or type of $J_n$. 
Recall that the factor of $w$ in each irreducible component of 
$\langle J_w\rangle$ is cuspidal therein by construction. 
In a component of type $A_l$, the factor is conjugate to a Coxeter element. 
In a component of type $D_l$  the factor is described 
in Subsection \ref{sec:cuspidal-classical}.

\begin{lema}\label{lem:dn-Jn-geq3}
If $|J_n|\geq 3$, then $\Oc$ is of type C. 
\end{lema}
\begin{proof}
If $|J_n|\geq 4$, then $J_n$ is of type $D_{|J_n|}$ and the statement is Lemma \ref{lem:Jw-contains-D}. 

\medbreak
Thus, we suppose that $J_n=\{\alpha_{n-2},\alpha_{n-1},\alpha_n\}$ is of type $A_3$. 
Without loss of generality, we can assume that $w$ is the product of the Coxeter element 
$\tau\coloneqq s_ns_{n-2}s_{n-1}\in\langle J_n\rangle$, of order $4$, with other elements
that commute with it and are in the other, possibly empty, components of $\langle J_w\rangle$. 

\medbreak
Consider the non-standard Levi subgroup  $\Me$ associated with the $w$-stable root system with base $\{\alpha_{n-2}+\alpha_{n-1}, \alpha_{n-2}+\alpha_n\}$, of type $2A_1$. Since $w$ interchanges the two components,  $M=[M,M]\simeq \SL_2(q^2)$; this gives \ref{S}. Condition \ref{R} follows from Lemma \ref{lem:Jw} and the inclusion $\Phi_{\Me}\subseteq \Phi_{J_w}$. 

\medbreak
Also, $W_{\Me}=\langle s_{\alpha_{n-2}+\alpha_{n-1}}s_{\alpha_{n-2}+\alpha_n}\rangle=\langle \tau^2\rangle$.
Assume that $\tau\cdot t=z\tau^{2a}\cdot t$ for some $z\in \bZ$ and some $a\in \I_{0,3}$. By regularity of $t$ in $\Le_{J_w}$ we exclude the possibility $z=e$. Also, comparing  the coefficients of $\alpha_1^\vee$ in the equality, the only admissible possibility for $z$ is $\alpha_{n-1}^\vee(-1)\alpha_n^\vee(-1)$. But then
\[\tau^2\cdot t=z^2\tau^{4a}\cdot t=t\]
which contradicts the regularity of $t$.  The non-centrality condition \ref{W2} follows from Lemma \ref{lem:Jw} because $\Phi_w$ is $\tau$-stable.
\end{proof}

\begin{lema}\label{lem:dn-A}If 
$\Delta_{J_w}$ has a component $\Delta_{J}$ of rank $l\geq 2$ with $J\neq J_n$, then $\Oc$ is of type C. 
\end{lema}
\begin{proof}
By Lemma \ref{lem:dn-Jn-geq3} it is enough to consider the case in which $J_n=\{s_n\}$. 
Since $s_{n-1}\in J_w$ and $s_{n-2}\not\in J_w$, we may assume that $\Delta_{J}=\{\alpha_1,\,\ldots,\alpha_{l}\}$ for some  $l\leq n-3$. Hence $n\geq 5$.

\medbreak

We first take $\Me\coloneqq \Le_{J}\leq \Le_{J_w}$. By construction $F_w$ is a Chevalley automorphism of $\Me$ so \ref{S} follows from Lemma \ref{lem:perfect-S} and \ref{R} from Lemma \ref{lem:Jw}. 

\medbreak

We verify \ref{W}. 
By construction $s_{n-1}, s_n\in \Cent_{\langle J_w\rangle}(w)$ and 
\[\Cent_{\Me}(s_{n}\cdot t)=\Cent_{\Me}(s_{n-1}\cdot t)=\T.\] 
Assume that \ref{W1} fails. Then \begin{align}\label{eq:equa}s_n\cdot t=z_1\sigma_1\cdot t,&&\mbox{ and } &&s_{n-1}\cdot t=z_2\sigma_2\cdot t\end{align} for some $z_1,z_2\in \bZ$ and $\sigma_1,\sigma_2\in \Cent_{\langle J\rangle }(w)$. Regularity of $t$ in $\Le_{J_w}$ forces $z_1,z_2\neq e$ so $q$ is odd. Comparing the coefficients of $\alpha_{n-1}^\vee$ and $\alpha_n^\vee$ in the  equalities \eqref{eq:equa} we see that $n$ is even,
\begin{align*}
z_1 &= \Big(\prod_{i=1}^{n-2}\alpha_i^\vee((-1)^i)\Big)\alpha_n^\vee(-1)&
&\text{and}&
z_2 &= \Big(\prod_{i=1}^{n-2}\alpha_i^\vee((-1)^i)\Big)\alpha_{n-1}^\vee(-1).
\end{align*}
This yields $\zeta_n^{2}=\zeta_{n-1}^2=-\zeta_{n-2}$. 
Failure of \ref{W1} and parity of $n$ give also $w_0\cdot t=t^{-1}=z_3 \sigma_3\cdot t$ for some $z_3\in \bZ$ and $\sigma_3\in \Cent_{\langle J\rangle }(w)$. Then 
\[\zeta_{n-2}=-\zeta_n^2=-\zeta_{n-2}^2=\pm1. \]
If $\zeta_{n-2}=-1$, then $\alpha_n(t)=1$ contradicting Lemma \ref{lem:Jw}. If $\zeta_{n-2}=1$, then $t$ lies in the simply-connected group
\[\langle \U_{\pm\gamma}, \gamma\in \Delta\setminus \alpha_{n-2}\rangle\simeq \SL_{n-2}(\ku)\times\SL_2(\ku)\times\SL_2(\ku)\]
and $w$ acts on its root system as a Weyl group element. Thus, 
\begin{align*}t\in\langle \U_{\pm\gamma}, \gamma\in \Delta\setminus \alpha_{n-2}\rangle^{F_w} \simeq \SL_{n-2}(q)\times\SL_2(q)\times\SL_2(q),\end{align*} with $n-2\geq 4$ and even. 
The component $t_1$  of $t$ in  $\SL_{n-2}(q)$ is non-central. We conclude invoking \cite[Proposition 5.9]{ACG-VII} and Remark \ref{obs:zeta=1}.
\end{proof}

\begin{lema}\label{lem:dn-2A1}If $J_w=J_n=\{s_{n-1},s_n\}$, then $\Oc$ is of type C, D or F.
\end{lema}
\begin{proof}
If $n=4$, then a triality automorphism maps $\Oc$ to a non-split semisimple class corresponding through the standard isomorphism to some $t'\in \T^{w'}$ with $J_{w'}=\{\alpha_1,\alpha_3\}$. This case is covered by \cite[Theorem 6.2]{ACG-VII}. For the rest of the proof $n\geq 5$. 

\medbreak
Since $t\in\T^{F_{s_{n-1}s_n}}$, we have 
\begin{align*}
\zeta_i^q &=\zeta_i &&\text{if}& i &\in\I_{n-2} &&\text{and} &
\zeta_{n-1}^{q+1} &= \zeta_n^{q+1}=\zeta_{n-2}. 
\end{align*}

\medbreak
If $q=2$, then $\bZ=e$ and $t\in \langle \U_{\pm\alpha_{n-2}},\U_{\pm\alpha_{n-1}},\U_{\pm\alpha_{n}}\rangle^{F_w}\simeq\SL_4(2)$ and it is non-central therein.  Because
of Remark \ref{obs:zeta=1} and \cite[Theorem 5.4]{ACG-VII} the claim follows in this case. 

\medbreak
If $q=3$, then $\zeta_i=\pm1$ for $i\in\I_{n-2}$. Assume $\zeta_{n-3}=1$ (if $n$ is even this is always possible up to multiplying $t$ by an element in $\bZ$). 
Then \begin{align*}t\in \langle \U_{\pm\alpha_j}, j\in\I_{n-4}\rangle^{F_w} \times \langle \U_{\pm\alpha_{n-2}},\U_{\pm\alpha_{n-1}},\U_{\pm\alpha_{n}}\rangle^{F_w}\simeq\SL_{n-3}(3)\times \SL_4(3)\end{align*} and the component $t_1$ in $\SL_4(3)$ is non-central therein. 
Again, Remark \ref{obs:zeta=1} and \cite[Theorem 5.4]{ACG-VII} imply the claim.

\medbreak
Assume that $n$ is odd and $\zeta_{n-3}=-1$. Then we may always assume that $\zeta_{n-4}=1$. Indeed, if $\zeta_{j}=-1$ for all $j\in \I_{n-3}$, then we replace $t$ by 
\[s_{n-4}s_{n-5}\cdots s_2 s_1\cdot t\in \T^{F_w}.\]
If, instead, $\zeta_a=1$ for some $a\geq 1$ and  $\zeta_{j}=-1$ for $j\in\I_{a+1,n-3}$, then we replace $t$ by 
\[s_{n-4}s_{n-5}\cdots s_{a+1}\cdot t\in \T^{F_w}.\]
Then, $t\in \langle \U_{\pm\alpha_j}, j\in\I_{n-5}\rangle^{F_w} \times \langle \U_{\pm\alpha_j}, j\in\I_{n-3,n}\rangle^{F_w}$. The second factor is of type $D_4$ and the component $t_1$ of $t$ is non-central therein. It lies in a torus corresponding to a Weyl group element $w_1$ with $J_{w_1}=\{s_4,s_3\}$. We apply an outer automorphism and invoke \cite[Theorem 6.2]{ACG-VII}
to obtain the claim. 

\medbreak
For the rest of the proof $q>3$. Assume first that there is $j\in \I_{n-3}$ such that $\alpha_j(t)\neq 1$. Then \ref{R} and \ref{S} hold for $\Me=\Le_{s_j}=\langle \T, \U_{\pm\alpha_j}\rangle$. 
We verify \ref{W}. By construction $\alpha_j(s_n\cdot t)=\alpha_j(t)\neq 1$.
If $s_n\cdot t=z s_j^a\cdot t$ for some $z\in \bZ$ and some $a\in\{0,1\}$, then $z^2=e$. 
Comparing the coefficients of $\alpha_{n-1}^\vee$ in the equality gives either $z=e$ or $n$ is even and $z=\prod_{i=1}^{n-2}\alpha_i^\vee((-1)^i)\alpha_n^\vee(-1)$. However, in the latter case $n\geq 6$, so 
there is an odd $i\in\I_{n-3}\setminus\{j\}$. Comparing the coefficients of $\alpha_i^\vee$ in the equality $s_n\cdot t=z s_j^a\cdot t$ leads to a contradiction, so $z=e$. Comparing the coefficients of $\alpha_n^\vee$ would then lead to $\zeta_n^2=\zeta_{n-2}$, so $\alpha_n(t)=1$, contradicting Lemma \ref{lem:Jw}. Therefore \ref{W} holds in this case. 

\medbreak

Assume now that $\alpha_j(t)=1$ for all $j\in\I_{n-3}$. Then 
\begin{align*}t=\Big(\prod_{i=1}^{n-2}\alpha_i^\vee(\zeta_1^i)\Big)\alpha_{n-1}^\vee(\zeta_{n-1})\alpha_n^\vee(\zeta_n).\end{align*} 
If $\zeta_1^2=1$, then either $\zeta_{n-3}=1$ or $\zeta_{n-4}=1$ and we proceed as for $q=3$. 

\medbreak
Assume for the rest of the proof that $\zeta_1^2\neq1$. Then ${\alpha_0}(t)\neq 1$. 
We consider the non-standard Levi subgroup $\Me=\langle \T,\U_{\pm\alpha_0}\rangle$, 
so $M\simeq \SL_2(q)$ and \ref{S} and \ref{R} are satisfied. 
Also, $\alpha_0(s_n\cdot t)=\alpha_0(t)\neq1$. 

\medbreak
We claim that $s_n\cdot t\not\in \bZ W_{\Me}\cdot t$.  
Assume for a contradiction that $s_n\cdot t=zs_0^a\cdot t$ for some $z\in \bZ$ and $a\in\{0,1\}$. Then $z^2=e$. Comparing the coefficient of $\alpha_1^\vee$  and making use of \eqref{eq:dn-formulas} shows that the equality $s_n\cdot t=zs_0\cdot t$ may occur only if $n$ is even, $\zeta_1^2=-1$ and 
\begin{align*}
z &=\prod_{i=1}^n\alpha_i^\vee((-1)^i)& &\text{or} & z &= \Big(\prod_{i=1}^{n-2}\alpha_i^\vee((-1)^i)\Big)\alpha_n^\vee(-1).
\end{align*}
However, comparing the coefficient of $\alpha_3^\vee$ would then give $\zeta_1^4=-1$, a contradiction. 
Hence, $a=0$ and regularity of $t$ in $\Le_{J_w}$ gives $z\neq e$. But then, comparing the coefficients of $\alpha_1^\vee$ and of $\alpha_{n-1}^\vee$ in the equality $s_n\cdot t=zt$ gives again a  contradiction.  
\end{proof}

\begin{lema}\label{lem:dn-rA1}
Assume that $(n,q)\neq (4,2)$ and that we are not in the situation
\begin{align}\label{eq:exception-d4}
\begin{aligned}
n&=4,& && q&\equiv3 \hspace{-5pt}\mod 4,
\\
\omega^2_1&=\omega_3^2=\omega_4^2=-1, && 
&t&=\alpha^\vee_1(\omega_1)\alpha^\vee_3(\omega_3)\alpha^\vee_4(\omega_4).
\end{aligned}
\end{align}

If $J_w$ is of type $rA_1$ with $r\geq 3$, then $\Oc$ is of type C, D, or F. 
\end{lema}
\begin{proof}
There is no loss of generality in assuming that $\{s_{n-3},s_{n-1},s_n\}\subseteq J_w$, so $s_{n-2},\,s_{n-4}\not\in J_w$. By construction $\zeta_i^q=\zeta_i$ whenever $s_i\not\in J_w$; also
\[
\zeta_n^{q+1}=\zeta_{n-1}^{q+1}=\zeta_{n-2} \hspace{10pt}\text{ and } \hspace{10pt}
\zeta_{n-3}^{q+1}=\zeta_{n-2}\zeta_{n-4}. 
\]

If $q=2$, then $n\geq 5$ and 
\begin{align*}t&\in\langle \U_{\pm\alpha_j}, j\in\I_{n-3}\rangle^{F_{w}}\times\langle  \U_{\pm\alpha_{n-1}}\rangle^{F_{s_{n-1}}}\times \langle\U_{\pm\alpha_n}\rangle^{F_{s_n}}\\
&\simeq \SL_{n-2}(2)\times\SL_2(2)\times\SL_2(2)\end{align*} and the component in $\SL_{n-2}(2)$ is non-trivial and non-cuspidal (i.e., not irreducible). By Remark \ref{obs:zeta=1} and \cite[Theorem 1.1]{ACG-III}, the claim holds.

\medbreak

Let now $q=3$. In this case $\zeta_i^2=1$ for all $i$ such that $s_i\not\in J_w$,
\begin{align*}
\zeta_n^{4} &=\zeta_{n-1}^{4}=\zeta_{n-2} &&\text{and} &
\zeta_{n-3}^{4} &= \zeta_{n-2}\zeta_{n-4}. 
\end{align*}

The case $\zeta_{n-2}=1$ and $n=4$ is \eqref{eq:exception-d4}, so it is discarded. If $\zeta_{n-2}=1$ and $n\geq 5$, then we apply Remark \ref{obs:zeta=1}. Indeed, 
\[ t\in\langle \U_{\pm\alpha_j},\,j\in\I_{n-3}\rangle^{F_w}\times\langle \U_{\pm\alpha_{n-1}},\U_{\pm\alpha_n}\rangle^{F_w}\simeq \SL_{n-2}(3)\times \SL_2(3)\times \SL_2(3)\]
and the factor of $t$ in the component isomorphic to $\SL_{n-2}(3)$ is non-trivial and non-cuspidal (i.e., not irreducible), so  \cite[Theorem 1.1]{ACG-III} applies.

\medbreak
Let now $\zeta_{n-2}=-1$ and $n\geq 5$. We consider the $F_w$-stable Levi subgroup $\Me=\langle \T,\U_{\pm\alpha_j},\,j\in\I_{n-3}\rangle$, so  $M\simeq \SL_{n-2}(3)\simeq [M,M]$ satisfies \ref{S}.  In addition, $\alpha_{n-3}(t)\neq 1$ by Lemma \ref{lem:Jw}, so Lemma \ref{lem:non-centrality} gives \ref{R}. We verify condition \ref{W} considering the element $s_n\cdot t$. 

First of all, $\alpha_{n-3}(s_n\cdot t)=\alpha_{n-3}(t)\neq 1$ so Lemma \ref{lem:non-centrality} gives \ref{W2}. In addition, the coefficient of $\alpha_n^\vee$ in $s_n\cdot t$ is $\zeta_{n}^{-1}\zeta_{n-2}=-\zeta_n^{-1}$, whilst the coefficient of $\alpha_n^\vee$ in any element of $\bZ W_{\Me}\cdot t$ is $\pm\zeta_n$. If they were equal, we would have $\zeta_n^4=1\neq\zeta_{n-2}=-1$, a contradiction. Hence \ref{W} holds in this case.

\medbreak
Let us assume next that $n=4$ and $\zeta_{2}=-1$, so 
\begin{align*}
\alpha_2(t) &=-\zeta^{-1}_{1}\zeta^{-1}_{3}\zeta^{-1}_4 &&\text{and} &
\zeta_1^4 &= \zeta_3^4=\zeta_4^4=-1. 
\end{align*}
A product of $3$ generators  in the cyclic group $\F_9^\times$ has never order $2$, so $\alpha_2(t)\neq 1$.
Thus, 
\[(w\alpha_2)(t)=\alpha_2(w^{-1}\cdot t)=\alpha_2(t^3)=\alpha_2(t)^3\neq1.\]
We consider then $\Me\coloneqq \langle \T, \U_{\pm\alpha_2},\U_{\pm(w\alpha_2)}\rangle$ of type $A_2$, so \[[\Me,\Me]\simeq\SL_3(\ku)=\PSL_3(\ku)\] 
as abstract groups.  
Here $w$ acts on the root system of $[\Me,\Me]$ 
as the non-trivial Dynkin diagram automorphism, hence
$M\simeq \SU_3(3)\simeq\PSU_3(3)$, giving \ref{S}.  

In addition, $\Cent_{\Me}(t)\leq \langle \T,\U_{\pm(\alpha_2+w\alpha_2)}\rangle$, so the Sylow $3$-subgroups in  $\Cent_M(t)$  have order $\leq3$, giving \ref{R}.  

We verify \ref{W}. Here $W_{\Me}=\langle s_{\alpha_2+w\alpha_2}\rangle=\langle s_{0}\rangle=\Cent_{W_{\Me}}(w)$ and 
\[s_{0}\cdot t=t\alpha_1^\vee(-1)\alpha_3^\vee(-1)\alpha_4^\vee(-1). 
\]
Also, $s_1\in \Cent_{W}(w)$ and $s_1\cdot t\not\in \bZ W_{\Me}\cdot t$ because the coefficient of $\alpha_1^\vee$ in $s_1\cdot t$  is $-\zeta_1^{-1}$ and $\zeta_1^4=-1$. This concludes the proof for $q=3$.

\medbreak
We assume for the rest of the proof that $q>3$. 
We take $\Me=\Le_{\{s_n\}}$, so $M\simeq\SL_2(q)$ and $\Cent_{\Me}(t)=\T$, giving \ref{S} and \ref{R}. We show that \ref{W} holds by showing that  $\{s_{n-3}\cdot t,\, s_{n-1}\cdot t,\, w_0\cdot t\}\not\subseteq \bZ W_{\Me}\cdot t$: the non-centrality condition follows from Lemma \ref{lem:Jw}. 

Assume  that $s_{n-3}\cdot t=z_1 s^{a_1}_n\cdot t$  for some $z_1\in \bZ$ and some 
$a_1\in\{0,1\}$. Then $z_1^2=e$ and  the regularity of $t$ in $\Le_{J_w}$ forces $|z_1|=2$. Comparing the coefficients of $\alpha_n^\vee$ and $\alpha_{n-1}^\vee$ in the equality we exclude $a_1=0$. Then, we have $s_{n-3}\cdot t=z_1 s_n\cdot t$. A comparison of the coefficients of $\alpha_1^\vee$ and $\alpha_{n-1}^\vee$ shows that that this equality can only occur for $n=4$, with $z_1=\alpha_1^\vee(-1)\alpha^\vee_4(-1)$ and $\zeta_4^2=\zeta_{1}^2=-\zeta_{2}$.  

Assume that we are in this situation. Then we consider $s_{3}\cdot t$. Arguing as above we see that if $s_{3}\cdot t=z_2s_4^{a_2}$ for some $z_2\in \bZ$ and $a_2\in\{0,1\}$ then $|z_2|=2$ and comparing coefficients of $\alpha_1^\vee$ and $\alpha_4^\vee$ we see that  $a_2=1$ and $z_2=\alpha_3^\vee(-1)\alpha_4^\vee(-1)$ whence $\zeta_4^2=\zeta_{1}^2=\zeta_3^2=-\zeta_{2}$. We then consider 
\[
w_0\cdot t=t^{-1}=\alpha_1^\vee(-\zeta_2\zeta_1)\alpha_2^\vee(\zeta_2^{-1})\alpha_3^\vee(-\zeta_2\zeta_3)\alpha_4^\vee(-\zeta_2\zeta_4).
\]
If it lies in  $\bZ W_{\Me}\cdot t$, then $\zeta_2^2=1$. If $\zeta_2=-1$, then $\zeta_i\in\F_q$ for all $i$, contradicting minimality of $J_w$. If $\zeta_2=1$, then the 
minimality of $J_w$ forces $q\equiv 3\mod 4$, 
but this is the excluded case \eqref{eq:exception-d4}. 
\end{proof}

\begin{obs}\label{obs:d4-exception}
Let $n=4$ and let $t=\alpha^\vee_1(\omega_1)\alpha^\vee_3(\omega_3)\alpha^\vee_4(\omega_4)\in \T^{F_{s_1s_3s_4}}$. Then conjugation by $s_i$ maps $\omega_i$ to its inverse leaving the other two coefficients unchanged. Hence, if $q=2$ or if we are in the situation  \eqref{eq:exception-d4}, there is only one conjugacy class of semisimple elements with $J_w=\{s_1,s_3,s_4\}$. 
\end{obs}

We collect the results on semisimple classes in $\Gb$. 
\begin{theorem}\label{thm:dn-one-exception}
Let $\Gb=\Pom^+_{2n}(q)$. If $n=4$ and $w=s_1s_3s_4$, assume that
$q\neq2$ and that $t$ is not as in \eqref{eq:exception-d4}. Then $\Oc$ is of type C, D, or F. 
\end{theorem}
\begin{proof}If $s_n\not\in J_w$, or $s_{n-1}\not\in J_w$ this is \cite[Theorem 6.2]{ACG-VII}. If $\{s_{n-1},s_n\}\subseteq J_w$ we invoke Lemmata \ref{lem:dn-Jn-geq3}, \ref{lem:dn-A}, \ref{lem:dn-2A1} and \ref{lem:dn-rA1}.
\end{proof}

\subsection{Semisimple classes in \texorpdfstring{$E_6(q)$}{}} \label{subsec:E6}

\

In this Subsection, we assume that  $\G$ is of type $E_6$. T
he center of $\G^F$ is described in Table \ref{tab:center}. 
We will  use the following formulas
\begin{align*}
\alpha_0 &=\alpha_1+2\alpha_2+2\alpha_3+3\alpha_4+2\alpha_5+\alpha_6,
\\
s_0\cdot t&= \alpha_1^\vee(\zeta_1\zeta_2^{-1})\alpha_2^\vee(\zeta_2^{-1})\alpha_3^\vee(\zeta_3\zeta_2^{-2})\alpha_4^\vee(\zeta_4\zeta_2^{-3})\alpha_5^\vee(\zeta_5\zeta_2^{-2})\alpha_6^\vee(\zeta_6\zeta_2^{-1}),\\ 
w_0\cdot t &= \alpha_1^\vee(\zeta_6^{-1})\alpha_2^\vee(\zeta_2^{-1})\alpha_3^\vee(\zeta_5^{-1})\alpha_4^\vee(\zeta_4^{-1})\alpha_5^\vee(\zeta_3^{-1})\alpha_6^\vee(\zeta_1^{-1}). 
\end{align*}

We verify the hypotheses of Theorem \ref{thm:potente-revised}, analysing separately the possibilities for $J_w$, up to $W$-conjugacy.

\begin{lema}\label{lem:e6-non-cuspidal}If $w$ is not cuspidal, then $\Oc$ is of type C, D, or F.
\end{lema}
\begin{proof}
By assumption, $J_w$ is not of type $E_6$. If $J_w$ has a component of type $D_d$, for some $d\geq 4$, then we invoke Lemma \ref{lem:Jw-contains-D}. If $J_w$ is of type $A_2+ A_1+ A_1$, $A_2+ A_1+ A_1$ or $A_2+ A_2$, then we invoke Lemma \ref{lem:levi}.

\medbreak
Assume that $J_w$ is of type $A_l+ A_1$ for some $l\geq 2$ and let $\Delta_1\coloneqq \{\beta_1,\,\ldots,\,\beta_l\}$ and $\{\beta\}$  be the sets of simple roots corresponding to these components. 
Without loss of generality, we may assume that $w=w_1s_{\beta}$, where $w_1=\prod_{i=1}^ls_{\beta_i}$ is a Coxeter element in its component, hence $|w_1|=l+1$. 
Take $\Phi_{\Me}$ to be the root system with base $\Delta_1$. Then, 
\[
M=[M,M]=\SL_{l+1}(q)
\]
and $t$ is regular in $\Me$, giving \ref{S} and \ref{R}. We verify \ref{W}. Here $\Cent_{W_{\Me}}(w)=\langle w_1\rangle$.  Assume that $s_\beta\cdot t=z w_1^a\cdot t$ for some $a\in\I_{0,l}$ and some $z\in \bZ$. Then, $t=s_\beta^3w_1^{3a}\cdot t=s_\beta w_1^{3a}\cdot t$, contradicting regularity of $t$ in $\Le_{J_w}$.

\medbreak
Assume now that $J_w$ is of type $A_l$ for $l\in\I_{2,5}$. Up to $W$-conjugation we may assume that $\{s_1\}\subseteq J_w\subseteq\{s_1,s_3,s_4,s_5,s_6\}$, and that $w$ is a Coxeter element in $\langle J_w\rangle$. There are two subcases: if $\zeta_2=1$, then
\[t\in \langle \U_{\pm\gamma},\,\gamma\in\Delta\setminus\{\alpha_2\} \rangle^{F_w}\simeq \SL_{6}(q)\]
so $\Oc$ has a subrack that projects to a non-trivial  semisimple class in $\PSL_6(q)$. 
Due to the Remark \ref{obs:zeta=1} and \cite[Theorem I]{ACG-VII}, the claim is valid. 

\medbreak
We assume next that $\zeta_2\neq1$, so $s_0\cdot t\neq t$.  
We take $\Delta_{\Me}=\Delta_{J_w}$, so $M\simeq\SL_{l+1}(q)$ and $\Cent_{\Me}(t)=\T$, giving \ref{S} and \ref{R}. We verify \ref{W}. We have $\Cent_{W_{\Me}}(w)=\langle w\rangle$ and $s_0\in \Cent_W(w)\setminus \Cent_{W_{\Me}}(w)$. In addition, $\alpha_1(s_0\cdot t)=\alpha_1(t)\neq1$, so Lemma \ref{lem:non-centrality} gives \ref{W2}. 

\medbreak
Suppose  that $s_0\cdot t=z w^a\cdot t$ for some $a\in\I_{0,l}$ and some $z\in \bZ$. 
Comparing the coefficients of $\alpha_2^\vee$ we see that this is possible only if $\zeta_2=-1$, so $q$ is odd, and in this case $s_0\cdot t=t\alpha_1^\vee(-1)\alpha_4^\vee(-1)\alpha_6^\vee(-1)$. But then, comparing the coefficients of $\alpha_6^\vee$ shows that the equality $s_0\cdot t=z w^a\cdot t$ may occur only if $l=5$, proving the claim for $l\in\I_{2,4}$. Let thus $l=5$, so $|w|=6$. Then
\begin{align}\label{eq:e6-l=5}t\alpha_1^\vee(-1)\alpha_4^\vee(-1)\alpha_6^\vee(-1)=s_0\cdot t=s^3_0\cdot t=w^{3a}\cdot t\end{align} 
and $a$ is necessarily odd, so $w^{3a}=w^3$. There is no loss of generality in assuming that $w=s_1s_6s_3s_5s_4$. A direct calculation in $W_{\Me}\simeq \mathbb S_6$ shows that in this case $w^3$ is the longest element in $W_{\Me}$. Then, $w_0=s_0w^3$ whence  \eqref{eq:e6-l=5} forces $w_0\cdot t=t$, that is, $\zeta_1\zeta_6=\zeta_3\zeta_5=\zeta_4^2=1$. On the other hand,  $t\in \T^{F_w}$ 
gives $t=s_1s_6s_3s_5s_4\cdot t^q$. Comparing the coefficient of $\alpha_4^\vee$ 
on both sides we get  $\zeta_4=\zeta_2\zeta_4^{-1}\zeta_3\zeta_5=-\zeta_4$, 
a contradiction. Hence,  $w_0\cdot t\neq t$ and \ref{W} holds also for $l=5$. 

\medbreak
We finally consider the case in which $J_w$ is of type $rA_1$ with $r\in\{1,2,3\}$.
Without loss of generality $J_w$ is respectively $\{s_1\}$, $\{s_1,s_6\}$ or $\{s_1,s_2,s_6\}$. We take $\Phi_{\Me}$ with base $\{\alpha_1,\alpha_3\}$. Then, $M=[M,M]\simeq \SL_3(q)$, and $w$ acts on $\Phi_{\Me}$ as an element of $W_{\Me}$, so Lemma \ref{lem:perfect-S} gives \ref{S}.
Lemma \ref{lem:Jw} gives $\alpha_1(t)\neq 1$, so \ref{R} follows from Lemma \ref{lem:non-centrality}. 

\medbreak
We next verify \ref{W}. Here, we have $\Cent_{W_{\Me}}(w)=\langle s_1\rangle$, 
$s_6\in \Cent_W(w)$
and $s_6\cdot \Delta_{\Me}=\Delta_{\Me}$, so $\Cent_M(t)\simeq \Cent_M(s_6\cdot t)$. 
Assume that there exist $a\in\{0,1\}$ and  $z\in \bZ$ such that
\[s_6\cdot t=z s_1^a\cdot t.\]
Then, $s_6\cdot t=s^3_6\cdot t=s_1^{a}\cdot t$. If $r$ is $2$ or $3$, 
this  contradicts Lemma \ref{lem:Jw} and we are done. 
Let $r=1$ and $s_6\cdot t=s_1^{a}\cdot t$. 
Comparing the coefficients of $\alpha_6^\vee$ on both sides gives $s_6\cdot t=t$, 
so $\alpha_6(t)=1$. Repeating the argument with $s_2$ instead of $s_6$ 
shows that condition \ref{W} may fail only if $\alpha_6(t)=\alpha_2(t)=1$. 
Assume this is the case. Repeating the argument with $s_0$, using the 
coefficient of $\alpha_2^\vee$ we see that condition \ref{W} may fail 
only if $\zeta_2=\pm1$. If $\zeta_2=1$, then 
\[t\in\langle \U_{\pm\gamma},\gamma\in\Delta\setminus\{\alpha_2\}\rangle^{F_w}\simeq \SL_6(q)\]
and we invoke \cite[Theorem I]{ACG-VII} and Remark \ref{obs:zeta=1} to get the claim. 

\medbreak
If $\zeta_2=-1$ then $1=\alpha_2(t)=\zeta_2^2\zeta_4^{-1}$, so $\zeta_4=1$ and 
\[t\in\langle \U_{\pm\gamma},\gamma\in\Delta\setminus\{\alpha_4\}\rangle^{F_w}.\]
The latter is isomorphic to product of $\SL_2(q)$ with a central product of two copies of $\SL_3(q)$. The factor in the component of $\SL_3(q)$ corresponding to $\alpha_1$ and $\alpha_3$ is non-central and not irreducible. Then we invoke Remark \ref{obs:zeta=1} and \cite[Theorem 1.1]{ACG-III} to get the claim.
\end{proof}

According to  \cite[Table B.4]{geck-pfeiffer} there are 5 cuspidal classes in $W$. Following \cite{carter-cc}, they are labeled by $E_6$ (the Coxeter class), $E_6(a_1)$, $E_6(a_2)$, $A_5+ A_1$ and $3A_2$. The label indicates the type of a minimal Weyl subgroup of $W$ containing a class representative.

\begin{lema}\label{lem:e6-cuspidal}If $w$ is cuspidal, then $\Oc$ is of type C, D, or F.
\end{lema}
\begin{proof}We recall that in this case $\Cent_{\G}(t)=\T$, and so $\Cent_{\G}(\sigma\cdot t)=\T$ for all $\sigma\in W$. Therefore \ref{R} and \ref{W2} are always granted. Hence, in all cases we just need to find $\Phi_{\Me}$ satsfying \ref{S} and \ref{W1}. 

 We discuss the different cuspidal classes separately, making use of the representatives, their order, and the sizes of their centralisers extracted from \cite[Table B.4, Table B.8]{geck-pfeiffer}. 

\medbreak
\noindent{\bf Case $E_6$, $w=s_1s_2s_3s_4s_5s_6$.} Here $|w|=12$. Let 
\begin{align*}\beta_1\coloneqq \alpha_1+\alpha_3+\alpha_4,&&\beta_2\coloneqq \alpha_2, &&\beta_3\coloneqq \alpha_3+\alpha_4+\alpha_5,&&\beta_4\coloneqq \alpha_4+\alpha_5+\alpha_6.\end{align*}
Then 
\begin{align*}w\beta_1=\beta_2+\beta_3,&&w\beta_2=\beta_1, &&w\beta_3=\beta_2+\beta_4,&&w\beta_4=-\beta_1-\beta_2,\\
s_{\beta_1}s_{\beta_2}w\beta_1=\beta_3,&&s_{\beta_1}s_{\beta_2}w\beta_2=\beta_2,&&s_{\beta_1}s_{\beta_2}w\beta_3=\beta_4,&&s_{\beta_1}s_{\beta_2}w\beta_4=\beta_1.\end{align*}
Hence $w$ stabilises the root subsystem $\Phi_{\Me}$ with base $\{\beta_1,\beta_2,\beta_3,\beta_4\}$, of type $D_4$, and $s_{\beta_1}s_{\beta_2}w$ acts on $\Phi_{\Me}$ as a graph automorphism of order $3$. 
Then condition \ref{S} follows from Lemma \ref{lem:cyclically-permuted}. 

\medbreak
We verify \ref{W1}. In this case $\Cent_W(w)=\langle w\rangle$ and the coset $wW_{\Me}$ in $\langle w, W_{\Me}\rangle$ has order $3$, so $\Cent_{W_{\Me}}(w)=\langle w^3\rangle$. If  $w\cdot t=z w^{3a}\cdot t$ for some $z\in \bZ$ and some $a\in\I_{0,3}$, regularity of $t$ forces $|z|=3$, which can hold only if $q\equiv 1\mod 3$.  Since
$w^3\cdot t=w^{9a}\cdot t$,  regularity  would also give $12|9a-3$, that is, $4|3a-1$, forcing $a=3$. Hence, $w^4\cdot t=zt$.  A direct calculation shows that 
\begin{align*}
 w\cdot t=&\alpha_1^\vee(\zeta_2\zeta_6^{-1})\alpha_2^\vee(\zeta_3\zeta_6^{-1})\alpha_3^\vee(\zeta_1\zeta_2\zeta_6^{-1})\alpha_4^\vee(\zeta_2\zeta_3\zeta_6^{-1})\alpha_5^\vee(\zeta_4\zeta_6^{-1})\alpha_6^\vee(\zeta_5\zeta_6^{-1}),\\
 w^2\cdot t=&\alpha_1^\vee(\zeta_3\zeta_5^{-1})\alpha_2^\vee(\zeta_1\zeta_2\zeta_5^{-1})\alpha_3^\vee(\zeta_2\zeta_3\zeta_5^{-1}\zeta_6^{-1})\alpha_4^\vee(\zeta_1\zeta_2\zeta_3\zeta_5^{-1}\zeta_6^{-1})\times\\
&\alpha_5^\vee(\zeta_2\zeta_3\zeta_5^{-1})\alpha_6^\vee(\zeta_4\zeta_5^{-1}),\\
w^4\cdot t=&\alpha_1^\vee(\zeta_6^{-1})\alpha_2^\vee(\zeta_1\zeta_5^{-1})\alpha_3^\vee(\zeta_1\zeta_2\zeta_4^{-1}\zeta_6^{-1})\alpha_4^\vee(\zeta_1\zeta_2\zeta_3\zeta_4^{-1}\zeta_5^{-1}\zeta_6^{-1})\times\\
&\alpha_5^\vee(\zeta_1\zeta_2\zeta_5^{-1}\zeta_6^{-1})\alpha_6^\vee(\zeta_1\zeta_6^{-1}),
\end{align*}
But then $w^4\cdot t=zt$ gives $t^3=e$, but for $q\equiv 1\mod 3$ this gives $t\in\T^F$, against the assumption that  $e\not\in\WO$. 

\medbreak
\noindent{\bf Case $E_6(a_1)$.} By \cite[Table B.8]{geck-pfeiffer}   the characteristic polynomial $c_w(X)$ of $w$ is the $9$-th cyclotomic polynomial $\varphi_9(X)=X^6+X^3+1$ and $|\Cent_W(w)|=9$. Then $c_{w^3}(X)=\varphi_3(X)^3=(X^2+X+1)^3$. Thus, $w^3$ is cuspidal and by inspection  it is necessarily labeled by $3A_2$. Let $W_1$ be a Weyl subgroup of $W$ containing $w^3$ and let $\Phi_1$ be the corresponding root  subsystem of $\Phi$. By \cite[Table 9]{carter-cc} the set $A$ of  root subsystems of $\Phi$ that are $W$-conjugate to $\Phi_1$ and contain $w^3$ contains precisely $4$ elements. The group $\langle w\rangle/\langle w^3\rangle$ acts on $A$. By order reasons, $w$ fixes an element $\Phi'$ of $A$ and we set $\Phi_{\Me}\coloneqq \Phi'$. 
By construction $\Phi_{\Me}$ is of type $3A_2$. Since $c_w(X)$ is irreducible, $w$ permutes the components without preserving any of them, and $w^3\in W_{\Me}$. Thus, $[\Me,\Me]$ is a central product of  $3$ copies of $\SL_3(\ku)$ and  \ref{S} follows from Lemma \ref{lem:cyclically-permuted}. Since $\Cent_W(w)=\langle w\rangle$, we have $\Cent_{W_{\Me}}(w)=\langle w^3\rangle$. Assume that $w\cdot t=z w^{3a}\cdot t$ for some $a\in\I_{0,2}$ and some $z\in \bZ$. Then, $w^3\cdot t=w^{9a}\cdot t=t$, contradicting regularity of $t$. This gives \ref{W1}.

\medbreak

\noindent{\bf Case $E_6(a_2)$, $w=s_1s_2s_3s_1s_4s_2s_3s_4s_5s_4s_6s_5$.} Here $|\Cent_W(w)|=2^3\cdot 3^2$ and $|w|=6$. Let
\begin{align*}
&\gamma_1\coloneqq \alpha_1+\alpha_3,\\
&\gamma_2\coloneqq w\gamma_1=\alpha_1+\alpha_2+\alpha_3+2\alpha_4+\alpha_5,\\
&\gamma_3\coloneqq w\gamma_2=\alpha_2+\alpha_3+2\alpha_4+\alpha_5+\alpha_6,\\
&\gamma_4\coloneqq w\gamma_3=\alpha_5+\alpha_6.
\end{align*}
A direct calculation shows that $w\gamma_4=-\gamma_1-\gamma_3$ and that $\{\gamma_1,\gamma_3\}\perp\{\gamma_2,\gamma_4\}$ so the root system $\Phi_{\Me}$ generated by $\{\gamma_1,\,\gamma_2,\gamma_3,\gamma_4\}$ is a $w$-stable subsystem of type $2A_2$ whose components are interchanged by $w$. Lemma \ref{lem:cyclically-permuted} gives  \ref{S}. 

 \medbreak
Condition \ref{W} follows from Lemma \ref{lem:index} because \begin{align*}|\bZ||\Cent_{W_{\Me}}(w)|\leq|\bZ||W_{\Me}|\leq 18<72=|\Cent_W(w)|.\end{align*}

\medbreak
\noindent{\bf Case $A_5+A_1$.} By \cite[Proposition 32]{carter-cc} (following \cite[p. 142]{dynkin}) we may assume that 
$w=s_1s_6s_3s_5s_4 s_0$, which is  contained in the Weyl subgroup
\[W'\coloneqq \langle s_1,s_3, s_4,s_5, s_6,s_0\rangle\]
of $W$ of type $A_5 + A_1$.  We take $\Delta_{\Me}\coloneqq \{\alpha_1,\alpha_3,\alpha_4,\alpha_5,\alpha_6\}$. Therefore $M\simeq\SL_6(q)$, giving \ref{S}, 
and $\Cent_{W_{\Me}}(w)=\langle s_1s_6s_3s_5s_4 \rangle$, 
which has order $6$. 
Arguing as we did in the proof of the case $A_5$ in Lemma \ref{lem:e6-non-cuspidal}, 
we see that the equality $s_0\cdot t=z (s_1s_6s_3s_5s_4 )^a\cdot t$ 
for some $z\in \bZ$ and some $a\in\I_{0,5}$ could  hold only if $s_0\cdot t=(s_1s_6s_3s_5s_4)^3\cdot t$. But when this happens, we have
$(s_1s_6s_3s_5s_4)^4\cdot t^q=w\cdot t^q=t$, contradicting the cuspidality of $t$.  

\medbreak
\noindent{\bf Case $3A_2$.} Here $|\Cent_W(w)|=2^3\cdot 3^4$.
By \cite[Proposition 32]{carter-cc} (following \cite[p. 142]{dynkin}) we may assume that 
$w=s_1s_3s_0s_2 s_5s_6$, which is  contained in the Weyl subgroup $W'\coloneqq \langle s_1,s_3, s_0,s_2, s_5,s_6\rangle$ of $W$ of type $3A_2$. We take $\Delta_{\Me}\coloneqq \{\alpha_1,\alpha_3\}$ so $M\simeq\SL_3(q)$ satisfies \ref{S}.  
Lemma \ref{lem:index} and the estimate  
\begin{align*}|\Cent_W(w)|>|\bZ||\Cent_{W_{\Me}}(w)|=3|\langle s_1s_3\rangle|=9\end{align*}
imply that Condition \ref{W} holds.
\end{proof}

We are then in a position to state a general result for $\Gb$ of type $E_6$.

\begin{theorem}\label{thm:e6-collapses}Every non-trivial class in a simple Chevalley group of type $E_6$ is of type C, D, or F. 
\end{theorem}
\begin{proof}
For unipotent, respectively mixed conjugacy classes this is \cite[Main Theorem]{ACG-IV}, respectively 
\cite[Theorem 1.1]{ACG-V}. For split semisimple conjugacy classes we rely on \cite[Theorem I]{ACG-VII}. Non-split semisimple classes are covered by Lemmata \ref{lem:e6-non-cuspidal} and \ref{lem:e6-cuspidal}. 
\end{proof}

\subsection{Semisimple classes in \texorpdfstring{$E_7(q)$}{}} \label{subsec:E7}

\

In this Subsection  $\G=\G_{sc}$ is of type $E_7$, so $w_0$ acts as $-\id$ on $\Phi$. The center of $\G^F$ is described in Table \ref{tab:center}.

\medbreak

We will make use of the following formulas
\begin{align}\label{eq:e7-s0}
\alpha_0&=2\alpha_1+2\alpha_2+3\alpha_3+4\alpha_4+3\alpha_5+2\alpha_6+\alpha_7;\\  
\nonumber s_0\cdot t&=\alpha_1^\vee(\zeta_1^{-1})\alpha_2^\vee(\zeta_2\zeta_1^{-2})\alpha_3^\vee(\zeta_3\zeta_1^{-3})
\alpha_4^\vee(\zeta_4\zeta_1^{-4})\alpha_5^\vee(\zeta_5\zeta_1^{-3})\\
\nonumber & \qquad \times \alpha_6^\vee(\zeta_6\zeta_1^{-2}) \alpha_7^\vee(\zeta_7\zeta_1^{-1}). 
\end{align}

\begin{obs}\label{obs:reduction-to-d6} 
Comparing coefficients, the equality $s_0\cdot t=zt$, with $z\in \bZ$ may occur only if $z=e$ and $\zeta_1=1$. 
Assume that this is the case. Then
\begin{enumerate}[leftmargin=*,label=\rm{(\roman*)}]
\item $\langle \U_{\pm\alpha_0}\rangle\leq \Cent_{\G}(t)$, so $t$ is not regular, and

\medbreak
\item $t\in\Ha:=\langle \U_{\pm\alpha},\,\alpha\in\Delta\setminus\{\alpha_1\}\rangle$, a simple algebraic  group of type $D_6$.  
If $\Ha$ is $w$-suitable, then $\Ha^{F_w}$ is a Chevalley group 
and the class of $\pi(t)$ is covered  either by \cite[Theorem 6.2]{ACG-VII} 
or by Theorem \ref{thm:dn-one-exception}; therefore $\Oc$ is of type C, D, or F.
\end{enumerate}
\end{obs}

We verify the hypotheses of Theorem \ref{thm:potente-revised}  considering the possible choices of $J_w$, up to $W$-conjugacy. 

\begin{lema}\label{lem:e7-non-cuspidal}If $w$ is not cuspidal, then $\Oc$ is of type C, D, or F.
\end{lema}
\begin{proof}
By assumption $J_w$ is not of type $E_7$. 

\medbreak
If $J_w$ has a component of type $D_d$, with $d\geq 4$, then Lemma \ref{lem:Jw-contains-D}
applies. 

\medbreak
If $J_w$ is of type $A_3+A_2+A_1$, $A_4+A_2$, 
$A_3+2A_1$, $A_3+A_2$, $3A_2$, $A_2+ 3A_1$, $2A_2$ or $A_2+2A_1$, then we invoke Lemma \ref{lem:levi}. 

\medbreak
If $J_w$ is of type $A_6$, $A_4+A_1$, $A_4$, $A_2+A_1$, or $A_2$, then 
Lemma \ref{lem:A-odd} applies. 

\medbreak
Assume $J_w$ is of type $A_1+A_l$ with  $l=3,5$. According to \cite[2.2]{sommersBC}, there is only one $W$-orbit of subsets of $S$ of type $A_1+A_5$, whereas there are two $W$-orbits of subsets of $S$ of type $A_1+A_3$. In all cases we can choose $J_w$ so that the component of type $A_1$ is $\{s_1\}$.  We take $\Phi_{\Me}$ to be the root subsystem corresponding to $A_l$. Then $M\simeq\SL_{l+1}(q)$ so \ref{S} holds. Lemma \ref{lem:Jw} guarantees that $\alpha_1(t)\neq1$, and the conditions \ref{R} and \ref{W1} for $\sigma=s_1$.  Then $s_1\cdot t\not\in \bZ W_{\Me}\cdot t$ is readily verified comparing the coefficients of $\alpha_1^\vee$. 

\medbreak
Suppose that $J_w$ is of type $A_l$ with  $l=3,5$. Then, $J_w$ is $W$-conjugate to 
$\{s_5,s_6,s_7\}$, or  $\{s_2,s_4,s_5,s_6,s_7\}$, or $\{s_3,s_4,s_5,s_6,s_7\}$ by
\cite[2.2]{sommersBC}.  In all cases we take $\Phi_{\Me}=\Phi_{J_w}$, giving \ref{S} and \ref{R}. Here $w$ is a Coxeter element in $W_{\Me}$, so $\Cent_{W_{\Me}}(w)=\langle w\rangle$ and $|w|=l+1$. 

\medbreak
Assume first  that $J_w$ is conjugate to $\{s_5,s_6,s_7\}$ or  $\{s_2,s_4,s_5,s_6,s_7\}$.  If $\alpha_1(t)\neq1$, then we verify \ref{W} as we did for $J_w$ of type $A_1+A_3$ or of type $A_1+A_5$. If $\alpha_1(t)=\zeta_1^2\zeta_3^{-1}=1$, then either $\zeta_3\neq1$, whence  $w_0\cdot t\in \bZ \Cent_{W_{\Me}}(w)$ and \ref{W} follows from Remark \ref{obs:w0}, or else $\zeta_3=1$. In the latter case $t\in \langle \U_{\pm\alpha},\alpha\in \Delta\setminus\{\alpha_3\}\rangle^{F_w}\simeq\SL_2(q)\times\SL_6(q)$ and the component in $\SL_6(q)$ is non-trivial. Then we apply Remark \ref{obs:zeta=1} and \cite[Theorem I]{ACG-VII}. 

\medbreak
Assume now that $J_w$ is conjugate to $\{s_3,s_4,s_5,s_6,s_7\}$, so $|w|=6$. Then, $s_0\in \Cent_W(w)$ and $
\Cent_{\Me}(s_0\cdot t)=\Cent_{\Me}(t)=\T$ by Lemma \ref{lem:Jw}. Assume for a contradiction that $s_0\cdot t=z w^a\cdot t$ for some $z\in \bZ$ and some $a\in\I_{0,5}$. Comparing the coefficients of $\alpha_1^\vee$ and $\alpha_2^\vee$ gives $\zeta_1^2=1$ and  $z=e$. Also, $t=w^{2a}\cdot t$, so regularity of $t$ forces $a\in\{0,3\}$. 

\medbreak
If $a=0$, then necessarily $z=e$ and $\zeta_1=1$  and we invoke Remark \ref{obs:reduction-to-d6}. Then, $t$ lies in a group of type $D_6$, it is non-split therein, and $w$ lies in a parabolic subgroup of type $A_5$. 
Then we conclude by \cite[Theorem 6.2]{ACG-VII}.

\medbreak
Let us thus assume that $a=3$. If $w_0\cdot t=z' w^b\cdot t$ for some $z'\in \bZ$ and some $b\in\I_{0,5}$, then $t=w^{2b}\cdot t$, so $b\in\{0,3\}$. Remark \ref{obs:w0} \ref{item:w0-due} forces $b=3$ and so $s_0\cdot t=w^3\cdot t=z' w_0\cdot t$. 
Comparing coefficients making use of $\zeta_1^2=1$ and \eqref{eq:e7-s0} gives $t^4=e$. Since $t$ is non-split, this could occur only if $t^2\neq e$ and $q\equiv 3\mod 4$. In this case,  $t=w\cdot t^q=w\cdot t^3$, so $w^2\cdot t=t$, a contradiction. 

\medbreak
Let now $J_w$ be of type $rA_1$, with $r\in\I_{4}$. For each $r$ there is only one $W$-orbit of subsets of $S$, so we may assume that $\{s_7\}\subseteq J_w\subseteq\{s_2,s_3,s_5,s_7\}$. We take $\Delta_{\Me}=\{\alpha_5,\alpha_6,\alpha_7\}$. 
Then $M\simeq \SL_4(q)$, giving \ref{S}. By Lemma \ref{lem:Jw} we have $\alpha_7(t)\neq 1$ so Lemma \ref{lem:non-centrality} gives \ref{R}. Now, $\alpha_0\perp \Phi_{J_w}$ so $s_0\in \Cent_W(w)$ and $\alpha_7(s_0\cdot t)\neq 1$. Then, \ref{W2} with $\sigma=s_0$ follows from Lemma \ref{lem:non-centrality}. Assume that $s_0\cdot t\in\bZ \Cent_{W_{\Me}}(w)\cdot t$. Remark \ref{obs:coefficient} applied to $\alpha_1$ and $\alpha_3$ yields $\zeta_1=1$. Hence, either \ref{W1} holds, or $s_0\cdot t=t$. In this case, Remark \ref{obs:reduction-to-d6} applies, invoking \cite[Theorem 6.2]{ACG-VII} if $\{\alpha_2,\alpha_3\}\not\subset J_w$, and Theorem \ref{thm:dn-one-exception} otherwise. 

\medbreak 
Assume finally that $J_w$ is of type $E_6$. In the proof of Lemma \ref{lem:e6-cuspidal}, for each cuspidal class in $\langle J_w\rangle$ we have exhibited a $w$-stable root subsystem $\Phi_{\Me}$ of $\Phi_{J_w}$ such that $[\Me,\Me]$ satisfies \ref{S}. Condition \ref{R} follows from Lemma \ref{lem:Jw} because $[\Me,\Me]\leq\Le_{J_w}$. In all cases
\begin{align*}|\Cent_W(w)|\geq |\Cent_{\langle J_w\rangle}(w)|>3|\Cent_{W_{\Me}}(w)|\end{align*}  so \ref{W} follows from Lemma \ref{lem:index}. 
\end{proof}

Let us now look at the cusp classes in $W$, of which there are 12. 
A representative, its order, the size of its centraliser, and the characteristic polynomials $c_w(X)$ of any such class can be found in \cite[Table B.5, B.8]{geck-pfeiffer}. 

\begin{lema}\label{lem:e7-cuspidal}If $w$ is cuspidal, then $\Oc$ is of type C, D, or F.
\end{lema}
\begin{proof}If $w$ is cuspidal then  $\Cent_{\G}(t)=\T$ and so $\Cent_{\G}(\sigma\cdot t)=\T$ for all $\sigma\in W$. Hence, \ref{R} and \ref{W2} are always granted. Therefore,  it is enough to find for any cuspidal class a root system $\Phi_{\Me}$ such that $\pi([M,M])$ is quasi-simple and $\Cent_{W}(w)\cdot t\not\subset \bZ \Cent_{W_{\Me}}(w)\cdot t$. 

\medbreak
\noindent{\bf Cases $2A_3+A_1,\,D_4+3A_1,\,A_5+A_2$.} We take $\Phi_{\Me}$ to be the component of type respectively $A_3$, $D_4$, or $A_5$ of $\Phi_{J_w}$ so \ref{S} is guaranteed by Lemma \ref{lem:perfect-S}. 
Since the order of the Weyl group of the complementary components $A_3+A_1$, $3A_1$, and $A_2$ in $\Phi_{J_w}$ is  greater than $2$, condition \ref{W} follows from Lemma \ref{lem:index}. 

\medbreak
\noindent{\bf Case $7A_1$, $w=w_0$.} We take $\Delta_{\Me}=\{\alpha_1,\alpha_3,\,\alpha_4\}$ so $w_0$ acts on $\Phi_{\Me}$ as the composition an element in $W_{\Me}$ and a non-trivial graph automorphism. Then, $M\simeq\SU_4(q)$, so it is quasi-simple, yielding \ref{S}, and \ref{W} follows from Lemma \ref{lem:index} and the estimate
\begin{align*}|\Cent_W(w_0)|=|W|=2^{10}\cdot 3^4\cdot 5\cdot 7>2\cdot 4!>|\bZ||W_{\Me}|.\end{align*}

\medbreak

\noindent{\bf Case $A_7$, $w=s_0s_1s_3s_4s_5s_6s_7$, $|w|=8$.}  We take $\Phi_{\Me}$ to be the root subsystem of type $4A_1$ with base 
\begin{align*}
\gamma_1 &\coloneqq \alpha_1+\alpha_3+\alpha_4+\alpha_5,&\gamma_2 &\coloneqq w\gamma_1= \alpha_3+\alpha_4+\alpha_5+\alpha_6,\\
\gamma_3&\coloneqq w^2\gamma_1=\alpha_4+\alpha_5+\alpha_6+\alpha_7,&
\gamma_4&\coloneqq w^3\gamma_1= \alpha_0-(\alpha_1+\alpha_3+\alpha_4).
\end{align*}
Since $w\gamma_4=-\gamma_1$, condition \ref{S} follows from Lemma \ref{lem:cyclically-permuted}. 
Furthermore, $|W_{\Me}|\leq 2$, so $|\Cent_W(w)|\geq|\langle w\rangle|=8>|\bZ||\Cent_{W_{\Me}}(w)|$, giving \ref{W} by virtue of Lemma \ref{lem:index}. 

\medbreak

\noindent{\bf Cases $D_6+A_1$ and $D_6(a_1)+A_1$.}   We may assume
without loss of generality that
$w=s_0w_1$ where $w_1$ is cuspidal in $W_1\coloneqq \langle s_2,s_3,s_4,s_5,s_6,s_7\rangle$. We take $\Me$ to be the Levi subgroup associated with $\{\alpha_2,\alpha_3,\alpha_4, \alpha_5,\alpha_6,\alpha_7\}$. Then 
\[[M,M]/Z([M,M])\simeq \Pom^+_{12}(q) \qquad\text{ and }\qquad W_{\Me}=W_1.\]
We verify \ref{W}, using that $s_0\in \Cent_W(w)\setminus W_{\Me}$. 

\medbreak
If $w$ is of type $D_6+A_1$, then $w_1$ lies in the Coxeter class of $W_1$, so $\Cent_{W_{\Me}}(w)=\Cent_{W_{\Me}}(w_1)=\langle w_1\rangle$ by \cite[Proposition 30]{carter-cc}.
Assume for a contradiction that $s_0\cdot t=zw_1^a\cdot t$ for some $z\in \bZ$ and some $a\in\I_{0,9}$. Regularity of $t$ and Remark \ref{obs:reduction-to-d6} force $a\neq0$. Then $t=w_1^{2a}\cdot t$ and Lemma \ref{lem:Jw} forces $a=5$. Since $w_1^5$ is the longest element of the parabolic subgroup $W_1$, it follows that $w_0=s_0 w_1^5$, therefore, $w_0\cdot t=zt$. However, this contradicts Remark  \ref{obs:w0} \ref{item:w0-due}.

\medbreak
If $w$ is of type $D_6(a_1)+A_1$ then by order reasons $w_1$ lies in the cuspidal class of $W_1$ corresponding to the partition $(4,2)$. Assume that $s_0\cdot t=z\sigma\cdot t$ for some $z\in \bZ$ and some $\sigma\in \Cent_{W_{\Me}}(w)$. Regularity of $t$ and Remark \ref{obs:reduction-to-d6} force $\sigma\neq e$ and $\sigma^{2}=e$. Using the inclusion $W_1\leq \mathbb S_{12}$ we see that the only non-trivial involution in
$\Cent_{W_{\Me}}(w)=\Cent_{W_{\Me}}(w_1)=\Cent_{\mathbb S_{12}}(w_1)\cap W_1$ is the longest element in $W_1$. Then we argue as we did in the case $D_6+A_1$. 
 
\medbreak
\noindent{\bf Case $E_7$, the  Coxeter class.} In this case, $|w|=18$ and $\Cent_W(w)=\langle w\rangle$. As $w_0\in Z(W)$, it is necessarily the unique involution in $\Cent_W(w)$, that is $w_0=w^9$. The group $\langle w\rangle/\langle w_0\rangle$ acts on the set $A$ consisting of the $W$-orbits of root subsystems of $\Phi$ of type $E_6$. By \cite[Lemma 34, Table 10]{carter-cc} we have $|A|=28$,  so $w$ necessarily preserves at least one of these root systems, that we set to be $\Phi_{\Me}$. Then,  \ref{S} follows from Lemma \ref{lem:cyclically-permuted}. Since $w$ is cuspidal, it is the product of $7$ reflections with respect to linearly independent roots, hence  $w\not\in W_{\Me}$.  
Therefore, either $|\bZ||\Cent_{W_{\Me}}(w)|<9|\bZ|=|\Cent_W(w)|$, or else $\Cent_{W_{\Me}}(w)=\langle w^2\rangle$. In the latter case, if $w_0\cdot t=z w^{2a}\cdot t$ for some $z\in \bZ$ and some $a\in\I_{0,8}$, then $a=0$ is excluded by Remark \ref{obs:w0} \ref{item:w0-due}. Applying $w_0$ on both terms gives $w^{4a}=e$, so $18|4a$, forcing $a=9$, a contradiction. 

\medbreak
\noindent{\bf Case $E_7(a_1)$.} In this case, $|w|=14$ and $\Cent_W(w)=\langle w\rangle$, so $w^7=w_0$.
Let $\Phi_1$ be a root subsystem of type $7A_1$ such that $w_0$ lies in its Weyl group. According to  \cite[Lemma 35, Table 10]{carter-cc} the set $A$ of  root subsystems of $\Phi$ that are $W$-conjugate to $\Phi_1$  and whose Weyl group contains $w_0$ has cardinality $135$. By order reasons $\langle w\rangle/\langle w^7\rangle$ acts on $A$ and fixes necessarily an element therein. We take $\Phi_{\Me}$ to be such a root subsystem. In addition, $\langle w\rangle/\langle w^7\rangle$ acts on the set of components of $\Phi_{\Me}$ and it cannot fix all of them, for in this case $w^2= e$. Therefore, $w$ must permute these components cyclically. Lemma \ref{lem:cyclically-permuted} gives \ref{S}. Then \ref{W} follows from Lemma \ref{lem:index} because 
 \begin{align*}|\Cent_W(w)|=14>4=|\bZ||\Cent_{W_{\Me}}(w)|.\end{align*}

\medbreak
\noindent{\bf Case $E_7(a_2)$.} Here, $|w|=12$ and  $[\Cent_W(w):\langle w\rangle]=2$. The characteristic polynomials of $w$ is  
$c_w(X)=\varphi_2(X)\varphi_6(X)\varphi_{12}(X)$, so $c_{w^3}(X)=\varphi_2(X)^3\varphi_4(X)^2$. By inspection $w^3$ is cuspidal and lies in the class $2A_3+A_1$. It is therefore of the form $w^3=s_\alpha w_1w_2$, where $w_1$ and $w_2$ are Coxeter elements with respect to orthogonal root systems $\Phi_1$ and $\Phi_2$ of type $A_3$, and $\alpha\perp\Phi_i$ for $i=1,2$. Therefore, $w^6=w_1^2w_2^2$ is the product of four reflections with respect to mutually orthogonal roots, say $\gamma_i$ for $i\in\I_{4}$.
Then, $w$ permutes the sets $\{\pm\gamma_i\}$, with at most one orbit of size $1$, because $w$ has no eigenspace of eigenvalue $1$ and a $1$-dimensional  eigenspace of eigenvalue $-1$. We take $\Phi_{\Me}$ to be generated by the subsets $\{\pm\gamma_i\}$ lying in an orbit of size $d>1$. Lemma \ref{lem:cyclically-permuted} gives \ref{S}, and \ref{W} follows from Lemma \ref{lem:index} because
\begin{align*}
|\Cent_W(w)|=24>2\cdot 2\geq |\bZ||{W_{\Me}}(w)|\geq |\bZ||\Cent_{W_{\Me}}(w)|.\end{align*}

\medbreak

\noindent{\bf Case $E_7(a_3)$.} Here $|w|=30$ and $\Cent_W(w)=\langle w\rangle$. By  
\cite[Table B.8]{geck-pfeiffer} the characteristic polynomial of $w^2$ is 
$c_{w^2}(X)=\varphi_1(X)\varphi_3(X)\varphi_5(X)$, so $w^2$ is cuspidal in a Weyl subgroup $W'$ of rank $6$. Inspection of \cite[Table 10]{carter-cc} shows that $W'$ is of type $A_2+ A_4$ and that there is no other subgroup that is $W$-conjugate to $W'$ and contains $w^2$. 
Hence, $w$ stabilizes a root system of type $A_2+A_4$, cf. \cite[Lemma 34]{carter-cc}. 
Let $\Phi_{\Me}$  be the component of type $A_4$ therein; we have \ref{S} by Lemma \ref{lem:cyclically-permuted}. If $w_1$ is the factor of $w^2\in W'$ in the component of type $A_4$, then $\Cent_{W_{\Me}}(w)=\langle w_1^2\rangle$ and
\[ |\Cent_W(w)|=30> 2\cdot 5=|\bZ||\Cent_{W_{\Me}}(w)|. \]
Then Lemma \ref{lem:index} applies, giving \ref{W}.

\medbreak
\noindent{\bf Case $E_7(a_4)$.} Comparing the characteristic polynomials we deduce that $w=w_1^3$ for some $w_1$ in the class labeled by $E_7$. We take the same root subsystem $\Phi_{\Me}$ we used for the class of $w_1$. Then, $w_1\in \Cent_W(w)\setminus \Cent_{W_{\Me}}(w)$. Assume that $w_1\cdot t=z\sigma\cdot t$ for some $z\in \bZ$ and some $\sigma\in \Cent_{W_{\Me}}(w)$. Applying $w_1$ on both sides gives $w_1^2\cdot t=\sigma^2\cdot t$, so  $w_1^2=\sigma^2\in W_{\Me}$ by regularity of $t$. Since $c_{w_1}(X)=\varphi_2(X)\varphi_{18}(X)$, there holds $c_{\sigma^2}(X)=\varphi_1(X)\varphi_9(X)$, so $\sigma^2$ has rank $6$ and it is therefore cuspidal in $W_{\Me}$. By inspection, it lies in the class $E_6(a_1)$ therein, so $|\sigma^2|=9$. Then, also $\sigma$ has to be cuspidal in $W_{\Me}$ and by order reason, it is necessarily conjugate to $\sigma^2$. But then, 
\begin{align*}
 w_0\cdot t=w_1^9\cdot t=z^9\sigma^9\cdot t =zt,  
\end{align*}
contradicting Remark \ref{obs:w0} \ref{item:w0-due}.
 \end{proof}

We are then in a position to state a general result for $\Gb$ of type $E_7$. 
\begin{theorem}\label{thm:e7-collapses}Every non-trivial class in a simple Chevalley group of type $E_7$ is of type C, D, or F. 
\end{theorem}
\begin{proof}
For unipotent, respectively mixed classes this is \cite[Main Theorem]{ACG-IV}, respectively \cite[Theorem 1.1]{ACG-V}. For split semisimple classes this is \cite[Theorem I]{ACG-VII}. The remaining classes are covered by Lemmata \ref{lem:e7-non-cuspidal} and  \ref{lem:e7-cuspidal}.
\end{proof}

\subsection{Semisimple classes in \texorpdfstring{$E_8(q)$}{}} \label{subsec:E8}

\

In this Subsection  $\G$ is of type $E_8$, so $w_0$ acts as $-\id$ on $\Phi$ and $\bZ=e$.  We will make use of the following formulas.
\begin{align*}
\alpha_0 &= 2\alpha_1+3\alpha_2+4\alpha_3+6\alpha_4+5\alpha_5+4\alpha_6+3\alpha_7+2\alpha_8,
\\
s_0\cdot t &= \alpha_1^\vee(\zeta_1\zeta_8^{-2})\alpha_2^\vee(\zeta_2\zeta_8^{-3})\alpha_3^\vee(\zeta_3\zeta_8^{-4})\alpha_4^\vee(\zeta_4\zeta_8^{-6})\alpha_5^\vee(\zeta_5\zeta_8^{-5})
\\
&\qquad \times \alpha_6^\vee(\zeta_6\zeta_8^{-4}) \alpha_7^\vee(\zeta_7\zeta_8^{-3}) \alpha_8^\vee(\zeta_8^{-1}).
\end{align*}

We verify the hypotheses of Theorem \ref{thm:potente-revised}, analyzing separately  the different possibilities for $J_w$, up to $W$-conjugacy.

\begin{lema}\label{lem:e8-non-cuspidal}If $w$ is not cuspidal, then $\Oc$ is of type C, D, or F. 
\end{lema}
\begin{proof}
By assumption $J_w$ is not of type $E_8$. If $J_w$ has a component of type $D_d$, for some $d\geq 4$, then we invoke Lemma \ref{lem:Jw-contains-D}. 

\medbreak
If $J_w$ has more than one component, one of which is of rank $\geq 2$, then we invoke Lemma \ref{lem:levi}. 

\medbreak
If $J_w$ is of type $A_l$ with $l$ even, then we apply Lemma \ref{lem:A-odd}.

\medbreak
There remain the cases where $J_w$ is of type $E_7$, $E_6$, $A_l$ with $l=3,5,7$, or $rA_1$ with $r=1,2,3,4$, and $w$ is cuspidal in $\langle J_w\rangle$. There exists a unique conjugacy class of parabolic subgroups in $W$ for each of these types, see \cite{sommersBC}.

\medbreak
\noindent{\bf Case $rA_1$ with $r=1,2,3,4$.} We may assume that 
$\{s_1\}\subseteq J_w\subseteq S\setminus\{s_3,s_4\}$ and  take $\Phi_{\Me}$ 
to be the root system with base $\{\alpha_1,\alpha_3\}$. 
Then $w$ acts on $\Phi_{\Me}$ as $s_1$, giving \ref{S}  
by Lemma \ref{lem:perfect-S}. 
Lemma \ref{lem:Jw} guarantees that $\alpha_1(t)\neq 1$, so \ref{R} follows from 
Lemma \ref{lem:non-centrality}. 
Then $\Cent_{W_{\Me}}(w)=\langle s_1\rangle$. 
Assume that \ref{W} fails. Then $w_0\cdot t=s_1\cdot t$ because $\id\not\in\WO$. 
Comparing coefficients we get $\zeta_2=1$. 
Hence, $t\in\langle \U_{\pm\alpha},\,\alpha\in \Delta\setminus\{\alpha_2\}\rangle^{F_w}$, isomorphic to $\SL_2(q)\times\SL_6(q)$, and the result follows from 
Remark \ref{obs:zeta=1} and \cite[Theorem 1.1]{ACG-III}. 

\medbreak
\noindent{\bf Cases $E_6$ and $E_7$.} Assume that $J_w$ is of type $E_6$ or $E_7$. In the proofs of Lemmata \ref{lem:e6-cuspidal} and \ref{lem:e7-cuspidal}, for each cuspidal class in $\langle J_w\rangle$ we have exhibited a $w$-stable root subsystem $\Phi_{\Me}$ of $\Phi_{J_w}$ such that $[\Me,\Me]$ satisfies \ref{S}. Condition \ref{R} holds for $t$ because $[\Me,\Me]\leq\Le_{J_w}$. In all cases we have 
\begin{align*}|\Cent_W(w)|\geq |\Cent_{\langle J_w\rangle}(w)|>|\Cent_{W_{\Me}}(w)|.\end{align*} Condition \ref{W} follows from Lemma \ref{lem:index}.

\medbreak

\noindent{\bf Cases $A_3,\, A_5$.} We may assume that $s_7,s_8\not\in J_w$. We take $\Phi_{\Me}=\Phi_{J_w}$ so $M=[M,M]\simeq\SL_4(q)$, respectively 
$\SL_6(q)$, and $\Cent_{\Me}(t)=\T$ by Lemma \ref{lem:Jw}, giving \ref{S} and \ref{R}. Then $\Cent_{W_{\Me}}(w)=\langle w\rangle$.   Assume that \ref{W} fails. We consider 
\[w_0, s_8\in \Cent_W(w)\setminus \Cent_{W_{\Me}}(w).\]
Since $s_8\Phi_{\Me}=\Phi_{\Me}$, we have $\Cent_{\Me}(s_8\cdot t)=\T$. 
Then 
\begin{align*}s_8\cdot t\in\langle w\rangle\cdot t, &&w_0\cdot t\in\langle w\rangle \cdot t.\end{align*} 
Comparing the coefficient of $\alpha_8^\vee$ we see that the first inclusion forces $\zeta_8^2=\zeta_7$ and the second inclusion gives $\zeta_7=\zeta_8^2=1$. Thus, if  \ref{W} fails, then $t\in\langle \U_{\pm\alpha},\alpha\in \Delta\setminus\{\alpha_7\}\rangle^{F_w}$, of type $E_6+A_1$, and the factor in the component of type $E_6$ is non-central. Then we invoke Remark \ref{obs:zeta=1} and Theorem \ref{thm:e6-collapses}. 

\medbreak

\noindent{\bf Case $A_7$.} We may assume that $J_w=\{s_1,s_3,s_4,s_5,s_6,s_7\}$ and choose the representative $w=s_1s_3s_4s_5s_6s_7$. Let 
\begin{align*}
&\gamma_1\coloneqq \alpha_1+\alpha_3+\alpha_4+\alpha_5,&&\gamma_2\coloneqq w\gamma_1=\alpha_3+\alpha_4+\alpha_5+\alpha_6,\\
&\gamma_3\coloneqq w^2\gamma_1=\alpha_4+\alpha_5+\alpha_6+\alpha_7,&&\gamma_4\coloneqq w^3\gamma_1=\alpha_5+\alpha_6+\alpha_7+\alpha_8.
\end{align*}
As $w\gamma_4=-\gamma_1$ and $\gamma_i\perp\gamma_j$ for $i\neq j$, the root system $\Phi_{\Me}$ with base $\{\gamma_i,\, i\in\I_{4}\}$ is $w$-stable, of type $4A_1$, and contained in $\Phi_{J_w}$. Hence, \ref{S} follows from Lemma \ref{lem:cyclically-permuted} and Lemma \ref{lem:Jw} gives $\Cent_{\Me}(t)=\T$, whence \ref{R}. Finally \ref{W} holds because $\Cent_{\Me}(w\cdot t)=\T$ and $w\cdot t\not\in W_{\Me}\cdot t$ by regularity of $t$ in $\Me$.
\end{proof}

There are $30$ cuspidal classes in $W$. A representative for each class is listed in in \cite[Table B.6, B.8]{geck-pfeiffer} together with its order, the size of its centraliser and its characteristic polynomial $c_w(X)$. We will also use the labeling and the information from \cite[Table 11]{carter-cc}.

\begin{lema}\label{lem:e8-cuspidal}If $w$ is cuspidal and not in class $E_8$ or $E_8(a_5)$, then $\Oc$ is of type C, D, or F. 
\end{lema}
\begin{proof}
We verify the hypothesis of Corollary \ref{cor:cuspidal-Ze}. 
If $w$ or $w_0\not\in W_{\Me}$, then the
condition \eqref{eq:cuspidal-Ze}, that is $\Cent_{W_{\Me}}(w)\subsetneq \Cent_{W}(w)$, 
is immediate. 

\medbreak 

\noindent{\bf Cases $4A_2$, $2A_3+2A_1$, $2A_4$, $E_6(a_2)+A_2$, $A_5+A_2+A_1$, $A_7+A_1$, $A_8$, $D_5(a_1)+A_3$, $E_6+A_2$, $2D_4(a_1)$, $D_4+4A_1$, $E_7(a_4)+A_1$, $2D_4$, $D_6+2A_1$, $E_7+A_1$, $E_7(a_2)+A_1$.} Here $w$ lies in the proper Weyl subgroup $W'$ indicated by the label, and we choose $\Phi_{\Me}$ to be the root system of any irreducible component  of rank $\geq 2$ in the root system of $W'$. The action of $w$ on $\Phi_{\Me}$ is inner, and  \ref{S} follows from Lemma \ref{lem:cyclically-permuted}. By inspection in each case $\{w, w_0\}\not\subset W'$.

\medbreak

\noindent{\bf Case $8A_1$, $w=w_0$.} We take $\Delta_{\Me}=\{\alpha_1,\alpha_3,\,\alpha_4\}$ so $w_0$ acts on $\Phi_{\Me}$ via an outer automorphism. Then, $M\simeq\SU_4(q)$ so \ref{S} holds, and 
\begin{align*}|\Cent_W(w_0)|=|W|=2^{14}\cdot 3^5\cdot 5^2\cdot 7>|W_{\Me}|\geq |\Cent_{W_{\Me}}(w)|.\end{align*}
That is, \eqref{eq:cuspidal-Ze} is valid.

\medbreak
\noindent{\bf Cases $D_8$, $D_8(a_1)$, $D_8(a_2)$, and $D_8(a_3)$.} We take $\Phi_{\Me}$ to be the root subsystem of type $D_8$ whose Weyl group $W'$ contains $w$, so Lemma \ref{lem:cyclically-permuted} gives \ref{S}. By construction $w$ is cuspidal in $W'$ and, since it lies in no proper Weyl subgroup of $W'$, it corresponds to a partition $\lambda=(\lambda_1,\lambda_2)$ with only two terms. The inclusion $W'\leq \mathbb S_{16}$ in \cite[Section 3.4]{BL} maps $w$ to a product $c_\lambda$ of $2$ disjoint cycles of size $2\lambda_1$ and $2\lambda_2$. Also, the cycle factors of $c_\lambda$ do not lie in $W_{\Me}$: they only lie in the overgroup of type $B_8$, see Subsection \ref{sec:cuspidal-classical}. This gives the estimate of centraliser sizes 
\[|\Cent_{W_{\Me}}(w)|\leq \frac{|\Cent_{\mathbb S_{16}}(c_\lambda)|}{2}.\]
We collect the relevant data in the table below.

\medbreak
\begin{center}
\begin{tabular}{|c|c|c|c|c|}
\hline
label& $|w|$ & partition $\lambda$ & $|\Cent_W(w)|$ & $|\Cent_{\mathbb S_{12}}(c_\lambda)|$ \\
\hline
$D_8$&$14$&$(7,1)$& $28$& $28$\\
\hline
$D_8(a_1)$&$12$&$(6,2)$& $72$& $48$\\
\hline
$D_8(a_2)$&$30$&$(5,3)$&$60$& $60$\\
\hline
$D_8(a_3)$&$8$&$(4,4)$&$192$& $128$\\
\hline
\end{tabular}
\end{center}
Hence $|\Cent_W(w)|>|\Cent_{W_{\Me}}(w)|$, i.e., \eqref{eq:cuspidal-Ze} follows. 

\medbreak

It remains to consider the case in which $w$ lies in a class labeled by $E_8(a_i)$ for $i\in\I_{8}\setminus\{5\}$. In this case $w$ lies in no proper Weyl subgroup of $W$. Therefore it is enough to find a proper $w$-stable root subsystem $\Phi_{\Me}$ such that \ref{S} holds, since in this case $w\not\in W_{\Me}$. 

\medbreak
\noindent{\bf Cases $E_8(a_4),\,E_8(a_6),\,E_8(a_8)$.} In these cases $|w|$ is even and 
\[w^{\frac{|w|}{2}}=w_0=-\id,\]
as one can verify using $c_w(X)$. For each $w$ we exhibit a suitable $b\in \NN$ coprime with  $\frac{|w|}{2}$ and such that $w^b$ lies in a proper Weyl subgroup $W'$ with root system $\Phi'$. We choose $\Phi'$ to be the root system corresponding to the label of $w^b$. The latter is computed via the characteristic polynomial of $w^b$. Then $\langle w\rangle =\langle w^b,w^{\frac{|w|}{2}}\rangle$ preserves $\Phi'$ and all its irreducible components. 
The table below contains the label, the order, and the characteristic polynomial of $w$, the choice of $b$, and the label and characteristic polynomial of  $w^b$. 

\medbreak
\begin{center}
\begin{tabular}{|c|c|c|c|c|c|}
\hline
$w$ & $c_w(X)$ & $|w|$  & $b$ & $c_{w^b}(X)$ & $w^b$\\
\hline
$E_8(a_4)$&$\varphi_6(X)\varphi_{18}(X)$&$18$&$2$&$\varphi_3(X)\varphi_9(X)$&$A_8$\\
\hline
$E_8(a_6)$&$\varphi_{10}(X)^2$&$10$&$2$&$\varphi_{5}(X)^2$&$2A_4$\\
\hline
$E_8(a_8)$&$\varphi_6(X)$&$6$&$2$&$\varphi_3(X)$&$4A_2$\\
\hline
\end{tabular}
\end{center}

\medbreak
We choose $\Phi_{\Me}$ to be an irreducible component of $\Phi'$. 
If $w$ has label $E_8(a_4)$ and $E_8(a_6)$, then \ref{S} follows from Lemma \ref{lem:perfect-S}. Let $w$ have label $E_8(a_8)$. Then $w=w_0w^{-2}$ acts on $A_2$ via an outer automorphism. Hence, if $q\neq 2$, condition \ref{S} follows from \cite[Proposition 24.21, Theorem 24.17]{malle-testerman}. 

\medbreak
We claim that if $q=2$ the element $w$ cannot be minimal in $W[\mathfrak O]$. Indeed, for $q=2$ the subgroup $\T^{F_w}$ is elementary abelian of order $=3^4$, \cite[Table 8]{MaximalTori}. Therefore, $t^q=t^{-1}$ whence $t^{-1}=w\cdot t$. But then $w^2\cdot t=t$ contradicting Lemma \ref{lem:Jw}.

\medbreak
\noindent{\bf Case $E_8(a_7)$.} Here $c_w(X)=\varphi_6(X)^2\varphi_{12}(X)$, so $c_{w^3}(X)=\varphi_2(X)^4\varphi_4(X)^2$ and by inspection, $w^3$ lies in the class labeled by $2A_3+2A_1$. Let 
 \begin{align*}
&\beta_1\coloneqq 2\alpha_1+2\alpha_2+3\alpha_3+4\alpha_4+3\alpha_5+2\alpha_6+\alpha_7,\\&\beta_2\coloneqq  \alpha_2+\alpha_3+2\alpha_4+2\alpha_5+2\alpha_6+\alpha_7.  
 \end{align*} We choose the representative $w^3=(s_3s_4s_2)(s_{\beta_2}s_6s_7)s_0s_{\beta_1}$. Then, $w$ preserves $\Phi_{\Me}\coloneqq \Phi\cap{\rm Ker}(w^6+\id)$ with base
 \begin{align*}\{\alpha_3+\alpha_4,\alpha_5,\alpha_2+\alpha_4,\alpha_6+\alpha_7\}=s_6s_5s_7s_4\{\alpha_2,\alpha_3,\,\alpha_4,\alpha_5\}.\end{align*} Hence, $\Me$ is a non-standard Levi subgroup of type $D_4$. Thus, $[\Me,\Me]$ is simply connected and $M/Z(M)\simeq [M,M]/Z([M,M])$ is either $\Pom^+_4(q)$ or $^{(3)}D_4(q)$, giving \ref{S}. We are done since by construction $w\not\in W_\Me$. 
 
\medbreak
\noindent{\bf Case $E_8(a_2)$.} This is the only cuspidal class in $W$ consisting of elements of order $20$. Let $\Psi_1$ and $\Psi_2$ be the mutually orthogonal root systems of type $A_4$, with base $\Delta_1\coloneqq \{\alpha_1,\alpha_2,\alpha_3,\alpha_4\}$ and $\Delta_1\coloneqq \{\alpha_6,\alpha_7,\alpha_8,\alpha_0\}$, respectively, and let $W_{(1)}$ and $W_{(2)}$ be their respective Weyl groups. We will exhibit a cuspidal element of order $20$ in the normaliser $N_W(W_{(1)}\times W_{(2)})$, thus in the stabiliser of $\Psi_1+\Psi_2$, which is necessarily conjugate to $w$. 

\medbreak
Let ${\rm Aut}_W(\Delta_1\cup\Delta_2)$ be the subgroup of ${\rm Aut}(\Delta_1\cup\Delta_2)$ consisting of the symmetries that can be realized by elements of $W$. It follows from the proof of \cite[Proposition 28]{carter-cc} that $N_W(W_{(1)}\times W_{(2)})=(W_{(1)}\times W_{(2)}){\rm Aut}_W(\Delta_1\cup\Delta_2)$. Let $\vartheta_1$  be the symmetry of $\Delta_1\cup\Delta_2$ acting as the non-trivial involution on $\Delta_1$ and trivially on $\Delta_2$ and let $\tau$ be the symmetry of order $4$ mapping $\alpha_1$ to $\alpha_0$ and $\alpha_0$ to $\alpha_2$. Then, ${\rm Aut}(\Delta_1\cup\Delta_2)=\langle \vartheta_1,\tau\rangle$ is dihedral of order $8$. 
We claim that ${\rm Aut}_W(\Delta_1\cup\Delta_2)=\langle \tau\rangle$. Indeed, \cite[Table 11]{carter-cc} gives that $|{\rm Aut}_W(\Delta_1\cup\Delta_2)|=4$ and $\vartheta_1\not\in {\rm Aut}_W(\Delta_1\cup\Delta_2)$, because by \cite[Lemma 1]{carter-cc} it would necessarily by realized by an element of the Weyl group of $\Phi\cap \Psi_{2}^\perp=\Psi_1$, that is, $W_{(1)}$, impossible. Hence, $\tau\in W$ and 
\begin{align*}
(s_1s_2s_3s_4\tau)^4&=(s_1s_2s_3s_4s_0s_6s_8s_7\tau^2)^2\\
&=s_1s_2s_3s_4s_0s_6s_8s_7s_2s_1s_4s_3s_6s_0s_7s_8\\
&=(s_1s_2s_3s_4s_2s_1s_4s_3) (s_0s_6s_8s_7s_6s_0s_7s_8).
\end{align*}
With the identifications $W_{(1)}\simeq W_{(2)}\simeq\mathbb S_{5}$ we verify that
$s_1s_2s_3s_4s_2s_1s_4s_3$ and $s_0s_6s_8s_7s_6s_0s_7s_8$ are $5$-cycles. Therefore, the characteristic polynomial of $(s_1s_2s_3s_4\tau)^4$ is $\varphi_5(X)^2$. Thus, $(s_1s_2s_3s_4\tau)^4$, and a fortiori $s_1s_2s_3s_4\tau$, are cuspidal. 
By order reasons $s_1s_2s_3s_4\tau$ is a representative of the class $E_8(a_2)$ and it preserves a root subsystem of type $2A_4$ interchanging the irreducible components. Then, \ref{S} follows from Lemma \ref{lem:cyclically-permuted} and we are done because $w\not\in W_\Me$ by construction.

 \medbreak
\noindent{\bf Cases $E_8(a_1)$ and $E_8(a_3)$.} The elements in these classes have order, respectively $24$ and $12$, and characteristic polynomial $\varphi_{24}(X)$ and $\varphi_{12}(X)^2$ . By inspection of \cite[Table B.8]{geck-pfeiffer}  the class $E_8(a_3)$ contains the squares of the elements in the class $E_8(a_1)$. We first show that the elements in both classes stabilize a suitable root system of type $4A_2$. Let
\begin{align*}
 &\beta_{11}\coloneqq \alpha_1,&&\beta_{12}=\alpha_3,\\
 &\beta_{21}=\alpha_{2},&&\beta_{22}\coloneqq -(\alpha_1+2\alpha_2+2\alpha_3+3\alpha_4+2\alpha_5+\alpha_6),\\
 &\beta_{31}\coloneqq \alpha_5,&&\beta_{32}=\alpha_6,\\
 &\beta_{41}=\alpha_{8},&&\beta_{42}\coloneqq -\alpha_0,\\
& \Delta_i\coloneqq \{\beta_{i1},\beta_{i2}\},&& i=1,2,3,4.
\end{align*}
For $i\in\I_{4}$, let $\Psi_i$ be the root system with base $\Delta_i$, let $W_{(i)}$ be its Weyl group and $\vartheta_i$ be the non-trivial automorphism of $\Delta_i$. 
Let $\Psi\coloneqq \sum_{i=1}^4\Psi_i$ and $W'\coloneqq \prod_{i=1}^4W_{(i)}$. Then $\mathbb S_4$ acts on $\coprod_{i=1}^4\Delta_i$  permuting the first indices of the roots $\beta_{ij}$. We set
\begin{align*}
N&\coloneqq \{ (\vartheta_1^{e_1},\vartheta_2^{e_2},\vartheta_3^{e_3},\vartheta_4^{e_4}), e_1,e_2,e_3,e_4\in\{0,1\}\}; \textrm{ hence}\\
\Gamma&\coloneqq {\rm Aut}(\coprod_{i=1}^4\Delta_i)\simeq N\rtimes\mathbb S_4.\end{align*} 

We show that  $N_W(W')$, that is the stabilizer of $\Psi$ in $W$,  meets the class $E_8(a_1)$. 
We believe that part of the arguments are available in the literature, but were unable to locate a reference. 

\medbreak
Let $H$ be the subgroup of $\Gamma$ that can be realized by elements in $W$.
First of all, the proof of \cite[Proposition 28]{carter-cc}  gives $N_W(W')=W'\rtimes H$ and \cite[Table 11]{carter-cc} gives $|H|=48$. We claim that $N\cap H=\langle  (\vartheta_1,\vartheta_2,\vartheta_3,\vartheta_4)\rangle$. Indeed,
$ (\vartheta_1,\vartheta_2,\vartheta_3,\vartheta_4)$ is realized in $W$ as $w_0\prod_{i=1}^4s_{\beta_{i1}}s_{\beta_{i2}}s_{\beta_{i1}}$, giving $\supseteq$.  To prove $\subseteq$, suppose that $\tau\in N\cap H\setminus\{e,(\vartheta_1,\vartheta_2,\vartheta_3,\vartheta_4)\}$ and let $e_1,e_2,e_3,e_4\in\{0,1\}$ be such that $\tau=(\vartheta_1^{e_1},\vartheta_2^{e_2},\vartheta_3^{e_3},\vartheta_4^{e_4})$. Then 
\[J\coloneqq \{i\in\{1,2,3,4\}~|~e_i\neq0\}\subsetneq\{1,2,3,4\} \neq\emptyset.\]
By \cite[Lemma 1]{carter-cc} we have 
\begin{align*}
\tau\in \begin{cases}\prod_{j\not\in J}W_{(j)}& \textrm{ if } |J|\in\{2,3\},\\
W_{E_6}& \textrm{ if }|J|=1\end{cases}
\end{align*}
where $W_{E_6}$ denotes the Weyl group of a root system of type $E_6$ containing $\bigcup_{i\in \I_4\setminus J}\Delta_i$. 
However, this is clearly impossible for $|J|=2,3$ and for $|J|=1$ this would imply that 
$-\id=\tau \prod_{i\not\in J}s_{\beta_{i1}}s_{\beta_{i2}}s_{\beta_{i1}}\in W_{E_6}$, a contradiction. 

\medbreak
Then $H/(N\cap H)\simeq NH/N$, has order $24$, forcing $|\Gamma|=|NH|$. Thus, $H/(N\cap H)\simeq (N\rtimes \mathbb S_4)/N$. Therefore for each $\sigma\in \mathbb S_4$ there is 
$\sigma'\in N$ such that $N\sigma\cap H=\{\sigma'\sigma, (\vartheta_1,\vartheta_2,\vartheta_3,\vartheta_4)\sigma'\sigma\}$. In particular, for $(123)$ and $(34)\in\mathbb S_4$ there exist $a_i,\,b_i\in\{0,1\}$ for $i\in \I_3$ and $j\in \I_{2,4}$ such that 
\begin{align*}
\rho_{(123)}&=(\vartheta_1^{a_1},\vartheta_2^{a_2},\vartheta_3^{a_3},\vartheta_4)(123)\in H,\\
\rho_{(34)}&=(\vartheta_1,\vartheta_2^{b_2},\vartheta_3^{b_3},\vartheta_4^{b_4})(34)\in H.
\end{align*}
The condition $\rho_{(123)}^3\in N\cap H$ and $\rho_{(34)}^2\in N\cap H$ force 
\begin{align*}&a_1+a_2+a_3\equiv 1\mod 2, &&b_3+b_4\equiv0\mod 2.\end{align*} 

Also, $b_2$ is necessarily $0$, for if $b_2=1$, then by \cite[Lemma 1]{carter-cc} we would have $(\vartheta_1,\vartheta_2,\vartheta_3,\vartheta_4)\rho_{(34)}=(\id,\id,\vartheta_3^{b_3+1},\vartheta_4^{b_4+1})(34)\in W_{(3)}\times W_{(4)}$, impossible. Let then
\begin{align*}
\rho\coloneqq s_{\beta_{11}}\rho_{(123)}\rho_{(34)}=s_{\beta_{11}}(\vartheta_1^{a_1+b_3},\vartheta_2^{a_2+1},\vartheta_3^{a_3},\vartheta^{1+b_3}_4)(1234).
\end{align*}
A direct calculation shows that 
\begin{align*}
&&\rho^2&=s_{\beta_{11}}s_{\vartheta_2^{a_2+1}\beta_{21}}
(\vartheta_1^{a_1+1},\vartheta_2^{a_1+a_2+1+b_3},\vartheta_3^{a_3+a_2+1},\vartheta^{1+b_3+a_3}_4)(13)(24),\\
&&\rho^4&=s_{\beta_{11}}s_{\vartheta_2^{a_2+1}\beta_{21}}s_{\vartheta_3^{a_3+a_2+1}\beta_{31}}
s_{\vartheta_4^{(a_2+a_3+b_3)}\beta_{41}}(\vartheta_1,\vartheta_2,\vartheta_3,\vartheta_4),
\\
&\text{so}& 
\rho^8 &= (s_{\beta_{11}}s_{\beta_{12}})(s_{\beta_{21}}s_{\beta_{22}})^{\epsilon_1}(s_{\beta_{31}}s_{\beta_{32}})^{\epsilon_2}(s_{\beta_{41}}s_{\beta_{42}})^{\epsilon_3}
\end{align*}
for some $\epsilon_1,\epsilon_2,\epsilon_3\in\{\pm1\}$.

\medbreak
Therefore $\rho^8\in W'$ and lies in the cuspidal class labeled  $4A_2$, whose elements have 
order $3$. Thus, $\rho$ lies in the unique cuspidal class of elements of order $24$, namely $E_8(a_1)$. We then take $w=\rho$ and $\Phi_{\Me}\coloneqq \Psi$.  Now, $w$ permutes cyclically the components $\Psi_i$ for $i\in\I_{4}$ and the action of $\rho^4$ on $\Phi_{\Me}$ is outer. Then \ref{S} follows from Lemma \ref{lem:cyclically-permuted} and $w\not\in W_\Me$, concluding the proof for the class $E_8(a_1)$. 

\medbreak
We turn to the class $E_8(a_3)$. The representative $\rho^2$ interchanges the components $\Delta_1$ and $\Delta_{3}$. We take $\Phi_{\Me}=\Psi_1\cup\Psi_3$, so $w\not\in W_\Me$ and  Lemma \ref{lem:cyclically-permuted} gives \ref{S}.
\end{proof}

\begin{obs}
Automorphisms of root systems whose characteristic polynomial is not a cyclotomic polynomial all stabilize a root subsystem of $\Phi$, see \cite[Lemma 3.4]{GLRY}. However, we could not avoid case-by-case considerations, as to verify \ref{S} we need to have a grip on the stabilized root systems. 
\end{obs}

Summarizing, we obtain the following result on conjugacy classes in groups of type $E_8$. 
\begin{prop}\label{prop:e7-excluded}
Let $\Oc$ be a conjugacy class in $E_8(q)$. If $\Oc$ is not the class of a regular semisimple element contained in a torus indexed by $w$ in class $E_8$ or $E_8(a_5)$, then $\Oc$ collapses.
\end{prop}
\begin{proof}
Unipotent, mixed, and split semisimple classes are covered by \cite[Main Theorem]{ACG-IV}, \cite[Theorem 1.1]{ACG-V}, and \cite[Theorem 4.1]{ACG-VII} respectively. The remaining classes are covered by Lemmata \ref{lem:e8-non-cuspidal} and \ref{lem:e8-cuspidal}.
\end{proof}

\begin{obs}Our strategy fails for cuspidal elements in tori corresponding to elements in the classes $E_8(a_5)$ or $E_8$ in $W$ because elements in these classes do not preserve any proper root subsystems of $\Phi$ as we now show.
\end{obs}

Let $w_1$ and $w_2$ be representatives of $E_8$ and $E_8(a_5)$, respectively. Their characteristic polynomials are  $c_{w_1}(X)=\varphi_{30}(X)$ and  $c_{w_2}(X)=\varphi_{15}(X)$, and they represent the only cuspidal classes in $W$ consisting of elements of order $30$ and $15$, respectively. Then $c_{w_1^2}(X)=c_{w_2}(X)$ so  $w_1^2$ is conjugate to $w_2$. Thus, it is enough to show  that $w_2$ does not preserves a proper root subsystem of $\Phi$. 

\medbreak
Since $c_{w_2}(X)$ is irreducible, $w_2$ can only stabilise a root subsystem $\Psi$ of maximal rank. Irreduciblity also forces $\Psi$ to consist of $d$ isomorphic components of rank $\frac{8}{d}$, with $d\in \NN_{\geq 1}$,  that are cyclically permuted by $w_2$. 
Since $|w_2|$ is odd, $d$ is necessarily $1$. Thus, $\Psi$ could only be of type $D_8$ or $A_8$.
Now, $w_2$ acts as an automorphism of $\Psi$, so there is a morphism $\rho\colon \langle w_2\rangle\to {\rm Aut}(\Psi)$, where the latter is either the Weyl group of type $B_8$, or $\langle \mathbb S_9, -(19)(28)(37)(46)\rangle\simeq \mathbb S_9\times \{\pm\id\}$. But the order of the elements in these groups that act irreducibly on $\mathbb R\Psi=\mathbb R\Phi$ are not divisors of $15$.

\subsection{Semisimple classes in \texorpdfstring{$F_4(q)$}{}} \label{subsec:F4}

\

In this Subsection  $\G$ is of type $F_4$, so $w_0$ acts as $-\id$ on $\Phi$ and $\bZ=e$. 
We need the formulas
\begin{align*}
\alpha_0&=2\alpha_1+3\alpha_2+4\alpha_3+2\alpha_4, &&
\\
\alpha_0^\vee(\zeta)&=\alpha_1^\vee(\zeta^{2})\alpha_2^\vee(\zeta^{3})\alpha_3^\vee(\zeta^{2})\alpha_4^\vee(\zeta), &\zeta &\in \ku^\times.
\end{align*}
The action of the simple reflections and of $s_0$ on $t$ is
\begin{align}
\label{eq:f4-reflections}
&\begin{aligned}
s_1\cdot t&= t\alpha_1^\vee(\zeta_1^{-2}\zeta_2),
\\ 
s_2\cdot t&= t\alpha_2^\vee(\zeta_1\zeta^{-2}_2\zeta_3^2),
\\
s_3\cdot t&= t\alpha_3^\vee(\zeta_2\zeta_3^{-2}\zeta_4),
\\
s_4\cdot t&= t\alpha_4^\vee(\zeta_4^{-2}\zeta_3),
\end{aligned}
\\
\label{eq:f4-s0}
s_0\cdot t&= t\alpha_0^\vee(\zeta_1^{-1})=t\alpha_1^\vee(\zeta_1^{-2})\alpha_2^\vee(\zeta_1^{-3})\alpha_3^\vee(\zeta_1^{-2})\alpha_4^\vee(\zeta_1^{-1}).
\end{align}
 
\begin{lema}\label{lem:f4-non-cuspidal}If $w$ is not cuspidal, then $\Oc$ is of type C, D, or F.
\end{lema}
\begin{proof} If $J_w$ is of type $A_2,\, A_2+ \widetilde{A}_1,\, \widetilde{A}_2+ A_1$ or $\widetilde{A}_2$, we invoke Lemma \ref{lem:A-odd}. 

\medbreak
If $J_w$ is of type $A_1+\widetilde{A}_1$, respectively $A_1$, then we may assume that  $\Delta_{J_w}=\{\alpha_1,\alpha_4\}$, respectively $\Delta_{J_w}=\{\alpha_1\}$. We take  $\Me\coloneqq \Le_{\{s_1,s_2\}}$, so $M=[M,M]=\SL_3(q)$, giving  \ref{S}. Condition \ref{R} follows from Lemmata \ref{lem:Jw} and \ref{lem:non-centrality}. In addition, $\Cent_{W_{\Me}}(s_1s_4)=\Cent_{W_{\Me}}(s_1)=\langle s_1\rangle$. 

\medbreak
If $J_w$ is of type $A_1+\widetilde{A}_1$, then $s_4\cdot t\not\in\{ t, s_1\cdot t\}$ by Lemma \ref{lem:Jw}. In addition $\alpha_1(s_4\cdot t)=\alpha_1(t)\neq 1$, so Lemma \ref{lem:non-centrality} applies, giving \ref{W} in this case. 

\medbreak
Let $J_w$ be of type $A_1$.  Lemma \ref{lem:Jw} gives $\alpha_1(s_4\cdot t)=\alpha_1(s_3\cdot t)=\alpha_1(w_0\cdot t)^{-1}=\alpha_1(t)\neq 1$ so \ref{W2} holds by virtue of Lemma \ref{lem:non-centrality}. Then \ref{W} might fail only if  $s_4\cdot t, s_3\cdot t, w_0\cdot t\in\{t, s_1\cdot t\}$. Remark \ref{obs:coefficient}, Lemma \ref{lem:Jw} and $w\neq e$ give then $s_4\cdot t=s_3\cdot t=t$ and $t^{-1}=s_1\cdot t=t^q$, so \eqref{eq:f4-reflections} gives $\zeta_2=\zeta_3=\zeta_4=1$. In this case, $t\in\langle \U_{\pm\alpha_1},\U_{\pm\alpha_2}\rangle^F_{s_1}\simeq \SL_3(q)$, and it is not cuspidal therein. The result follows from Remark \ref{obs:zeta=1} and \cite[Theorem 1.1]{ACG-III}. 

\medbreak
If $J_w$ is of type $\widetilde{A}_1$, may assume $w=s_4$. We argue as we did for $J_w$ of type $A_1$, interchanging the roles of $s_1$ and $s_4$, and replacing $s_3$ by $s_2$. 
 
\medbreak
If $J_w$ is of type $C_2$, that is, $\Delta_{J_w}=\{\alpha_2,\,\alpha_3\}$, then we set $\Me=\Le_{J_w}$, so $M\simeq\Sp_4(q)$. Then, \ref{S} is immediate for $q>2$ and is guaranteed by Lemma \ref{lem:S-b2} if $q=2$. 
Condition \ref{R} follows from Lemma \ref{lem:Jw}. 
Assume that  \ref{W} fails. Since $\alpha_0\perp\Phi_{\Me}$, we have 
$1\neq\beta(t)=(s_0(\beta))(t)=\beta(s_0\cdot t)$ for any $\beta\in\Phi_{\Me}$, so $s_0\cdot t\in W_{\Me}\cdot t$. Remark \ref{obs:coefficient} applied to $\alpha_4^\vee$ and \eqref{eq:f4-s0} give $\zeta_1=1$. Then $t\in \langle U_{\pm\alpha_2},\U_{\pm\alpha_3},\U_{\pm\alpha_4}\rangle^{F_w}$, a simply-connected group of type $C_3$, and it is not split therein. 
We invoke Remark \ref{obs:zeta=1} and \cite[Theorem 7.1]{ACG-VII} and conclude that 
$\Oc$ is of type C, D, or F.

\medbreak
If $J_w$ is of type $B_3$, that is, if $\Delta_{J_w}=\{\alpha_1,\,\alpha_2,\,\alpha_3\}$, we take $\Phi_{\Me}$  to be the root system of type $A_3$ consisting of all long roots in $\Phi_{J_w}$. Then 
\[M\simeq \SL_4(q)\]
by \cite[Corollary II.5.4]{sp-st} so \ref{S} follows. 
The inclusion $\Me\subseteq\Le_{J_w}$ together with Lemma \ref{lem:Jw} gives \ref{R}. 
Let $w_0'$ be the longest element in $\langle {J_w}\rangle$. 
Then $w_0'\in \Cent_W(w)\setminus \Cent_{W_{\Me}}(w)$. 
Lemma \ref{lem:Jw} gives $w_0'\cdot t\neq \sigma\cdot t$ for any $\sigma\in W_{\Me}$ and $\Cent_{\Me}(w_0'\cdot t)=\T$, so \ref{W} follows.

\medbreak
Finally, if $J_w$ is of type $C_3$, that is, if $\Delta_{J_w}=\{\alpha_2,\,\alpha_3,\,\alpha_4\}$, then
there are three possibilities for a representative of the conjugacy class of $w$ in $\langle J_w\rangle$: 
\begin{itemize}[leftmargin=*]
\item a Coxeter element,

\medbreak
\item $s_{\alpha_2}s_{\alpha_3}s_{\alpha_2+2\alpha_3+2\alpha_4}$, or

\medbreak
\item $w_0s_0$ (i.e., the longest element in $\langle J_w\rangle$). 
\end{itemize}
Their orders are $6$, $4$, and $2$ respectively. 
If $\zeta_1=1$, then $t$ belongs to $M\simeq\Sp_6(q)$ and it is cuspidal there. 
We invoke Remark \ref{obs:zeta=1} and \cite[Theorem 7.1]{ACG-VII}
to conclude that $\Oc$ is of type C, D, or F. 

\medbreak
We assume in the remainder of the proof that $\zeta_1\neq1$.
We first consider $\Me=\Le_{J_w}$, so condition \ref{S} is ensured by Lemma \ref{lem:perfect-S} and \ref{R} follows from Lemma \ref{lem:Jw}. Assume that \ref{W} fails. Then $w_0\cdot t=t^{-1}=\sigma\cdot t$ for some $\sigma\in \Cent_{W_{\Me}}(w)$. Comparing coefficients we see that this equality may hold only if $q$ is odd and $\zeta_1=-1$. Regularity of $t$ in $\Me$ and minimality of $J_w$ imply $|\sigma|=2$. If $w$ is the Coxeter element in $W_{\Me}$, then $\Cent_{W_{\Me}}(w)=\langle w\rangle$, hence 
$\sigma=w^3=w_0s_0$, the only non-trivial involution in  $\Cent_{W_{\Me}}(w)$. 
Then, 
\begin{align*}t\alpha_2^\vee(-1)\alpha_4^\vee(-1)=s_0\cdot t=w_0\sigma\cdot t=t,\end{align*}  a contradiction, concluding the proof for the Coxeter element. 

\medbreak
Let $w$ be either $w_0s_0=s_{\alpha_2}s_{\alpha_2+2\alpha_3}s_{\alpha_2+2\alpha_3+2\alpha_4}$ or  $s_{\alpha_2}s_{\alpha_3}s_{\alpha_2+2\alpha_3+\alpha_4}$. Take now $\M = \langle \T, \U_{\pm\alpha_2},\U_{\pm\alpha_3}\rangle\subset \Le_{J_w}$, so \ref{S} holds by Lemma \ref{lem:perfect-S} since $q$ is odd and \ref{R} holds by Lemma \ref{lem:Jw}. By definition,  
$s_{\alpha_2+2\alpha_3+2\alpha_4}\in \Cent_{W}(w)$ and $s_{\alpha_2+2\alpha_3+2\alpha_4}\cdot t\not\in \left(\Cent_{W_{\Me}}(w)\cdot t\cup \Cent_{\T}(M)\right)$ by Lemma \ref{lem:Jw}, giving \ref{W}. 
\end{proof}

\begin{lema}\label{lem:f4-cuspidal}If $w$ is cuspidal, then $\Oc$ is of type C, D, or F.
\end{lema}
\begin{proof}We verify the hypothesis of Corollary \ref{cor:cuspidal-Ze}. A list of representatives of cuspidal classes in $W$ is given in \cite[Table B.3]{geck-pfeiffer}, together with a minimal Weyl subgroup containing each of them.

\medbreak
Let $\Phi_D$ be the $W$-stable root subsystem of type $D_4$ of $\Phi$  consisting of all long roots, with Weyl group $W_D\leq W$. 
We first take $\Phi_{\Me}=\Phi_D$, so $[\Me,\Me]$ is simply-connected ensuring \ref{S} by virtue of Lemma \ref{lem:perfect-S}. 
If $w\not\in W_D$, then we are done. Assume $w\in W_D$. Then up to conjugation $w$ is either $w_0$, or  $w_1=s_1s_2s_{\alpha_2+2\alpha_3}s_{\alpha_2+2\alpha_3+2\alpha_4}$ (the Coxeter element in $W_D$), or  $w_2=s_1s_2s_{\alpha_1+\alpha_2+2\alpha_3}s_{\alpha_2+2\alpha_3+2\alpha_4}$. If $w=w_0$ or $w=w_1$ then \eqref{eq:cuspidal-Ze} holds because $s_3\in \Cent_W(w)\setminus \Cent_{W_{\Me}}(w)$. Finally, if $w=w_2$, we replace $\Phi_{\Me}$ with the $w_2$-stable root subsystem with base $\{\alpha_2,\alpha_3\}$. 
Then $[\Me,\Me]\simeq\Sp_4(\ku)$ so \ref{S} follows from Lemma \ref{lem:perfect-S} if $q>2$ and Lemma \ref{lem:S-b2} if $q=2$. A direct calculation shows that $s_{\alpha_1+\alpha_2+\alpha_3}s_{\alpha_2+2\alpha_3+2\alpha_4}\in \Cent_W(w)\setminus \Cent_{W_{\Me}}(w)$.
\end{proof}

We are now ready to state a general result for $\Gb$ of type $F_4$. 
\begin{theorem}\label{thm:f4-collapses}Every non-trivial class in a simple Chevalley group of type $F_4$ is of type C, D, or F. 
\end{theorem}
\begin{proof}
For unipotent, mixed, and split semisimple classes this is \cite[Main Theorem]{ACG-IV}, \cite[Theorem 1.1]{ACG-V}, \cite[Theorem I]{ACG-VII}, respectively. The remaining classes are covered by Lemmata \ref{lem:f4-non-cuspidal} and \ref{lem:f4-cuspidal}.
\end{proof}

\subsection{Semisimple classes in \texorpdfstring{$G_2(q)$}{}} \label{subsec:G2}

\

In this Subsection  $\G$ is of type $G_2$ so $\bZ=e$ and the highest root is $\alpha_0=3\alpha_1+2\alpha_2$. Recall that $\Gb$ is simple if and only if $q>2$. The Weyl group $W$ is isomorphic to the dihedral group of order $12$ and the action of the simple reflections on $t$ is
\begin{align}\label{eq:g2-reflections}
s_1\cdot t&= \alpha_1^\vee(\zeta_1^{-1}\zeta_2)\alpha_2^\vee(\zeta_2),&s_2\cdot t&= \alpha_1^\vee(\zeta_1)\alpha_2^\vee(\zeta^{-1}_2\zeta_1^3).
\end{align}

The table below collects a representative of each conjugacy class in $W$ and its characteristic polynomial.

\medbreak
\begin{table}[ht]
\begin{center}
\begin{tabular}{|c|c|c|c|c|c|}
\hline
$\id$&$s_1$&$s_2$&$s_1s_2$&$(s_1s_2)^2$&$w_0=-\id$ \\ 
\hline
$\varphi_1(X)^2$&$\varphi_1(X)\varphi_2(X)$&$\varphi_1(X)\varphi_2(X)$&$\varphi_6(X)$&$\varphi_3(X)$&$\varphi_2^2(X)$\\
\hline\end{tabular}
\end{center}
\end{table}

\begin{lema}\label{lem:g2-non-cuspidal} If  $q>2$ and $w$ is not cuspidal, 
then $\Oc$ is of type C, D, or F.
\end{lema}

\begin{proof}By assumption $w=s_\alpha$ for some $\alpha\in\Delta$. Let $\Phi_\Me$ be the root subsystem of type $A_2$ in $\Phi$ consisting of all long roots. Then $[\Me,\Me]$ is simply-connected and \ref{S} holds by Lemma \ref{lem:perfect-S} because $q>2$. 

\medbreak
Assume first that $\alpha=\alpha_2$. Then, \ref{R} holds by Lemmata \ref{lem:non-centrality} and Lemma \ref{lem:Jw}, because $\Cent_{W_{\Me}}=\langle s_2\rangle$, $\Cent_{W}(w)=\langle w_0,s_2\rangle$. By minimality of $w$, condition \ref{W} fails for $\Me$ if and only if $w_0\cdot t=s_2\cdot t$, that is, if and only if $\zeta_1=1$. In this case, $t=\alpha_2^\vee(\zeta_2)\in M\simeq \SL_3(q)$, and it is not cuspidal therein. Then,
the statement follows from \cite[Theorem 1.1]{ACG-III}.

\medbreak
Assume now that $\alpha=\alpha_1$ and that $\zeta_2\neq1$. Then $\alpha_0(t)=\zeta_2\neq1$ so \ref{R} follows from Lemma \ref{lem:non-centrality}.
We have 
\begin{align*}
\Cent_{W_{\Me}}(s_1) &=\langle s_0\rangle && \text{and}&
\Cent_W(s_1) &= \langle s_1,s_{0}\rangle.
\end{align*}
Then $w_0\cdot t\neq t$ because $e\not\in\WO$ and $w_0\cdot t\neq s_{0}\cdot t$ because $w_0=s_0s_1$ and 
$s_1\cdot t\neq t$ by Lemma \ref{lem:Jw}. 
Therefore Remark \ref{obs:w0} \ref{item:w0-uno} applies.

\medbreak
Let now $\alpha=\alpha_1$ and $t=\alpha_1^\vee(\zeta_1)$. Then $t^2\neq 1$ by minimality of $w$, 
\begin{align*}
(2\alpha_1+\alpha_2)(t) &=\zeta_1\neq 1& &\text{and} 
& (\alpha_1+\alpha_2)(t)&=\zeta^{-1}_1\neq 1.
\end{align*}
Consider the $F_{s_1}$-stable group $\Vu\coloneqq \langle\U_{2\alpha_1+\alpha_2},\U_{\alpha_1+\alpha_2}\rangle$. 
Then 
\begin{align*}
\Vu &= \begin{cases}
&\U_{\alpha_1+\alpha_2}\U_{2\alpha_1+\alpha_2}, \text{ if }p=3,
\\
&\U_{\alpha_1+\alpha_2}\U_{2\alpha_1+\alpha_2}\U_{\alpha_0}, \text{ if }p \neq 3,
\end{cases}
\end{align*}
because of the relations  \cite[(3.10)]{ree-g2}. 
By \cite[Proposition 23.8]{malle-testerman}, the subgroup $V\coloneqq \Vu^{F_{s_1}}$ is:

\begin{itemize}[leftmargin=*]
\item when $p=3$, an abelian group of order $q^2$ whose elements are of the form 
$x_{\alpha_1+\alpha_2}(\eta_1)x_{2\alpha_1+\alpha_2}(\eta_2)$
for suitable $\eta_1,\eta_2\in\overline{\F_q}$; 

\medbreak
\item when  $p\neq 3$, a non-abelian group of order $q^3$ whose elements have the form 
$x_{\alpha_1+\alpha_2}(\eta_1)x_{2\alpha_1+\alpha_2}(\eta_2)x_{\alpha_0}(\eta_3)$
for suitable $\eta_1,\eta_2,\eta_3\in\overline{\F_q}$.
\end{itemize}

\medbreak
Let $H\coloneqq \langle t, V\rangle$. Then 
\begin{align*}
\Oc_{t^{\pm1}}^H &\subseteq t^{\pm1}V,& &\text{so}& 
\Oc_t^H\cap\Oc_{t^{-1}}^H& =\emptyset; &&\text{also,}
&|\Oc_{t^{\pm1}}^H|&=|\Oc_{t^{\pm1}}^V|\geq p>2.
\end{align*}
In addition, $t^{-1}(v\trid t)\in\langle \Oc_t^V\rangle$ for any $v\in V$, and such elements generate $V$, so $\langle \Oc_t^H\rangle=H$. Since $H$ is non-abelian, there is $r\in\Oc_t^H$ such that $r\trid t\neq t$. Taking $s\coloneqq t^{-1}$, we conclude that $\Oc$ is of type C
by definition.
\end{proof}

\begin{lema}\label{lem:g2-cuspidal}Assume that $q>2$. If $w$ is cuspidal, then $\Oc$ is of type C.
\end{lema}
\begin{proof}Here $w$ is either $s_1s_2$, or $(s_1s_2)^2$, or $w_0$. Let $\Phi_\Me$ be the root subsystem of type $A_2$ consisting of all long roots  in $\Phi$. Then $[\Me,\Me]$ is simply-connected and \ref{S} follows from Lemma \ref{lem:perfect-S} because $q>2$. In addition,  $w_0\in \Cent_W(w)\setminus {W_{\Me}}$, so Corollary \ref{cor:cuspidal-Ze} applies.
\end{proof}

We are now in a position to state a general result for $\Gb$ of type $G_2$. 
\begin{theorem}\label{thm:g2-collapses}
Every non-trivial class in a simple Chevalley group of type $G_2$ is of type C, D, or F. 
\end{theorem}
\begin{proof}
For unipotent, mixed, and split semisimple classes this is \cite[Main Theorem]{ACG-IV}, \cite[Theorem 1.1]{ACG-V}, \cite[Theorem I]{ACG-VII}, respectively. The remaining classes are covered by Lemmata \ref{lem:g2-non-cuspidal} and \ref{lem:g2-cuspidal}.
\end{proof}

\medbreak 
We end this subsection dealing with  the non simple group $G_2(2)$, for future application to recursive arguments when addressing  Steinberg groups.

\begin{lema}\label{lem:normal-simple-psu}
Let $G={\rm Aut}(\PSU_3(3))$ and let $r$ be an element of order $7$ in $\PSU_3(3)=\SU_3(3)$.  Then $\Oc_r^G$ is of type C.
\end{lema}
\begin{proof}
The only  maximal torus in $\PSU_3(3)$ whose order is divisible by $7$, corresponds to the $\theta$-conjugacy class of $w=(123)$, hence it is 
conjugate in $\SL_3(\overline{\F_3})$ to a diagonal matrix of the form $d={\diag}(\eta,\eta^4,\eta^{2})$ for some $\eta\in \F_{3^6}$ such that $\eta^7=1$. The Frobenius endomorphism ${\Fr_3}$ lies in $G$, see \cite[Theorem 36]{yale}, 
and $\Fr_3\trid r = \Fr_3(r) = r^3$ is conjugate in $\SL_3(\overline{\F_3})$ 
to  $d^3={\diag}(\eta^3,\eta^5,\eta^{6})$ and thus, 
it is not conjugate to $r$ in $\SL_3(\overline{\F_3})$. 
Therefore $\Fr_3(r)$ is not conjugate to $r$ in $\PSU_3(3)$, 
giving
\[\Oc_r^G\cap\PSU_3(3)\neq \Oc_r^{\PSU_3(3)}.\]
Since $\PSU_3(3)$ is simple, the class $\Oc_r^G$ is of type C 
by Lemma \ref{lema:normal-simple}. 
 \end{proof}

\begin{prop}\label{prop:g22}
Every non-trivial semisimple class in $G_2(2)$ is of type C, D, or F.
\end{prop}
\begin{proof}Recall that $G_2(2)\simeq{\rm Aut}(\PSU_3(3))$. We identify $\PSU_3(3)$ with the normal subgroup of inner automorphisms. By \cite[Table 4]{MaximalTori}  non-trivial semisimple elements have order $3$ or $7$. The automorphisms of twisted groups of Lie type are compositions of inner, diagonal, and field (i.e., Frobenius) automorphisms, where in our situation a diagonal morphism is conjugation by the class of a diagonal matrix in $\mathbf{PGU}_3(3)$, \cite[Theorem 36]{yale}. Frobenius automorphisms normalize diagonal automorphisms, and have order $2$ in $G_2(2)$ because $\PSU_3(3)\simeq \SU_3(3)\leq \GL_3(9)$. 
In addition, diagonal elements of $\mathbf{PGU}_3(3)$ have order a power of $2$ because the diagonal entries lie in $\F_{9}^\times$. Thus the elements of order $3$ and $7$ lie in $\PSU_3(3)$. The classes of elements of order $3$ in $\PSU_3(3)$ are of type C, D, or F by \cite[Proposition 5.1]{ACG-IV}. Those of elements of order $7$ are of type C by Lemma \ref{lem:normal-simple-psu}. 
\end{proof}

\section{Kthulhu classes that collapse by the  criterion of
type \texorpdfstring{$\tipo$}{}}\label{sec:notC-Omega}

In this Section we exhibit conjugacy classes in simple groups of Lie type that collapse by virtue of the new criterion $\tipo$ and could not be handled with the previous techniques. On the other hand, in Remark \ref{obs:SLp-irreducible} we exhibit a family classes that cannot be addressed even with the new method. 
We retain notation from Sections \ref{sec:notation-M} and \ref{sec:typeC}.

\medbreak
\subsection{The groups \texorpdfstring{$\Pom^+_{2n+1}(q)$}{}} 

\ 

We first address the semisimple classes in $\Pom^+_{2n+1}(q)$, 
$n\geq 2$ that are not covered by Theorem \ref{thm:bn-typeC}. 

\medbreak
The class $\Oc$ for $t$ as in \eqref{eq:bn-w=1-even} is kthulhu 
for $n=2$ and $p=3,5$, in virtue of the isomorphism 
$\Pom^+_5(q)\simeq \PSp_4(q)$, see \cite[Table 2]{ACG-VII}. 

\begin{lema}\label{lem:bn-discarded-omega} 
Let $\Gb=\Pom^+_{2n+1}(q)$ with $n\geq 2$ and let $t\in\T^{F_w}$ be as in \eqref{eq:bn-discarded-case1} with $q\neq3$, \eqref{eq:bn-discarded-case2}, \eqref{eq:bn-w=1-odd}, or \eqref{eq:bn-w=1-even} with $q\neq 3$. 
Then, $\mathfrak O$ and $\Oc$ are of type $\tipo$.
\end{lema}
\begin{proof}
For all choices of $t$ we verify the hypotheses of Corollary \ref{cor:normal-abelian-order}.
In detail, we construct an $F_w$-stable elementary abelian $p$-subgroup $V$ of $G_w$ normalized by $t$, such that $\Cent_V(t)=e$ and show that the affine subrack $\Oc_t^V=tV$ does not belong to the list \eqref{eq:list-affine-simple-racks}. 
Since $V\cap{\rm Ker}(\pi)=e$ and $v\in \Cent_V(t)$ if and only if $\pi(v)\in \Cent_{\pi(V)}(\pi(t))$, the same argument applies to $\pi(tV)$. 

\medbreak
If $w=e$ or $w=s_n$, then $q>3$. We take $V=\U^{F_w}_{\alpha_1+\cdots+\alpha_n}$. Since $(\alpha_1+\cdots+\alpha_n)(t)=-1$, the conjugation action of $t$ on $V$ is inversion and $\Cent_V(t)=e$, hence $\left|\U_{\alpha_1+\cdots+\alpha_n}^{F_w}\right|=q$. Therefore the affine rack $tV$ is isomorphic to $(\F_q,q-1)$, which does not belong to the list \eqref{eq:list-affine-simple-racks} since $q>3$.

\medbreak
If $w=s_ns_{\alpha_{n-1}+\alpha_n}$, then $q=3$. We consider 
\begin{align*}V=\langle \U_{\alpha_1+\cdots+\alpha_{n-2}},\U_{w(\alpha_1+\cdots+\alpha_{n-2})}\rangle^{F_w}=(\U_{\alpha_1+\cdots+\alpha_{n-2}}\U_{w(\alpha_1+\cdots+\alpha_{n-2})})^{F_w},\end{align*}
which is an abelian group of order $9$, cf. \cite[Proposition 23.8]{malle-testerman}. 
Since 
\begin{align*}(\alpha_1+\cdots+\alpha_{n-2})(t)=-\omega^{-2}\neq1,\mbox{ and }w(\alpha_1+\cdots+\alpha_{n-2})(t)=-\omega^{2}\neq1\end{align*} there holds $\Cent_V(t)=e$. By size reasons the affine rack $tV$ does not belong to the list \eqref{eq:list-affine-simple-racks}. 
\end{proof}

\begin{lema}\label{lem:bn-discarded-omega3} 
Let $\Gb=\Pom^+_{2n+1}(3)$ with $n\geq 2$ and let $t\in\T^{F_w}$ be as in \eqref{eq:bn-discarded-case1} or \eqref{eq:bn-w=1-even}. 
Then, $\mathfrak O$ and $\Oc$ are of type $\tipo$.
\end{lema}
\begin{proof}
We show that the hypothesis of Corollary \ref{cor:solvable-order} holds.
Let 
\begin{align*}
\gamma_1 &\coloneqq \alpha_1+\cdots+\alpha_n & &\text{and}& 
\gamma_2 &\coloneqq \alpha_2+\cdots+\alpha_n.
\end{align*}
 We consider the solvable subgroup $\mathbf{S}\coloneqq \langle t, \U_{\gamma_1},\,\U_{\gamma_{2}}\rangle^{F_w}\subset \T\U\cap G_w$. In \eqref{eq:bn-discarded-case1} we have $n\geq 3$, so we may choose a representative of $w=s_n$ in $N_{\G}(\T)$ acting trivially on $\U_{\gamma_1}$ and $\U_{\gamma_2}$. Hence, for both choices of $t$ we have $\U^{F_w}_{\gamma_1}=\U^F_{\gamma_1}$ and $\U^{F_w}_{\gamma_2}=\U^{F}_{\gamma_2}$. 
The multiplication induces a bijection of sets 
\begin{align*}\mathbf{S} \simeq \langle t\rangle \times \U^F_{\gamma_1} \times \U^F_{\gamma_2} \times \U_{\gamma_1+\gamma_2}^F,\end{align*} 
so $|{\mathbf{S}}|=54$. As $\gamma_1(t)=\gamma_2(t)=-1$, 
we have $\Cent_{\mathbf{S}}(t)=\langle t,\U_{\gamma_1+\gamma_2}^F\rangle$ so $|\Oc_t^{\mathbf{S}}|=9$. Also, $\U_{\gamma_i}^F\trid t=t\U^F_{\gamma_i}$ for $i=1,2$, so 
\begin{align*}{\mathbf{S}}=\langle  t\U^F_{\gamma_1}, t\U^F_{\gamma_2}\rangle \subseteq \langle \Oc_t^{\mathbf{S}}\rangle\subseteq {\mathbf{S}},\end{align*} giving the statement for $\mathfrak O$. The statement for $\Oc$ follows observing that the restriction of $\pi$ to ${\mathbf{S}}$ is injective. 
\end{proof}

We are now in a position to give a general statement for $\Pom_{2n+1}^+(q)$, for $q$ odd. 
If $q$ is even, then $\Pom_{2n+1}^+(q)\simeq \PSp_{2n}(q)$, to be treated in Section \ref{sec:abelian-techniques}.

\begin{theorem}\label{thm:bn-completo} Every non-trivial conjugacy class in $\Pom_{2n+1}^+(q)$ collapses for any odd $q$ and any $n\geq 2$. \qed
\end{theorem}
\begin{proof}
Unipotent and mixed classes are covered by \cite[Main Theorem]{ACG-IV}, \cite[Theorem 1.1]{ACG-V}, respectively. Semisimple classes are covered by \cite[Theorem I]{ACG-VII}, Theorem  \ref{thm:bn-typeC}, and
Lemmata \ref{lem:bn-discarded-omega} and \ref{lem:bn-discarded-omega3}.\end{proof}

\subsection{The groups \texorpdfstring{$\Pom^+_{2n}(q)$}{}}

\

We address the semisimple classes in $\Pom^+_{2n}(q)$, $n\geq 4$ that are not covered by Theorem \ref{thm:dn-one-exception}.

\begin{lema}\label{lem:d4-exception}
Let $\Gb=\Pom^+_{4}(q)$ be of type $D_4$, and let  \begin{align*}t=\alpha^\vee_1(\omega_1)\alpha^\vee_3(\omega_3)\alpha^\vee_4(\omega_4)\in \T^{F_{w}}\end{align*} with  $w=s_1s_3s_4$ and $J_w$ minimal. If $q=2$ or $t$ is as in \eqref{eq:exception-d4}, then $\mathfrak O$ and $\Oc$ are of type $\tipo$. 
\end{lema}
\begin{proof}Without loss of generality
$t=\alpha^\vee_1(\omega)\alpha^\vee_3(\omega)\alpha^\vee_4(\omega^{-1})$ with $\omega$ a primitive $3$-rd root of $1$ if $q=2$, and a primitive $4$-th root of $1$ if $q\equiv3\mod4$, see Remark \ref{obs:d4-exception}. 
In all cases $\alpha_2(t)=\omega^{-1}\neq1$ and $(w\alpha_2)(t)=\alpha_2(t^{-1})=\omega\neq 1$ so $(\alpha+w\alpha)(t)=\alpha_0(t)=1$. 

\medbreak

Let $q=2$. In this case $\bZ=e$ and $\mathfrak O=\Oc$. We consider the $w$-stable subgroup $\Ha=\langle \U_{\pm\alpha_2},\U_{\pm w\alpha_2}\rangle$ of $\G$. Its root system is of type $A_2$. A direct calculation shows that 
\begin{align*}
\alpha_2^\vee(\omega)\alpha_{w\alpha_2}^\vee(\omega^2) &=
\alpha_1^\vee(\omega^{-1})\alpha_3^\vee(\omega^{-1})\alpha_4^\vee(\omega^{-1})
\in Z(\Ha)
\end{align*} 
so $\Ha\simeq\SL_3(\ku)$ and $H\coloneqq \Ha^{F_w}\simeq\SU_3(2)$, a solvable group.  
The coefficients of $\alpha_3^\vee$ and $\alpha_4^\vee$ 
in the expression of $t$ are distinct, so 
$t\not\in H$. However, $t$ normalizes $H$, so $\widetilde{H}\coloneqq \langle t, H\rangle\simeq \langle t\rangle \ltimes H$ is a solvable group. Let 
\begin{align*}
V^+ &\coloneqq \langle \U_{\alpha_2},\U_{w\alpha_2}\rangle^{F_w} &  
&\text{and} & V^- &\coloneqq \langle \U_{-\alpha_2},\U_{-w\alpha_2}\rangle^{F_w}.
\end{align*}
These groups, described in \cite[Example 23.10]{ACG-VII}, are both isomorphic to the quaternion group $Q_8$. Then $Z(V^{\pm})=\U_{\pm\alpha_0}^{F_w}$  and $\Cent_{V^{\pm}}(t)=Z(V^{\pm})$. Hence, $\langle \Oc_t^{V^\pm}\rangle= \langle t,V^{\pm}\rangle$ and 
\begin{align*}\widetilde H= \langle t, V, V^-\rangle=\langle \Oc_t^V,\Oc_t^{V^-}\rangle\leq \langle \Oc_{t}^{\widetilde H}\rangle\leq \widetilde H.\end{align*} 
In addition, $\left|\Oc_t^{\widetilde H}\right|\geq |\Oc_t^{V^+}\cup \Oc_{t}^{V^-}|=4+4-|\Oc_t^{V^+}\cap \Oc_{t}^{V^-}|=7$. Since  $\left|\Oc_t^{\widetilde H}\right|$ divides $|\SU_3(2)|=2^3\cdot3^3$, we have $\left|\Oc_t^{\widetilde H}\right|\geq 8$. Corollary \ref{cor:solvable-order} applies.

\medbreak
Let now $q\neq 2$, retaining notation from the case $q=2$. We consider the $w$-stable solvable subgroup  
\begin{align*}
{\mathbf{S}} &\coloneqq \langle t,\, V^+\rangle \simeq \langle t\rangle \ltimes V^+ \leq G_w.
\end{align*}

By \cite[Example 23.10]{ACG-VII} it has order $q^3(q+1)$, and 
\[|\Cent_{\mathbf{S}}(t)|=|\langle t, \U_{\alpha_2+w\alpha_2}^{F_w}\rangle|=q(q+1),\] 
whence $|\Oc_t^{\mathbf{S}}|=q^2$. In addition, $\langle \Oc_t^{\mathbf{S}}\rangle$ contains $t$ and $q^2$ elements of the form $x_{\alpha_2}(\xi)x_{w\alpha_2}(\eta)x_{\alpha_2+w\alpha_2}(\zeta)$ with $\xi,\eta,\zeta$ in $\overline{\F_{q}}$ and the pairs $(\xi,\eta)$ all distinct. By Chevalley commutator formula, $\langle \Oc_t^{\mathbf{S}}\rangle$ also contains $\U_{\alpha_2+w\alpha_2}^{F_w}$, so ${\mathbf{S}}=\langle \Oc_t^{\mathbf{S}}\rangle$. Corollary \ref{cor:solvable-order} gives the statement for $\mathfrak O$.  The statement for $\Oc$ follows because the restriction of $\pi$ to ${\mathbf{S}}$   is injective.
\end{proof}

 We are  in a position to state a general result for $\Pom^+_{2n}(q)$. 
\begin{theorem}\label{thm:dn-completo} Every non-trivial conjugacy class in $\Pom_{2n}^+(q)$ for any $n\geq 4$ and any  $q$ collapses. \qed
\end{theorem}
\begin{proof}
Unipotent classes and mixed classes are covered by \cite[Main Theorem]{ACG-IV}, \cite[Theorem 1.1]{ACG-V}, respectively. Semisimple classes are covered by \cite[Theorem I]{ACG-VII}, Theorem  \ref{thm:dn-one-exception}, and
Lemma \ref{lem:d4-exception}.\end{proof}

\subsection{Unipotent classes in symplectic groups}

\

We now strengthen \cite[Theorems I, II, III]{ACG-VII} by showing that some of the kthulhu classes in $\PSp_{2n}(q)$ are of type $\tipo$. For unexplained notation we refer to \cite[Section 4.2.1]{ACG-II}. Recall that $q = p^m$.

\begin{lema}\label{lem:kthulhu-Sp-are-tipo}Let $\Gb=\PSp_{2n}(q)$. The following conjugacy classes are of type $\tipo$:
\begin{enumerate}[leftmargin=*,label=\rm{(\roman*)}]
\item\label{item:kthulhu-Sp-are-tipo1} 
The class represented by the image through $\pi$ of a split involution, for $n=2$ and $q=3,5,7$;

\item\label{item:kthulhu-Sp-are-tipo2} 
The unipotent conjugacy classes labeled by $V(2)\oplus W(1)^{n-1}$, for $n\geq 2$, $p=2$, and $m>1$; 

\item\label{item:kthulhu-Sp-are-tipo3} 
The unipotent conjugacy classes labeled by $W(2)$, for $n=2$, $p=2$, and $m> 1$.
\end{enumerate}
\end{lema}
\begin{proof}
\ref{item:kthulhu-Sp-are-tipo1} follows from Lemmata \ref{lem:bn-discarded-omega}, \ref{lem:bn-discarded-omega3} because $\Pom_{5}^+(q)\simeq \PSp_{4}(q)$.  

\ref{item:kthulhu-Sp-are-tipo2} follows from Example \ref{exa:SL2-unipotent}, because the 
group embedding 
\begin{align*}
\SL_2(2^m) &\to \Sp_{2n}(2^m)=\PSp_4(2^m),& \left(\begin{smallmatrix}a&b\\
c&d\end{smallmatrix}\right)&\mapsto \left(\begin{smallmatrix}a&0&b\\
0&\id_{2n-2}&0\\
c&0&d\end{smallmatrix}\right),
\end{align*}
induces a rack inclusion from the non-trivial unipotent class in $\SL_2(2^m)$ to the class labeled by $V(2)\oplus W(1)^{n-1}$ in $\PSp_4(2^m)$.

\ref{item:kthulhu-Sp-are-tipo3} The non-trivial graph automorphism of $\Sp_4(2^m)$ interchanging short and long root subgroups \cite[12.1]{carter-simple} interchanges the classes 
labeled by $V(2)\oplus W(1)$ and  $W(2)$. The statement follows from \ref{item:kthulhu-Sp-are-tipo2}.  
\end{proof}

\medbreak

\subsection{Unipotent classes in unitary groups}

\

To underline the range of applicability of the criterion of type $\tipo$, we address another infinite family of khtulhu unipotent conjugacy classes, this time in the Steinberg group $\PSU_{n+1}(2^{m})$.

\begin{prop}\label{prop:kthulhu-SU-are-tipo}
The unipotent conjugacy classes in $\PSU_{n+1}(2^{m})$, where $n\geq 3$ and $m\geq 1$,
corresponding to the partition $(2,1,\ldots)$ are of type $\tipo$. 
\end{prop}
\begin{proof}
First of all, unipotent classes in $\PSU_{n+1}(q)$ corresponding to the same partition are isomorphic as racks, see \cite[Section 5.1]{ACG-IV}, so it is enough to prove the statement for a class in each group.

\medbreak

If $m>1$ then we use Example \ref{exa:SL2-unipotent} and the rack inclusion stemming from the natural group homomorphism 
\begin{align*}\SL_2(2^m)\simeq\PSL_2(2^m)&\longrightarrow \SU_{n+1}(2^m)\longrightarrow \PSU_{n+1}(2^m),\\ 
\left(\begin{smallmatrix}a&b\\c&d\end{smallmatrix}\right)&\mapsto \left(\begin{smallmatrix}a&&b\\
&\id_{n-1}&\\
c&&d\end{smallmatrix}\right).\end{align*}

\medbreak
For the rest of the proof $m=1$. The group $\PSU_3(2)$ is solvable, has a unique class of involutions and contains a subgroup $\mathbf S$ isomorphic to 
the non-trivial semidirect product $(\F_3,+)\rtimes \s_3$, 
where $\s_3$ acts on $(\F_3,+)$ by the sign character. 
The conjugacy class of $(0,(12))$ contains $(0,\sigma)$ for any transposition $\sigma$, and contains the element
\[(-1,\id)\trid(0,(12))=(-1,(12))(1,\id)=(1,(12)).\]
Hence, $\mathbf S=\langle\Oc_{(0,(12))}^{\mathbf{S}}\rangle$. A direct calculation shows that $\Cent_{\mathbf{S}}(0,(12))=\langle (0,(12))\rangle$, so  $\left|\Oc_{(0,(12))}^{\mathbf{S}}\right|=9$. Therefore the class of involutions in $\PSU_3(2)$ is of type $\tipo$.  

\medbreak

Through the isomorphism $\PSU_4(2)\simeq\PSp_4(3)$ the kthulhu unipotent class labeled by the partition $(2,1,1)$ corresponds to the class of the split involution in $\PSp_4(3)\simeq\Pom^+_5(3)$, as both their centralisers have order $2^6\cdot 3^2$. Hence, it is of type $\tipo$ by Lemma \ref{lem:bn-discarded-omega3}. For  $n\geq4$ we use the rack inclusions stemming from the group inclusions  \begin{align*}\PSU_4(2)\simeq\SU_4(2)&\longrightarrow \SU_{n+1}(2)\longrightarrow\PSU_{n+1}(2)\\ 
\left(\begin{smallmatrix}A&B\\C&D\end{smallmatrix}\right)&\mapsto \left(\begin{smallmatrix}A&&B\\
&\id_{n-3}&\\
C&&D\end{smallmatrix}\right).\end{align*}
\end{proof}

\subsection{Split semisimple classes in Chevalley groups}

\

We can also complete the analysis of split semisimple classes in Chevalley groups. We start by a conjugacy class in $\aco$ which is sober \cite[Remark 5.3]{ACG-VII}.

\begin{lema} \label{lema:12^2}
The conjugacy class of cycle type $(1, 2^2)$ in $\aco$ is of type $\tipo$.
\end{lema}

\pf If $t = (25)(34)$, $x = (12345)$ and $\varGamma = \langle x \rangle$, 
then $t x t = x^{-1}$ and $\Cent_{\varGamma} (t) = \{e\}$.
Corollary \ref{cor:normal-abelian-order} implies 
that $\Oc_t^{\aco}$ is of type $\tipo$.
\epf

\begin{theorem}\label{thm:split-semisimple-collapse}
Let $\Gb$ be a simple Chevalley group and let $t\in \T^F\setminus e$.
Then $\Oc$ collapses. 
\end{theorem}
\begin{proof}
For $\Gb=\PSL_2(4)  \simeq \aco$, the split semisimple elements have cycle type 
$(1^2, 3)$ and their class is of type C by \cite[Example 2.7]{ACG-VII}.
For $\Gb= \PSL_2(5) \simeq \aco$, the split semisimple elements have 
cycle type $(1, 2^2)$ and their class is of type $\tipo$ by Lemma \ref{lema:12^2}.
For $\Gb= \PSL_2(9) \simeq \as$, the split semisimple elements have cycle type $(1^2, 2^2)$ and their class  is of type C by \cite[Example 2.6]{ACG-VII}.

\medbreak
For $\Gb=\PSL_2(q)$ with $q\neq 4,5,9$ and for $\Gb=\PSL_{n+1}(q)$ for $n\geq 2$ this is contained in \cite[Theorem 1.1]{ACG-III}. If $\Gb\neq \PSL_{n+1}(q)$ and $t$ is not a split involution in $\Pom_{2n+1}^+(q)$ with $n\geq 2$ and $q=3,5,7$ the result was established in \cite[Theorem 4.1]{ACG-VII}. The statement then follows from Lemmata \ref{lem:bn-discarded-omega} and  \ref{lem:bn-discarded-omega3} and the inclusion $ \Pom^+_5(q)\leqslant  \Pom^+_{2n+1}(q)$.
\end{proof}

The criterion of type $\tipo$ does not allow to deal with all kthulhu conjugacy classes, as the following Remark shows.

\begin{obs}\label{obs:SLp-irreducible}Let $n+1$ be an odd prime. The class of any irreducible semisimple element $x$ in $\PSL_{n+1}(q)$ or in  $\SL_{n+1}(q)$ is kthulhu,  \cite[Lemma 5.13]{ACG-VII}. The proof  relies on an analysis of the possible subgroups containing $x$. The same analysis shows that if  a solvable group $H$ contains $x$, then $H$ is contained in the normaliser of the unique maximal torus $T$ containing $x$. Then $\Oc_x^H\subseteq T$ consists of commuting elements. But then $\langle \Oc_x^H\rangle\subseteq T=H$ so $H$ is abelian and $\Oc_x^H=\{x\}$, a trivial subrack. Hence, irreducible semisimple classes in $\PSL_{n+1}(q)$ or  $\SL_{n+1}(q)$  are not of type $\tipo$. However, abelian techniques ensure that these classes fall down, see \cite[Theorem 1.2]{ACG-III}.
\end{obs}

\section{Abelian techniques}\label{sec:abelian-techniques}
Let $G$ be a finite group, $V \in \yd{\ku G}$ irreducible and $\Oc = \supp V$.
Given a subrack $X$ of $\Oc$, the subspace $V_X \coloneqq \oplus_{x \in X} V_x$
is a braided subspace of $V$. If $X$ is trivial, then $V_X$ is of diagonal type
and the knowledge of this class allows to decide 
when $\dim \toba(V_X)$ and a fortiori $\dim \toba(V)$
are infinite; in other words, when $V$ falls down (but not whether 
$\Oc$ collapses).  
These arguments are collectively called \emph{abelian techniques}.

\medbreak
For example, a conjugacy class $\Oc = \Oc_{g}^G$ is \emph{real} if $g$ is
conjugated, but not equal, to $g^{-1}$. In this case it is enough to consider 
$X = \{g, g^{-1}\}$.

\begin{lema}\label{lema:real} \cite[Lemma 2.2]{AZ}
If $\Oc = \Oc_{g}^G$ is real and $g$ has odd order, then $\Oc$ falls down.
\qed
\end{lema}

In this Section we use abelian techniques to show that the open cases in the simple groups $\Gb=E_8(q)$ and $\PSp_4(q)$ fall down. 
As well, we discuss the collapse of the groups $\PSL_2(q)$.

\subsection{Nichols algebras over \texorpdfstring{$E_8(q)$}{} 
and \texorpdfstring{$\PSp_{2n}(q)$}{}}
\label{subsec:PSp-E8}

\

\begin{theorem}\label{thm:e8-completo}The group $\Gb=E_8(q)$ collapses.
\end{theorem}

\begin{proof}
Unipotent classes and mixed classes collapse by \cite[Main Theorem]{ACG-IV} 
and \cite[Theorem 1.1]{ACG-V}, respectively. 
Split semisimple classes collapse by \cite[Theorem 4.1]{ACG-VII}. 
Let $s\in \T_w^F\setminus e$, for  $w\in W$. 
If the class of $w$ is not labeled by $E_8$ or $E_8(a_5)$, 
then $\Oc_s^{\Gb}$ collapses by Lemmata \ref{lem:e8-non-cuspidal} 
and \ref{lem:e8-cuspidal}. 

\medbreak
If the class of $w$ is labeled by $E_8$ or $E_8(a_5)$, then $|T_w^F|\in\{\varphi_{15}(q),\varphi_{30}(q)\}$ so $|s|$ divides
$q^8\pm q^7\pm q^5-q^4\pm q^3\pm q+1$, and $s$ is real because $w_0=-\id$. Then $\Oc_s^{\Gb}$ falls down by Lemma \ref{lema:real}.  
\end{proof}

Recall that  
$\PSp_{2n}(q)$ is not simple if $n=q=2$.

\begin{theorem}\label{thm:cn-completo} The group $\Gb=\PSp_{2n}(q)$, $n\geq 2$ and $(n,q)\neq(2,2)$, collapses.
\end{theorem}
\begin{proof}
For $n\geq 3$ this is \cite[Theorem III]{ACG-VII}. Let $n=2$. 
Mixed classes collapse by \cite[Theorem 1.1]{ACG-V}. 
Semisimple classes collapse by \cite[Theorem I (ii)]{ACG-VII} and Lemma \ref{lem:kthulhu-Sp-are-tipo} (1). By \cite[Theorem II (ii)]{ACG-VII}, the kthulhu unipotent classes are: 

\begin{itemize} [leftmargin=*]\renewcommand{\labelitemi}{$\circ$} 
\item those labeled by $W(2)$ and $V(2)\oplus W(1)$ for $q$ even, and

\medbreak
\item those associated with a partition of type $(2,1^2)$, for $q$  odd not a square.

\end{itemize}

\medbreak
  Those  labeled by $W(2)$ and $V(2)\oplus W(1)$ collapse by Lemma \ref{lem:kthulhu-Sp-are-tipo} \ref{item:kthulhu-Sp-are-tipo2} and \ref{item:kthulhu-Sp-are-tipo2} because $q>2$.  
  
\medbreak
Let $u=\left(\begin{smallmatrix}1&0&0&1\\
0&1&0&0\\
0&0&1&0\\
0&0&0&1
\end{smallmatrix}\right)$. Without loss of generality $\pi(u)$ is a representative of a class labeled by $(2,1^2)$ in $\Gb$. 

\medbreak
If $p>3$, then there exists $a\in\F^\times_p$ such that $a^2\in \F_p\setminus\{1\}$. Hence 
\begin{align*}\pi(u)\neq\pi(u)^{a^2}=\pi\left({\diag}(a,1,1,a^{-1})\trid u\right)\in\Oc_{\pi(u)}^{\Gb}\end{align*} and $|\pi(u)|=p$ is odd and coprime to $3$. The class $\Oc_{\pi(u)}^{\Gb}$ falls down by \cite[Lemma 1.9]{AF}.

\medbreak
Assume that $p=3$. Then the class contains a subrack isomorphic to a non-trivial unipotent class in $\PSL_4(3)$, and we invoke \cite[Lemma 6.1]{ACG-V}.
\end{proof}

\subsection{Nichols algebras over \texorpdfstring{$\PSL_{2}(q)$}{}}
\label{subsec:PSL(2)}

\

\medbreak
The collapse of the groups $\PSL_{2}(q)$, where either $q\equiv 1\mod 4$ or $q >2$ is even, was proved  in \cite[Proposition 3.1]{fgvI}, 
\cite[Theorem 1.2]{ACG-III}.
We round up these results reporting which classes are known to collapse.
There are $3$ types of non-trivial  conjugacy classes in $\PSL_{2}(q)$: 
unipotent, split semisimple (i.e., with a diagonal representative)
and irreducible (non-split) semisimple.

\begin{lema}\label{lema:psl(2)}
Let $q \notin \{2,3,4,5\}$ and let $\Oc$ be a conjugacy class in
$\PSL_{2}(q)$.
\begin{enumerate}[leftmargin=*,label=\rm{(\roman*)}]
\item \label{item:psl(2)-unipotent-even} 
Assume that $\Oc$ is unipotent and $q > 2$ is even.
Then $\Oc$ is kthulhu but it  collapses, being of type $\tipo$.

\medbreak
\item \label{item:psl(2)-unipotent-odd-square} 
Assume that $\Oc$ is unipotent and $q$ is an odd square. Then
$\Oc$ collapses.

\medbreak
\item \label{item:psl(2)-unipotent-odd} 
Assume that $\Oc$ is unipotent and $q$ is odd but not a square. Then $\Oc$ is kthulhu.
In case that $q\equiv 1\mod 4$, then $\Oc$ falls down.

\medbreak
\item \label{item:psl(2)-split-semisimple} 
If the class $\Oc$ is split semisimple, then  it  collapses, being of type C.

\medbreak\item \label{item:psl(2)-non-split-semisimple} 
If either $q$ is even but not a square and the elements
in the class have order 3; or $q$ is arbitrary and the elements in the class have order $>3$; then $\Oc$ is kthulhu. Otherwise,
$\Oc$ collapses. In the cases when $q>2$ is even or $q\equiv 1\mod 4$, then $\Oc$ falls down.
\end{enumerate}
\end{lema}

\pf 
\ref{item:psl(2)-unipotent-even}:
$\Oc$ is  kthulhu by \cite[Lemma 3.5]{ACG-I} and of type $\tipo$
by Example \ref{exa:SL2-unipotent}.

\medbreak
\ref{item:psl(2)-unipotent-odd-square}: If $q > 9$ is an odd square, then
$\Oc$ is of type D by \cite[Lemma 3.6]{ACG-I}, while
if $q=9$ then $\Oc$ is of type C, see \cite[Example 2.7]{ACG-VII}. Hence 
$\Oc$ collapses.

\medbreak
\ref{item:psl(2)-unipotent-odd}:
Again, $\Oc$ is  kthulhu by \cite[Lemma 3.5]{ACG-I}.
If $q\equiv 1\mod 4$, then $\Oc$ is real (since $-1$ is a square mod $q$)
and the elements of $\Oc$ have odd order, hence $\Oc$ falls down by Lemma \ref{lema:real}.

\medbreak
\ref{item:psl(2)-split-semisimple}: This is \cite[Lemma~3.9]{ACG-III}.

\medbreak
\ref{item:psl(2)-non-split-semisimple}:
The first claim follows from \cite[Theorem 1.1]{ACG-III} and the
second from \cite[Proposition 4.2]{ACG-III}.
Assume that either $q > 2$ is even or else $q\equiv 1\mod 4$.
Then the semisimple irreducible classes are real and
the orders of their elements  are odd (the non-split torus has order $(q+1)/2$ which is
odd when $q\equiv 1\mod 4$.). 
Thus the classes fall down by Lemma \ref{lema:real}.
\epf

As a straightforward consequence, we get:

\begin{theorem} \label{thm:psl2-completo}
The groups  $\PSL_2(q)$, where either $q\equiv 1\mod 4$ or $q >2$ is even, collapse.
\qed
\end{theorem}

\begin{remark}
If $q=2$ or $3$, then $\PSL_2(q)$ is not simple.

\medbreak
If $q=4$ or $5$, then $\PSL_2(q) \simeq \aco$, which collapses
by \cite[Theorem 2.6]{AF}. 

\medbreak
The question of the collapse of the conjugacy classes in $\aco$ 
was addressed in \cite{ACG-VII}.
There are two sober classes in $\aco$: 
\begin{itemize} [leftmargin=*]\renewcommand{\labelitemi}{$\circ$} 
\item cycle type $(5)$, which is unipotent in $\PSL_2(5)$ 
but semisimple in $\PSL_2(4)$;

\medbreak
\item cycle type $(1, 2^2)$, which is semisimple in $\PSL_2(5)$ but  unipotent 
when considered in $\PSL_2(4)$.
\end{itemize}

\medbreak
The latter collapses because it is of type $\tipo$, see Lemma \ref{lema:12^2}. 
\end{remark}

\begin{remark} A unipotent class $\Oc$ is not known to collapse when   
$q$ is odd but not a square, in particular the case $q\equiv 3\mod 4$ is open.

\medbreak
Let $3 < q$ be odd. Then, the unipotent conjugacy classes in $\PSL_2(q)$ are not of type $\tipo$. Indeed, the solvable subgroups of $\PSL_2(q)$ are either: 

\begin{itemize} [leftmargin=*] 
\item dihedral of order coprime to $p$, 

\medbreak
\item abelian, 

\medbreak
\item 
a Borel subgroup, 

\medbreak
\item 
isomorphic to $\mathbb S_4$, 

\medbreak
\item 
or one of their subgroups.
\end{itemize}

A unipotent class in a Borel subgroup generates an abelian subgroup. 
A non-trivial unipotent conjugacy class 
meets a solvable non-abelian subgroup different from the Borel subgroup only if $p=3$. 
But $3$-cycles do not generate $\mathbb S_4$, 
and the class of $3$-cycles in $\mathbb  A_4$ is isomorphic 
to the affine rack $(\F_4,\omega)$ which belongs to the list \eqref{eq:list-affine-simple-racks}.
\end{remark}

\section*{Acknowledgements} 
N.~A. was  partially supported by the Secyt-UNC (Proyecto Consolidar 33620230100517CB),
CONICET (PIP 11220200102916CO) and FONCyT-ANPCyT (PICT-2019-03660).

\medbreak
Part of the work of N. A. was done during a long term  visit  to the 
Shenzhen International Center for Mathematics at the Southern University of Science and Technology; he thanks Efim Zelmanov and Slava Futorny for the warm hospitality.

\medbreak
Part of the work  of N. A. was carried out while he held the position of Visiting Professor at the Vrije Universiteit Brussel (2025-2026). He thanks Leandro Vendramin 
for his friendly hospitality and both Leandro and Istv\'an Heckenberger for interesting 
discussions in Brussels.

\medbreak

The authors thank Mauro Costantini, Fernando Fantino, Gast\'on A. Garc\'ia, Mat\'ias Gra\~na and Leandro Vendramin 
for their collaboration in the papers 
\cite{ACG-I,ACG-II,ACG-III,ACG-IV,ACG-V,ACG-VII,AF,AFGV-ampa,AFGV-II,AG,CC-VI}
 where techniques were developed and results were obtained that support our work.

\medbreak

 Project partially funded by the EuropeanUnion – NextGenerationEU under the National
Recovery and Resilience Plan (NRRP), Mission 4 Component 2 Investment 1.1 -
Call PRIN 2022 No. 104 of February 2, 2022 of Italian Ministry of
University and Research; Project 2022S8SSW2 (subject area: PE - Physical
Sciences and Engineering) ``Algebraic and geometric aspects of Lie theory". G.C. is a member of the INdAM group GNSAGA.

\bibliographystyle{abbrv}
\bibliography{refs-ssfglt.bib}

\end{document}